\RequirePackage{fix-cm}
\documentclass[smallextended]{svjour3}       
\smartqed  
\usepackage{xcolor}
\usepackage{amsmath}
\usepackage{amssymb}
\usepackage{mathrsfs}
\usepackage{graphicx}

 \usepackage{bm}
\usepackage{color}
\usepackage{float}
\usepackage{epstopdf}
\usepackage{subcaption}
\usepackage[breaklinks,colorlinks=true,linkcolor=blue,citecolor=red,backref=page]{hyperref}

\DeclareMathAlphabet{\itbf}{OML}{cmm}{b}{it}
 \DeclareMathAlphabet\mathbfcal{OMS}{cmsy}{b}{n}

\renewcommand{\hat}{\widehat}
\renewcommand{\tilde}{\widetilde}
\def\RR{\mathbb{R}}

\def\bx{{{\itbf x}}}
\def\bxi{\boldsymbol{\xi}}
\def\vx{{\vec{\itbf x}}}

\def\by{{{\itbf y}}}

\def\bu{{{\itbf u}}}

\def\bg{{\itbf g}}
\def\be{{\itbf e}}

\def\bU{{\itbf U}}
\def\bE{{\itbf{E}}}

\def\bB{{\itbf B}}
\def\bL{{\itbf L}}

\def\bJ{{\itbf J}}

\def\bGa{\tensor{\boldsymbol{\Gamma}}}
\def\bth{\tensor{\boldsymbol{\gamma}}}

\def\cP{{\cal P}}

\def\bv{{\itbf v}}

\def\bn{{\itbf n}}

\def\bphi{{\boldsymbol{\varphi}}}
\def\bpsi{{\bm{\psi}}}

\def\bbeta{{\boldsymbol{\beta}}}
\def\bet{{\boldsymbol \eta}}

\def\bY{{\itbf Y}}
\def\beps{\underline{\underline{\boldsymbol{\varepsilon}}}}
\def\eps{\varepsilon}

\def\bc{\underline{\underline{\itbf{c}}}}

\DeclareFontFamily{OMS}{oasy}{\skewchar\font48 }
\DeclareFontShape{OMS}{oasy}{m}{n}{%
         <-5.5> oasy5     <5.5-6.5> oasy6
      <6.5-7.5> oasy7     <7.5-8.5> oasy8
      <8.5-9.5> oasy9     <9.5->  oasy10
      }{}
\DeclareFontShape{OMS}{oasy}{b}{n}{%
       <-6> oabsy5
      <6-8> oabsy7
      <8->  oabsy10
      }{}
\DeclareSymbolFont{oasy}{OMS}{oasy}{m}{n}
\SetSymbolFont{oasy}{bold}{OMS}{oasy}{b}{n}

\DeclareMathSymbol{\smallleftarrow}     {\mathrel}{oasy}{"20}
\DeclareMathSymbol{\smallrightarrow}    {\mathrel}{oasy}{"21}
\DeclareMathSymbol{\smallleftrightarrow}{\mathrel}{oasy}{"24}
\newcommand{\tensor}[1]{\underline{\underline{#1}}}
\def\bcG{\tensor{\boldsymbol{\mathcal{G}}}}
\def\bbG{\tensor{\boldsymbol{\mathbb{G}}}}
\def\bG{\tensor{\itbf G}}
\def\bH{{\itbf H}}
\def\bV{{\itbf V}}
\def\bK{{\itbf K}}
\def\bR{{\itbf R}}
\def\bI{{\itbf I}}

\def\bde{\tensor{\boldsymbol{\delta}}}
\def\bM{{\itbf M}}

\def\cP{{\mathcal P}}

\def\cI{{\mathcal{I}}}

\def\cD{\boldsymbol{\mathcal D}}
\def\bbD{{\boldsymbol{\mathbb{D}}}}

\def\cA{\mathbb{A}}

\def\12{{\frac{1}{2}}}
\def\RM{{\scalebox{0.5}[0.4]{ROM}}}
\def\RTM{{\scalebox{0.5}[0.4]{RTM}}}

\def\bcU{{\boldsymbol{\mathcal{U}}}}
\def\bcV{{\boldsymbol{\itbf{V}}}}
\def\bi{\itbf{i}}
\def\buR{\bu^{\RM}}
\def\cPR{\cP^{\RM}}

\def\bbU{\bm{\mathbb{U}}}

\def\bbM{\bm{\mathbb{M}}}
\def\bbF{\bm{\mathbb{F}}}

\def\bphi{{\boldsymbol{\varphi}}}

\def\bpsi{{\boldsymbol{\psi}}}

\def\bsig{{\boldsymbol{\sigma}}}

\def\om{\omega}
\def\la{\lambda}

\def\FWI{{\scalebox{0.5}[0.4]{FWI}}}

\spnewtheorem{rem}{Remark}{\bf}{\it}
\spnewtheorem{algorithm}{Algorithm}{\bf}{\it}

\graphicspath{{./FIG/},{./FIG/NORMG/},{./FIG/INVERSION/},{./FIG/COMBINED/},
{./FIG/INTERNAL_WAVE/},{./FIG/PicturesPublish}}
\begin{document}

\title{Electromagnetic inverse wave scattering  in anisotropic media via  reduced order modeling}


\author{Liliana Borcea       \and
        Yiyang Liu \and
        J\"{o}rn Zimmerling
}


\institute{Liliana Borcea\at
              Department of Mathematics,\\
              University of Michigan,\\
  		Ann Arbor,\\
  		 MI 48109,
  		 USA
  		 \email{borcea@umich.edu}
           \and
          Yiyang Liu \at
              Department of Mathematics,\\
              University of Michigan,\\
  		Ann Arbor,\\
  		 MI 48109,
  		 USA
  		 \email{yiyangl@umich.edu}
              \and
               J\"{o}rn Zimmerling\at
               Uppsala Universitet, Department of Information Technology,\\
               Division of Scientific Computing,\\ 
               75105 Uppsala,\\ 
               Sweden\\
               \email{ jorn.zimmerling@it.uu.se}
}

\date{\today}

\maketitle

\begin{abstract}
The inverse wave scattering problem seeks to estimate a heterogeneous, inaccessible medium, modeled by unknown variable coefficients in wave equations, from transient recordings of waves generated by probing signals. It is a widely studied inverse problem with important applications, that is typically formulated as a nonlinear least squares data fit optimization. For typical measurement setups and band-limited probing signals,  
 the least squares objective function has spurious local minima far and near the true solution, so Newton-type optimization methods 
fail to obtain satisfactory answers. We introduce a different approach, for electromagnetic inverse wave scattering in lossless, anisotropic media. It is an extension of recently developed
data driven reduced order modeling methods for the acoustic wave equation in isotropic media. 
Our reduced order model (ROM) is an algebraic, discrete time dynamical system derived from Maxwell's equations. It has four important properties: (1) It can be computed in a data driven way, without knowledge of the medium. (2) The data to ROM mapping is nonlinear and yet the ROM can 
be obtained in a non-iterative fashion, using numerical linear algebra methods. (3) The ROM has a special algebraic structure that captures the causal propagation of the wave field in the unknown medium. (4) It is an interpolation ROM i.e.,   
it fits the data on a uniform time grid. 
We show how to obtain from the ROM an estimate of the wave field at inaccessible points inside the unknown medium. The use of this wave is twofold: First, it defines a computationally inexpensive imaging function designed to estimate the support of reflective structures in the medium, modeled by jump discontinuities of the matrix valued dielectric permittivity. Second, it gives an objective function for quantitative estimation of the dielectric permittivity, that has better behavior than the least squares data fitting objective function. The methodology introduced in this paper applies to Maxwell's equations in three dimensions. To avoid high computational costs, we limit the study to a cylindrical domain filled with an orthotropic medium, so the problem becomes two dimensional.

\end{abstract}


%
%

\section{Introduction}
\label{sect:intro}
The estimation of an inaccessible heterogeneous medium from transient recordings of waves generated by 
probing signals is important in radar imaging, remote sensing, medical diagnostics, 
nondestructive evaluation of materials and aging structures, exploration geophysics, underwater acoustics, 
and so on. Depending on the application, it is formulated as an inverse problem for the scalar (acoustic) wave equation, 
or for systems of equations that govern the propagation of vectorial (elastic or electromagnetic) waves. The medium 
 is modeled by unknown variable coefficients of these equations, which are scalar valued in the isotropic case or matrix valued in the anisotropic case.  The inverse wave scattering problem is to estimate these coefficients from the measurements. 

We study inverse scattering with electromagnetic waves,  governed by Maxwell's equations in lossless media.
It is known that the dielectric permittivity and magnetic permeability of an isotropic 
medium occupying a bounded and simply connected domain are uniquely determined by the admittance map, 
which takes the tangential boundary trace of the electric field to the tangential boundary trace of the magnetic field.  
\cite{ola1993inverse,kenig2011inverse}. The admittance map measured for time harmonic fields or transient (multi-frequency) fields is not sufficient 
for inverse scattering in anisotropic media.   It determines the matrix valued coefficients only up to a diffeomorphic transformation that leaves the boundary unchanged. This behavior is typical of other inverse 
problems in anisotropic media~\cite{lee1989determining} and has sparked interest in inversion 
algorithms that use prior information, such as the direction of anisotropy~\cite{heino2002estimation,abascal2011electrical}.
Interestingly, the shape of inclusions filled with anisotropic media is uniquely recoverable, as shown in ~\cite{piana1998uniqueness,cakoni2010inverse} for measurements of the far field pattern of the magnetic field,  with transverse electric polarization. In many applications, it is possible to get more information about the medium from data acquisitions that exploit different polarizations of the probing waves. The study~\cite{elbau2018inverse} shows that, at least in perturbative regimes,  the matrix valued dielectric permittivity of orthotropic media can be determined from measurements of the time harmonic electric  field generated by two different incident polarizations. 

We study inverse   wave scattering with measurements  gathered by an array 
of antennas that emit probing waves with two different polarizations and then measure the resulting electric wave field. 
Our approach applies in principle to Maxwell's equations in an arbitrary three dimensional domain. To minimize the computational cost, 
we work in a cylindrical domain $\Omega \times \RR$, with bounded and simply connected cross-section 
$\Omega \subset \RR^2$ that has Lipschitz continuous boundary $\partial \Omega$ \cite[Definition 3.1]{monk2003finite}. 

Introduce the coordinate system $\vx = (\bx,z)$, with $\bx \in \Omega$ and $z \in \RR$ and assume that the medium does not vary in $z$. In most applications the medium is non-magnetic \cite[Section 1.2]{cloude2009polarisation}, so we set the magnetic permeability equal to the scalar constant $\mu_o$.
Electromagnetic wave propagation 
in two dimensional isotropic media reduces to the acoustic wave equation. Thus,  we consider  an anisotropic, orthotropic medium~\cite{piana1998uniqueness,cakoni2010inverse}, also known as a medium with monoclinic symmetry \cite[Section 13.2.2]{Torquato}, modeled by the  piecewise smooth dielectric permittivity matrix
\begin{equation}
\boldsymbol{\eps}(\vx) = \begin{pmatrix} 
\beps(\bx) & {\bf 0} \\
{\bf 0}^T & \varepsilon_{z}(\bx) 
\end{pmatrix},  \quad  \beps(\bx) = \begin{pmatrix} \varepsilon_{11}(\bx) & \varepsilon_{12}(\bx) \\
 \varepsilon_{12}(\bx) &  \varepsilon_{22}(\bx) 
 \end{pmatrix}. \label{eq:1}
 \end{equation}
 Here $\beps\in \RR^{2 \times 2}$  is symmetric and positive definite, as it should be in passive and lossless media  \cite[Section 1.2.2]{hanson2013operator}, and $\varepsilon_z$ is positive.  We use consistently two underlines for $2\times 2$ matrix valued fields. 
  
We model the boundary $\partial \Omega \times \RR$ of the cylindrical domain, with outer normal 
\begin{equation}
\vec \bn(\vec{\bx}) = \begin{pmatrix} \bn (\bx) \\ 0 \end{pmatrix}, \quad \vx = (\bx,z) \in \partial \Omega \times  \RR,
\label{eq:defn}
\end{equation}
as perfectly conducting. This boundary  may be physical, in which  case we have a waveguide with open ends, 
or fictitious and justified by hyperbolicity of the problem: The waves propagate at finite speed, so they will not sense over the duration of the measurements a boundary that is sufficiently far from the source of waves. 

The wave is defined by the electric field $\vec \bE$ and the magnetic induction $\vec \bB$. These 
satisfy Faraday's law and Amp\`{e}re's  circuital law in Maxwell's system of equations
\begin{align}
\vec{\nabla} \times \vec{\bE}(t,\vec{\bx}) + \partial_t \vec{\bB}(t,\vx) &= 0, \label{eq:3} \\
-\vec{\nabla} \times \vec{\bB}(t,\vec{\bx}) + \mu_o \begin{pmatrix} 
\beps(\bx) & {\bf 0} \\
{\bf 0}^T & \varepsilon_{z}(\bx) 
\end{pmatrix} \partial_t \vec{\bE}(t,\vec{\bx}) &= - \mu_o \vec{\bJ}(t,\bx) ,  \label{eq:4}  
\end{align}
for $t \in \RR$ and $\vx = (\bx,z) \in \Omega \times \RR$, with  boundary condition 
\begin{align}
\vec{\bn}(\bx) \times \vec{\bE}(t,\vec{\bx}) &= \vec {\bf 0}, \quad  t \in \RR, ~~ \vx = (\bx,z) \in \partial \Omega \times \RR, \label{eq:5}
\end{align}
and with 
initial conditions
\begin{equation}
 \vec{\bE}(t,\vec{\bx}) \equiv \vec {\bf 0}, \quad  \vec{\bB}(t,\vec{\bx}) \equiv \vec {\bf 0}, \quad t \ll 0, ~~ \vx = (\bx,z) \in  \Omega \times \RR.
\label{eq:6}
\end{equation}
The wave excitation in ~\eqref{eq:4} is a $z-$independent transient source current density $\vec{\bJ} = (\bJ,J_z)$ supported away from $\partial \Omega$  and $t \ll 0$ means prior to the excitation. 


The assumption~\eqref{eq:1} on the dielectric permittivity, the choice of the domain and the $z-$independent source excitation, give that the initial boundary value problem~\eqref{eq:3}-\eqref{eq:6}
admits $z$ independent solutions
\begin{equation*}
\vec{\bE}(t,\bx) = \begin{pmatrix}\bE(t,\bx) \\ E_z(t,\bx) \end{pmatrix}, \quad 
\vec{\bB}(t,\bx) = \begin{pmatrix}\bB(t,\bx) \\ B_z(t,\bx) \end{pmatrix},
\end{equation*}
where $\bE$ and $\bB$ are vectors  in the plane orthogonal to the $z$ axis. To determine the dielectric pemittivity~\eqref{eq:1}, it suffices to work 
with the electric field, whose components  satisfy the decoupled wave equations 
\begin{equation}
\nabla^\perp \big[ \nabla^\perp \cdot \bE(t,\bx) \big] - \bc^{-2}(\bx) \partial_t^2 \bE(t,\bx) = \mu_o \partial_t \bJ(t,\bx),
\label{eq:We1} 
\end{equation}
and 
\begin{equation}
\Delta E_z(t,\bx) - c^{-2}_z(\bx) \partial_t^2 E_z(t,\bx) =\mu_o \partial_t J_z(t,\bx),
\label{eq:We4} 
\end{equation}
for $t \in \RR$ and $\bx \in \Omega$. Here $\bc$ and $c_z$ are the wave speeds defined by 
\begin{equation}
\bc(\bx) = \left[ \mu_o \beps(\bx)\right]^{-1/2}, \quad c_z(\bx) = 1/\sqrt{\mu_o \eps_z(\bx)},
\label{eq:defc}
\end{equation}
where  we take the unique symmetric and positive definite matrix square root \cite[Theorem 7.2.6]{horn2012matrix}. 
The operator 
$\nabla^\perp = (-\partial_{x_2},\partial_{x_1})$ is the two dimensional gradient 
$\nabla = (\partial_{x_1},\partial_{x_2})$ rotated by $90^o$ and $\Delta = \partial_{x_1}^2 + \partial_{x_2}^2$ is the Laplacian. 
At $t\ll 0$ we have the 
homogeneous initial conditions~\eqref{eq:6} and we deduce 
from~\eqref{eq:defn} and~\eqref{eq:5} the boundary conditions
\begin{align}
\bn^\perp(\bx) \cdot \bE(t,\bx) &= 0, \quad E_z(t,\bx) = 0, 
\label{eq:We6}
\end{align}
where $\bn^\perp = (-n_2,n_1)$ is the  vector $\bn = (n_1,n_2)$ rotated by $90^o$. 

Note that~\eqref{eq:We4} is the scalar wave equation in a medium with wave speed $c_z$. A data driven reduced order model 
methodology for estimating such a speed has been introduced recently in~\cite{borcea2022waveform,borcea2022internal,borcea2023data}. Thus, we study  henceforth the estimation of $\bc$  using the vectorial wave equation~\eqref{eq:We1}.

The data for the inversion are gathered by an array of $m$ point-like antennas, which emit probing pulses $f(t)$ supported in the 
short time interval $(-T_f,T_f)$ and with Fourier transform $\hat f(\om)$ that is non-negligible at frequencies $\om$ 
satisfying $|\om \mp \om_o | < \pi b$. Here $b$ is the bandwidth and $\om_o$ is the central angular frequency.  In all imaging applications the pulse satisfies \begin{equation}
\hat f(0) = \int_{-T_f}^{T_f} dt \, f(t) = 0,
\label{eq:pulsezero}
\end{equation}
 and typically, $\om_o \gg b$. 
 
 The antennas are  modeled by functions $F^{(s)} \in L^2(\Omega)$ supported around $\bx_s \in \Omega$,  in a set of small diameter with respect to the central wavelength $\la_o$, for $s = 1, \ldots, m$. 
Since the equation is linear, we can normalize $F^{(s)}$ to integrate to one and formally, we can view it as the Dirac delta at $\bx_s$. The separation between the antennas is less than $\la_o$, so they behave like a collective entity, aka the array.
The excitation from the $s^{\rm th}$ antenna has two polarizations and is given by the current density
\begin{equation}
\bJ^{(s,p)}(\bx) =\mu_o^{-1} f(t) F^{(s)}(\bx) \be_p, \quad s = 1, \ldots, m, ~~ p = 1,2,
\label{eq:Ant1}
\end{equation}
where  $\be_1, \be_2$ are the canonical basis vectors in $\RR^2$.  The resulting electric field, the solution 
of equation~\eqref{eq:We1} with $\bJ$ replaced by  $\bJ^{(s,p)}$, is denoted by $\bE^{(s,p)}$. Its recordings at all the antennas 
define the $(s,p)$ column of the   $2m \times 2m$ time dependent array response 
matrix $\boldsymbol{\mathcal{W}}(t)$, with entries 
\begin{equation}
{\mathcal{W}}^{(s',p'),(s,p)}(t) = \be_{p'}^T \int_{\Omega} d \bx \, F^{(s')}(\bx)\bE^{(s,p)}(\bx,t) \approx 
\be_{p'}^T\bE^{(s,p)}(t,\bx_{s'}),
 \label{eq:Ant2}
 \end{equation}
 for $s,s' = 1, \ldots, m$ and $p,p' = 1,2$. 
 
 The inverse problem is to estimate the $2\times 2$ symmetric, positive definite  matrix $\bc$ 
 in $\Omega$, from $\boldsymbol{\mathcal{W}}(t)$ measured in some time interval $t \in (t_{\min}, t_{\max})$. The classic 
 approach to solving it uses a  nonlinear least 
 squares data fit optimization 
 \begin{equation}
 \min_{\tilde \bc \in \mathscr{C}} \mathcal{O}^{\FWI}(\tilde \bc ) + \mbox{ regularization}, \quad 
 \mathcal{O}^{\FWI}(\tilde \bc) = \int_{t_{\min}}^{t^{\max}} dt \, \|\boldsymbol{\mathcal{W}}(t)-\boldsymbol{\mathcal{W}}(t;
 \tilde \bc)\|_F^2.
 \label{eq:FWI2}
\end{equation}
Here $\| \cdot \|_F$ is the Frobenius norm, $\tilde \bc$ denotes the search wave speed in some user defined space
$\mathscr{C}$, 
and $\tilde \bc \mapsto \boldsymbol{\mathcal{W}}(t;\tilde \bc)$ is the forward map, with output defined by 
the analogue of~\eqref{eq:Ant2}, calculated with $\tilde \bc$ instead of the true $\bc$. Note that to simplify the 
notation, we use the following convention: We omit the true speed $\bc$ 
from the list of arguments of the electric field and measurements. Thus, $\bE^{(s,p)}(t,\bx)$ is the field in the 
true medium, with unknown $\bc$, whereas $\bE^{(s,p)}(t,\bx;\tilde \bc)$, which defines $\boldsymbol{\mathcal{W}}(t;\tilde \bc)$, is the computable field for the search speed $\tilde \bc$.

\begin{figure*}[t]
\begin{center}
\includegraphics[width=0.3\textwidth]{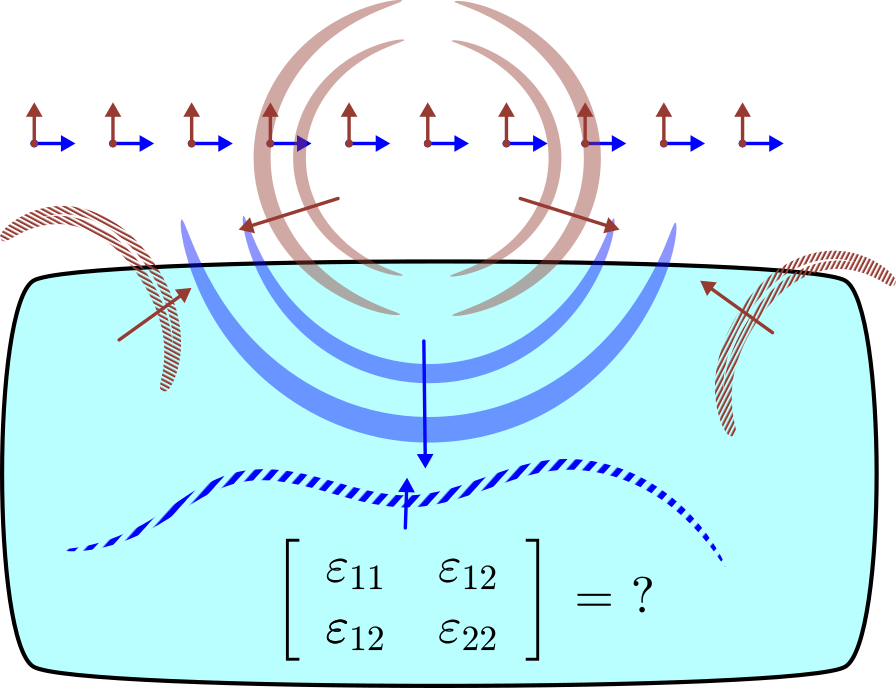}
\end{center}
\vspace{-0.in}
\caption{Illustration of  the setup and data acquisition with an active array of antennas that generate probing waves and measure the backscattered waves. }
\label{fig:1}
\end{figure*}

The nonlinear least squares data fit optimization formulation~\eqref{eq:FWI2} of inverse wave scattering is known as ``full waveform inversion" (FWI). It has drawn a lot of attention especially in the geophysics community, where it is studied mostly for acoustic and elastic waves~\cite{virieux2009overview}. The FWI formulation is attractive because it is robust to Gaussian additive noise and applies to an arbitrary data acquisition geometry, although the result is dependent of it. In particular, it is easier to invert from measurements gathered all around 
$\Omega$, so that both the transmitted and backscattered waves can be recorded. Having very low frequencies in the bandwidth of the probing pulse also
helps~\cite{bunks1995multiscale,chen1997inverse}. In most applications, low frequencies are not available and the placement of the 
sensors is limited to one side of $\Omega$, see Fig. \ref{fig:1}, so only the backscattered waves can be measured.  These restrictions and  the complicated nonlinear forward map  result 
in an FWI objective function~\eqref{eq:FWI2} that has numerous spurious local minima near and far from the true 
$\bc$~\cite{virieux2009overview,engquist2022optimal,huang2018source}. This pathological behavior is known as ``cycle skipping" and means that  gradient based optimization methods, like Gauss-Newton, may not give good results even for a good initial guess.  Cycle skipping has motivated the pursuit of better objective functions obtained, for example, by replacing the $L^2$ norm in~\eqref{eq:FWI2} with the Wasserstein metric~\cite{EngquistFroese,engquist2022optimal} or by introducing in a systematic way additional degrees of freedom in the optimization, known as ``extended FWI"~\cite{huang2018source,herrmann2013,symes2022error}. 

The linearization of the FWI objective function about the coefficients of a given reference medium, which is typically homogeneous, 
leads to a popular qualitative, aka imaging method, known as reverse time migration in the geophysics literature~\cite{biondi20063d,bleistein2001multidimensional,symes2008migration} or backprojection in radar~\cite{curlander1991synthetic,cheney2009fundamentals}. This method succeeds in locating the rough part of the medium, called the reflectivity,  if the smooth part of the medium, which determines the kinematics of wave propagation, 
is known. However, even with known kinematics, the  images display ``multiple artifacts", due to multiple scattering effects that are ignored by the linearization. The mitigation of such effects has been addressed partially in ~\cite{verschuur2013seismic,weglein1997inverse,borcea2012filtering,herrmann2007non}.

Our goal in this paper is to introduce a reduced order model approach for both imaging and  inverse 
wave scattering with electromagnetic waves in lossless, non-magnetic, anisotropic media. The reduced order model (ROM) is an algebraic dynamical system that 
captures the wave propagation at discrete time steps $t_j = j \tau$, separated by a time interval $\tau$ that is chosen according to the Nyquist sampling criterium for the frequency content of $f(t)$. The ROM is obtained via Galerkin 
projection of an exact time stepping scheme for equation~\eqref{eq:We1}, on the  time grid $\{t_j\}_{j \ge 0}.$ The projection space is spanned by the snapshots of the electric field at the first $n$ time instants and for $\bx \in \Omega$. These snapshots are not known, because the medium is not known inside $\Omega$. Nevertheless, the ROM can be computed in a data driven way, from the array response matrices $\boldsymbol{\mathcal{W}}(t_j)$ at $j = 0, \ldots, 2n-1$. The number $n$ is chosen so that over the duration $2 n \tau$, the waves can travel to the desired depth in the medium, scatter along the way and then travel back to the array, where they are recorded. 

The ROM computation is an extension of that introduced in~\cite{borcea2020reduced,borcea2019robust,DtB,druskin2016direct}
for the acoustic wave equation. We describe it in section~\ref{sect:ROM_oONSTR}. Then, we explain in section~\ref{sect:INT_WAVE} how we can use the ROM to estimate the snapshots at points in $\Omega$. 
We use this estimate in section~\ref{sect:IMAGE} to formulate a qualitative i.e., imaging method, and in section~\ref{sect:INVERSE}
to formulate a quantitative i.e., inversion method. These two methods can be used in conjunction, as we explain in section 
\ref{sect:CONJ}. The performance of the methods is illustrated with numerical simulations in sections \ref{sect:IMAGE}--\ref{sect:CONJ}. We end with a summary in section~\ref{sect:SUM}.

\section{ROM for electromagnetic wave propagation}
\label{sect:ROM_oONSTR}
The derivation of the ROM  involves several steps. The first, described in section~\ref{sect:Transf},  introduces a transformation  of the wave field, that leads to a convenient form of the wave equation, with a self-adjoint operator that has $\bc$ dependent coefficients. This operator has a nontrivial null space, 
but we explain in section~\ref{sect:waveDec} that the wave components in this space are short lived and can be removed from the measurements and the computation of the ROM. This computation is based on a Galerkin projection of an exact time stepping scheme given in section~\ref{sect:TimeStep}. The data driven ROM computation is in section~\ref{sect:ROMComp}. 
We assume there ideal and  noiseless measurements, but we explain in section~\ref{sect:ROMnoise} how to regularize the 
ROM computation so that noise and other factors are mitigated. 

\subsection{Transformation of the wave field}
\label{sect:Transf}
We assume that the medium is known, isotropic  and homogeneous near $\partial \Omega$ and in the vicinity of the array, with dielectric permittivity $\eps_o \underline{\underline{\bI}}$, which defines the reference, scalar wave speed 
$
c_o= {1}/{\sqrt{\mu_o \eps_o}}.
$
Here vicinity means within a distance $2 c_o T_f$ traveled by the waves over the duration of the probing pulse $f$. 
At larger distance the medium is anisotropic and heterogeneous, with variable wave speed $\bc$.

To derive the expression of the ROM, we will use functional calculus on the operator  $\bc^2 \nabla^\perp [ \nabla^\perp \cdot ]$
with domain 
\begin{equation}
\itbf{H}_0(\nabla^\perp;\Omega) = \left\{ \bpsi(\bx) \in \bL^2(\Omega): ~ \nabla^\perp \cdot \bpsi(\bx) \in \bL^2(\Omega), ~~ 
\bn^\perp \cdot \bpsi |_{\partial \Omega} = 0 \right\}.
\end{equation}
This operator 
 is the two dimensional version of the operator $\bc^2\mbox{curl}\, \mbox{curl}$ with domain  ${\itbf{H}}_0(\mbox{curl};\Omega)$ \cite[Theorem 3.33]{monk2003finite}. 
 
 Note that $\bc^2 \nabla^\perp [ \nabla^\perp \cdot ]$ is symmetric with respect to the inner product weighted by $\bc^{-2}$. We prefer to work with the Euclidian inner product, so we define 
\begin{equation}
\bU^{(s,p)}(t,\bx) = c_o \bc^{-1}(\bx) \bE^{(s,p)}(t,\bx), \quad t \in \RR, ~ \bx \in \Omega,
\label{eq:defu}
\end{equation}
which equals the electric field in the vicinity of the array, because $\bc = c_o \underline{\underline{\bI}}$ there. 
We derive from ~\eqref{eq:6},~\eqref{eq:We1} and~\eqref{eq:We6}, with excitation~\eqref{eq:Ant1}, the equation
\begin{equation}
\left(A + \partial_t^2 \right)\bU^{(s,p)}(t,\bx) = -c_o^2 f'(t) F^{(s)}(\bx) \be_p, \quad t > 0, ~ \bx \in \Omega,
\label{eq:u1}
\end{equation}
with homogeneous initial and boundary conditions
\begin{align}
\bU^{(s,p)}(t,\bx)&= {\bf 0}, \quad t  < -T_f, ~ \bx \in \Omega, \label{eq:u01}\\
\bn^\perp \cdot \bU^{(s,p)}(t,\bx) &= 0, \quad t \in \RR, ~ \bx \in \partial \Omega.
\label{eq:u3}
\end{align}
Here we introduced the operator 
\begin{equation}
A = - \bc(\bx) \nabla^\perp \left[ \nabla^\perp \cdot \left( \bc(\bx) \cdot \right) \right],
\label{eq:u4}
\end{equation}
with domain $\cD$ consisting of vector valued functions whose multiplication by $\bc$ lies in the space ${\itbf{H}}_0(\nabla^\perp;\Omega)$. 

\subsection{Wave decomposition and data transformation}
\label{sect:waveDec}
The operator $A$ defined in~\eqref{eq:u4} is self-adjoint and positive semidefinite. Its null space  is a closed subspace of $\cD$ and consists of functions of the form 
$\bc^{-1} \nabla {N}$, where ${N}$ is scalar valued in $H^1(\Omega)$, with constant trace at $\partial \Omega$ \cite[Lemma 4.5]{monk2003finite}.

The 
solution of~\eqref{eq:u1}--\eqref{eq:u3} has the Helmholtz decomposition \cite[Lemma 4.5]{monk2003finite}
\begin{equation}
\bU^{(s,p)}(t,\bx) = \bbU^{(s,p)}(t,\bx) + \bc^{-1}(\bx) \nabla N^{(s,p)}(t,\bx),
\label{eq:u5}
\end{equation}
where $\bc^{-1} \bbU^{(s,p)}$ is divergence free, in the weak sense. The first term in this decomposition solves the wave equation
\begin{equation}
\left( A + \partial_t^2 \right) \bbU^{(s,p)}(t,\bx) = -c_o^2 f'(t) \bbF^{(s,p)}(\bx),
\quad t \in \RR, ~ \bx \in \Omega,
\label{eq:u7}
\end{equation}
driven by the divergence free part $\bbF^{(s,p)}$ of the excitation, the projection of  $F^{(s)}(\bx) \be_p$ on $\mbox{range}(A)$.
This range is a closed subspace of $\bL^2(\Omega)$, consisting of functions $\bpsi$ such that $\bc^{-1} \bpsi $ is divergence free, in the weak sense.  According to \cite[Section 6.6]{mcowen2004partial} or by direct calculation, one can verify that $\left[\mbox{range}(A)\right]^\perp = \mbox{null}(A).$

The second term 
in~\eqref{eq:u5} lies in $\mbox{null}(A)$ and solves
\begin{equation}
\partial_t^2 \nabla \cdot \big[\bc^{-2}(\bx) \nabla N^{(s,p)}(t,\bx)\big] = - c_o f'(t) \partial_{x_p}F^{(s)}(\bx),   \quad t \in \RR,\,  \bx \in \Omega,
\label{eq:u8}
\end{equation}
with  homogeneous initial condition at $t < -T_f$. This term is  short lived, because when we integrate equation~\eqref{eq:u8} in time and use the initial condition, we obtain 
\begin{equation}
\nabla \cdot \big[\bc^{-2}(\bx) \nabla N^{(s,p)}(t,\bx)\big] = - c_o \int_{-T_f}^t dt' \, f(t') \partial_{x_p} F^{(s)}(\bx).
\label{eq:u9}
\end{equation}
The right hand side  of this equation vanishes for $t > T_f$, due to the assumption~\eqref{eq:pulsezero}.  

Since the  
medium is known and homogeneous within a distance $2 c_o T_f$ from the array, we can compute $\bU^{(s,p)}$ at 
$t \in (-T_f,T_f)$ and remove its component in the null space.  The array response matrix  is then transformed to 
$\mathbb{W}(t) \in \RR^{2m \times 2m}$, with entries 
\begin{equation}
\mathbb{W}^{(s',p'),(s,p)}(t) = \be_{p'}^T \int_{\Omega} d \bx \, F^{(s')}(\bx)  \bbU^{(s,p)}(t,\bx) \approx 
\be_{p'}^T \bbU^{(s,p)}(t,\bx_{s'}),
\label{eq:newM}
\end{equation}
for $s,s' = 1, \ldots, m$ and $p,p' = 1,2.$ 

Once we have removed the null space components of the wave, we can work with the operator  $\cA$, the restriction of $A$ to $\cD \setminus \mbox{null}(A)$. This operator is  positive definite, with compact and self-adjoint resolvent \cite[Section 4.7]{monk2003finite}. It has a countable infinite set of positive eigenvalues $(\theta_n)_{n \ge 1}$ sorted in increasing order, with $\theta_n \to \infty$ as $n \to \infty$, and the eigenfunctions $(\bphi_n)_{n \ge 1}$ form an orthonormal basis of $\mbox{range} (A)$. Functional calculus on $\cA$ is defined as usual: If $\Psi:\mathbb{C} \mapsto \mathbb{C}$ is a continuous function, then $\Psi(\cA)$ is the operator with the same eigenfunctions as $\cA$ and the eigenvalues $\Psi(\theta_n)$, for $n \ge 1$.

Let us write the solution of equation~\eqref{eq:u7} as a  series\footnote{The square bracket equals 
$ \int_{\Omega} d \bx' \, \bbF^{(s,p)}(\bx') \cdot  \bphi_j(\bx')$, because $F^{(s)}(\bx) \be_p -  \bbF^{(s,p)} \in 
\left[\mbox{range}(A)\right]^\perp$ and $\bphi_j \in \mbox{range}(A)$.}
 using the spectum of $\cA$,
\begin{equation*}
\hspace{-0.1in}\bbU^{(s,p)}(t,\bx) = -c_o^2 f(t) \star_t 1_{[0,\infty)}(t) \sum_{j=1}^\infty  \cos (t \sqrt{\theta_j}\, ) \Big[\be_p^T \int_{\Omega} d \bx' \,
F^{(s)}(\bx') \bphi_j(\bx')\Big] \bphi_j(\bx),
\end{equation*}
where $1_{[0,\infty)}$ is the indicator function of the interval $[0,\infty)$ equal to $1$ if $ t \ge 0$ and zero otherwise and 
$\star_t$ denotes convolution in $t$. 
For the derivation of the ROM, it is convenient to work with the wave 
\begin{align}
\bbU^{(s,p)}_e(t,\bx) &= -c_o^{-2} \left[f(-t)\star_t\bbU^{(s,p)}(t,\bx) + f(t)\star_t\bbU^{(s,p)}(-t,\bx)\right]\nonumber \\
&= \sum_{j=1}^\infty \cos (t \sqrt{\theta_j}\, ) |\hat f (\sqrt{\theta_j}\, )|^2 \Big[\be_p^T \int_{\Omega} d \bx' \,
F^{(s)}(\bx') \bphi_j(\bx')\Big]\bphi_j(\bx),
\label{eq:waveEv}
\end{align}
where the last equality is obtained by writing the time convolutions  in the Fourier domain, using the Fourier transform formula
\[
\int_{-\infty}^\infty dt \, 1_{[0,\infty)}(t)  \cos (t \sqrt{\theta_j}) e^{i \om t} = \frac{\pi}{2} \left[ \delta (\om - \sqrt{\theta_j}\, )
+ \delta (\om + \sqrt{\theta_j}\, )\right] + \frac{ i \om}{\theta_j - \om^2}\, ,
\]
and also that $|\hat f(\om)|$ is even.


The purpose of introducing the even in time wave $\bbU_e^{(s,p)}$ is twofold: 
First,   the transformation in~\eqref{eq:waveEv} is like a Duhamel principle which maps the source excitation to an initial condition. Indeed, the right hand side in~\eqref{eq:waveEv}
equals 
\begin{equation}
 \bbU^{(s,p)}_e(t,\bx) = | \hat f ( \sqrt{\cA}\,  ) | \bu^{(s,p)}(t,\bx),
\label{eq:Sn0}
\end{equation}
where 
\begin{equation}
\bu^{(s,p)}(t,\bx)  = \cos (t \sqrt{\cA} \, ) \bu^{(s,p)}_0(\bx), 
\label{eq:Sn}
\end{equation}
solves the homogeneous wave equation 
\begin{equation}
[\cA + \partial_t^2] \bu^{(s,p)}(t,\bx) = 0, \quad t > 0, ~~ \bx \in \Omega,
\label{eq:Sn1}
\end{equation}
with initial condition
\begin{equation}
\bu^{(s,p)}(0,\bx) = \bu_0^{(s,p)}(\bx), \quad \partial_t \bu^{(s,p)}(0,\bx) = 0, \quad \bx \in \Omega,
\label{eq:Sn2}
\end{equation}
defined by
\begin{align}
\bu_0^{(s,p)}(\bx)&= \sum_{j=1}^\infty | \hat f(\sqrt{\theta_j}\, ) | \Big[\be_p^T \int_{\Omega} d \bx' \,
F^{(s)}(\bx') \bphi_j(\bx')\Big]  \bphi_j(\bx).
\label{eq:DM3}
\end{align}
The second purpose of~\eqref{eq:waveEv} is that the measurements~\eqref{eq:newM}, which are approximately 
point evaluations of the wave fields,  can be mapped to a new data matrix 
\begin{equation}
\bbD(t) = -c_o^{-2}\left[f(-t) \star_t \mathbb{W}(t) + f(t) \star_t \mathbb{W}(-t)\right],
\label{eq:newD}
\end{equation}
whose entries can be expressed as the inner products 
\begin{align} 
\mathbb{D}^{(s',p'),(s,p)}(t) &= \int_{\Omega} d \bx \, \be_{p'}^T F^{(s')}(\bx)  \bbU^{(s,p)}_e(t,\bx) \nonumber \\
&\hspace{-0.25in}\stackrel{\eqref{eq:Sn0},\eqref{eq:DM3}}{=}  \int_{\Omega} d \bx \, \left[\bu^{(s',p')}_0(\bx) \right]^T   \bu^{(s,p)}(t,\bx),
\label{eq:Dinn}
\end{align}
for $s,s' = 1, \ldots, m$ and $p,p' = 1,2$. As we shall see, the matrices arising in the Galerkin projection scheme that define our ROM also involve such inner products, which is why we can compute the ROM in a data driven way.

Note that the second term in definition~\eqref{eq:newD}  contributes only at $t < 2 T_f$ and it can be determined 
from the measurements if the recordings start at $t =-T_F$. Otherwise, we can compute it  because the medium is known and homogeneous near the array. We work henceforth with the data matrix $\bbD$ and show below how we can compute the ROM from it.

\subsection{Wave snapshots and time stepping}
\label{sect:TimeStep}
We are interested in the evolution of the wave fields~\eqref{eq:Sn} on a uniform time grid $\{t_j = j \tau, j \ge 0\}$, with step  chosen according to the Nyquist criterium  $\tau \lesssim \pi/\om_o$. To avoid heavy index notation, we use block algebra and define 
the $2 \times 2m$ dimensional  field, called the wave snapshot at time $t_j$,
\begin{align}
\bu_j(\bx) &= \left( \bu^{(1,1)}(t_j,\bx), \bu^{(1,2)}(t_j,\bx), \ldots, \bu^{(m,1)}(t_j,\bx), \bu^{(m,1)}(t_j,\bx)\right) \nonumber \\
&\hspace{-0.07in}\stackrel{\eqref{eq:Sn}}{=}\cos (j\tau \sqrt{\cA} \, ) \bu_0(\bx), \quad \bx \in \Omega, ~~ j \ge 0.
\label{eq:Sn3}
\end{align}

We can now see the advantage of the transformations 
in the previous section: The trigonometric identity 
$
\cos[(j+1) \alpha] + \cos[|j-1|\alpha] = 2 \cos(\alpha) \cos(j \alpha),$ for  $\alpha \in \RR,
$
and functional calculus give that the snapshots satisfy the exact time stepping scheme
\begin{equation}
\bu_{j+1}(\bx) = 2 \cP \bu_j (\bx) - \bu_{|j-1|}(\bx), \quad j \ge 0, ~ \bx \in \Omega,
\label{eq:Sn4}
\end{equation}
driven by the ``propagator operator" 
\begin{equation}
\cP = \cos (\tau \sqrt{\cA} \, ).
\label{eq:Sn5}
\end{equation}
Equation~\eqref{eq:Sn4} defines a discrete time dynamical system, with initial state $\bu_0$. The ROM derived next is an 
algebraic discrete dynamical system, obtained via Galerkin projection of~\eqref{eq:Sn4} on the space 
\begin{equation}
\mathscr{S} = \mbox{span} \{\bcU(\bx)\}, \quad \bcU(\bx) = \left( \bu_0(\bx),\ldots, \bu_{n-1}(\bx) \right), \quad \bx \in \Omega.
\label{eq:Sn6}
\end{equation}
The end time index $n$ is chosen according to the distance from the array at which we wish to image. If this distance is 
$L$, then we should have $n c_o \tau \gtrsim L$. 

\begin{rem} 
\label{rem:FullRank}
If the data are noiseless, the time step $\tau$ is chosen according to the Nyquist criterium and the 
antenna separation is of order $\la_o$, the dimension of $\mathscr{S}$ is typically $2 nm$. We assume that this is so for now, but we discuss later, in section~\ref{sect:ROMnoise},  how to deal with the case of linearly dependent snapshots.
\end{rem}

\subsection{Data driven ROM}
\label{sect:ROMComp}
The Galerkin approximation of the snapshots~\eqref{eq:Sn3} is defined in a standard way
\begin{equation}
\bu^{\rm Gal}_j(\bx) = \bcU(\bx) \boldsymbol{\mathfrak{g}}_j,  \quad j \ge 0, ~ \bx \in \Omega, \label{eq:DS8}
\end{equation}
where $\boldsymbol{\mathfrak{g}}_j \in \RR^{2mn \times 2m}$ are the  Galerkin coefficients, calculated so that when substituting
\eqref{eq:DS8} in~\eqref{eq:Sn4}, the residual is orthogonal to the approximation space $\mathscr{S}$. 

To write the Galerkin equation,  we use an orthonormal basis stored in the $2 \times 2nm$  dimensional field 
$
\bcV(\bx) =  \left( \bv_0(\bx),\ldots, \bv_{n-1}(\bx) \right),$ whose $2 \times 2m$ components satisfy 
\begin{equation}
\int_{\Omega} d \bx \, \bv_j^T(\bx) \bv_k(\bx) = \bI_{2m} \delta_{jk}, \quad j,k = 0, \ldots, n-1.
\label{eq:defV2}
\end{equation}
Here $\bI_{2m}$ denotes the $2m \times 2m $ identity matrix and $\delta_{jk}$ is the Kronecker delta symbol. We index the components 
$\bv_j$ of $\bcV$ the same way as the snapshots, because our basis is causal i.e., 
\begin{equation}
\bv_j(\bx) \in \mbox{span} \left\{ \bu_0(\bx), \ldots, \bu_j(\bx) \right\}, \quad j = 0, \ldots, n-1.
\label{eq:defV3}
\end{equation}
Explicitly, the definition of $\bcV$ is via the Gram-Schmidt orthogonalization of $\bcU$, 
\begin{equation}
\bcU(\bx) = \bV(\bx) \bR, \quad \bx \in \Omega,
\label{eq:defV4}
\end{equation}
where $\bR$ is block upper triangular, with $2m \times 2m$ blocks. 

The equation for the Galerkin coefficients can now be written as 
\begin{align}
{\bf 0} &= \int_{\Omega} d \bx \, \bV^T(\bx) \left[ \bcU(\bx) \left( \boldsymbol{\mathfrak{g}}_{j+1} + \boldsymbol{\mathfrak{g}}_{|j-1|} \right) - 2 \cP \bcU(\bx) \boldsymbol{\mathfrak{g}}_j \right] \nonumber \\
&= \bR^{-T} \left[{\mathbb{M}} \left( \boldsymbol{\mathfrak{g}}_{j+1} + \boldsymbol{\mathfrak{g}}_{|j-1|} \right) - 2 \mathbb{S}  \boldsymbol{\mathfrak{g}}_j \right], \quad j \ge 0,
\label{eq:Ga1}
\end{align}
with initial condition
\begin{equation}
\boldsymbol{\mathfrak{g}}_0 = \bi_0, \mbox{ such that} ~ \bu^{\rm Gal}_0(\bx) = \bu_0(\bx), \quad \bx \in \Omega.
\label{eq:Ga2}
\end{equation}
Here $\bR^{-T}$ is the transpose of the inverse of $\bR$ and we introduced the Gramian, also known as ``mass" matrix 
\begin{equation}
\mathbb{M} = \int_{\Omega} d \bx \, \bcU^T(\bx) \bcU(\bx) \in \RR^{2nm \times 2nm},
\label{eq:defmass}
\end{equation}
and the ``stiffness matrix" 
\begin{equation}
\mathbb{S} = \int_{\Omega} d \bx \, \bcU^T(\bx) \cP \bcU(\bx) \in \RR^{2nm \times 2nm},
\label{eq:defstiff}
\end{equation}
of the Galerkin scheme. We also denote  by $\bi_j$ the $(j+1)^{\rm th}$ column block of size $2nm \times 2m$, of the $2nm \times 2nm$ identity matrix $\bI_{2nm}$.

The following theorem is a straightforward generalization of the results in ~\cite{borcea2020reduced,borcea2019robust,DtB,druskin2016direct}, 
obtained for the scalar wave equation. We include its proof for the convenience of the reader. The result  says that  even though we do 
not know the snapshots stored in $\bcU$ and therefore the approximation 
space~\eqref{eq:Sn6}, all the terms in the Galerkin equation~\eqref{eq:Ga1} are data driven.

\vspace{0.05in}
\begin{theorem}
\label{thm.1} 
The $2m \times 2m$ blocks of the mass and stiffness matrices~\eqref{eq:defmass}-\eqref{eq:defstiff} are determined 
by the data matrices~\eqref{eq:newD} as follows
\begin{align}
\mathbb{M}_{j,l} &= \frac{1}{2} \left[ \mathbb{D}(t_{j+l}) + \mathbb{D}(t_{|j-l|})\right], \label{eq:calcGram} \\
 \mathbb{S}_{j,l} &= \frac{1}{4} \left[ \mathbb{D}(t_{j+l+1}) + \mathbb{D}(t_{|j-l-1|}) + \mathbb{D}(t_{j+l-1}) + \mathbb{D}(t_{|j-l+1|} )
 \right],
\label{eq:calcStif}
\end{align}
for $j, l = 0, \ldots, n-1$. The block upper triangular matrix $\bR$ in the Gram-Schmidt orthogonalization is the 
block Cholesky square root of the mass matrix\footnote{
To ensure a unique block Cholesky factorization, we take the symmetric, positive definite square root of the diagonal blocks. }
\begin{equation}
\mathbb{M} = \bR^T \bR.
\label{eq:Chol}
\end{equation}
The first $n$  Galerkin coefficients are the $2nm \times 2m$ block columns of $\bI_{2nm}$, 
\begin{equation}
\boldsymbol{\mathfrak{g}}_j = \bi_j, \quad j = 0, \ldots, n-1.
\label{eq:FirstGal}
\end{equation}
\end{theorem}
\begin{proof}
Note that equation~\eqref{eq:FirstGal} combined with definition~\eqref{eq:DS8} give that the Galerkin approximation of the 
snapshots is exact for the first $n$ times instants. This seems natural because the first $n$ snapshots span the approximation space. However, for the result to hold, it is essential that the time stepping scheme~\eqref{eq:Sn4} is exact, so that 
with the coefficients~\eqref{eq:FirstGal} we have a zero residual. If we used an approximate time stepping, obtained 
via some finite difference approximation of $\partial_t^2$ in the wave equation, ~\eqref{eq:FirstGal} would not be satisfied.

The proof of~\eqref{eq:FirstGal} is as follows: According to Remark \ref{rem:FullRank}, $\mbox{dim} (\mathscr{S}) = 2nm$.
This means that the matrix $\bR$ in the Gram-Schmidt orthogonalization~\eqref{eq:defV4}  and the mass matrix are full rank. Then, equation~\eqref{eq:Ga1} with initial condition~\eqref{eq:Ga2} has a unique solution $\boldsymbol{\mathfrak{g}}_j$, for $ j \ge 1.$ 
Since 
\begin{equation*}
\bcU(\bx) (\bi_{j+1} + \bi_{|j-1|}) - 2 \cP \bcU(\bx) \bi_j = \bu_{j+1} + \bu_{|j-1|} - 2 \cP \bu_j \stackrel{\eqref{eq:Sn4}}{=} {\bf 0}, \quad j = 0, \ldots, n-1,
\end{equation*}
the coefficients  $\boldsymbol{\mathfrak{g}}_j = \bi_j$ solve~\eqref{eq:Ga1} for $j = 0, \ldots, n-1$. This proves~\eqref{eq:FirstGal}.

The block Cholesky factorization~\eqref{eq:Chol}  is deduced from the definition~\eqref{eq:defmass} of the mass matrix and 
the Gram-Schmidt equation~\eqref{eq:defV4}
\begin{equation*}
\mathbb{M} = \bR^T \int_{\Omega} d \bx \, \bV^T(\bx) \bV(\bx) \bR \stackrel{\eqref{eq:defV2}}{=} \bR^T \bR.
\end{equation*}

The expression~\eqref{eq:calcGram} of the blocks of  $\mathbb{M}$ is derived using the 
definition of the snapshots and the symmetry of $\cA$, as follows
\begin{align*} 
\mathbb{M}_{j,l} &\stackrel{\eqref{eq:defmass}}{=} \int_{\Omega} d \bx \, \bu_j^T(\bx) \bu_l(\bx) \\
 &\stackrel{\eqref{eq:Sn3}}{=} \int_{\Omega} d \bx \, \bu_0^T(\bx) \cos( j \tau \sqrt{\cA} \, ) \cos( l  \tau\sqrt{\cA} \, ) 
 \bu_0(\bx) \\
 &\hspace{0.08in}=  \int_{\Omega} d \bx \, \bu_0^T(\bx) \frac{1}{2}\left[ \cos( (j+l) \tau \sqrt{\cA} \, ) + \cos( |j-l| \tau \sqrt{\cA} \, )
 \right] \bu_0(\bx) \\
 &\stackrel{\eqref{eq:Sn3},\eqref{eq:Dinn}}{=} \frac{1}{2} \left[ \mathbb{D}(t_{j+l}) + \mathbb{D}(t_{|j-l|})\right], \quad j, l = 0,\ldots, n-1.
 \end{align*}
Finally, the blocks of the stiffness matrix are 
\begin{align*} 
\mathbb{S}_{j,l} &\stackrel{\eqref{eq:defmass}}{=} \int_{\Omega} d \bx \, \bu_j^T(\bx) \cP \bu_l(\bx) \stackrel{\eqref{eq:Sn4}}{=} \int_{\Omega} d \bx \, \bu_j^T(\bx) \frac{1}{2}\left[ \bu_{l+1}(\bx) + \bu_{|l-1|}(\bx) \right],
\end{align*}
and the result~\eqref{eq:calcStif} follows from the calculation above. \qed
\end{proof}

\subsubsection{The ROM} We define the ROM as an algebraic, discrete time dynamical system analogue of~\eqref{eq:Sn4}, 
with states 
\begin{equation}
\buR_j = \bR \boldsymbol{\mathfrak{g}}_j, \quad j \ge 0,
\label{eq:ROMsnap}
\end{equation}
called the ROM snapshots. The first $n$ such snapshots are the $2nm \times 2m$ block columns of $\bR$, 
as follows from equation~\eqref{eq:FirstGal},
\begin{equation}
\left(\buR_0, \ldots, \buR_{n-1} \right) = \bR.
\label{eq:ROM1}
\end{equation}
The evolution of the ROM snapshots is dictated by the equation 
\begin{equation}
\buR_{j+1} = 2 \cPR \buR_j - \buR_{|j-1|}, \quad j \ge 0,
\label{eq:ROM2}
\end{equation}
derived from~\eqref{eq:Ga1} and the Cholesky factorization~\eqref{eq:Chol}, where 
\begin{equation}
\cPR = \bR^{-T} \mathbb{S} \bR^{-1}
\label{eq:ROM3}
\end{equation}
is the $2nm \times 2nm$ ROM propagator matrix.

By  Theorem \ref{thm.1}, $\bR$ and $\mathbb{S}$ are data driven, so equations~\eqref{eq:ROMsnap} and~\eqref{eq:ROM3} 
give that the ROM can be computed directly from $\mathbb{D}(t_j)$, for $j = 0, \ldots, 2n-1$. 

\subsubsection{Properties of the ROM} If we solve for $\bR$ in the Gram-Schmidt orthogonalization equation~\eqref{eq:defV4}
and then use the result in the definition~\eqref{eq:ROMsnap} of the ROM snapshots, we obtain that 
\begin{equation}
\buR_j = \int_{\Omega} d \bx \, \bV^T(\bx) \bU(\bx) \boldsymbol{\mathfrak{g}}_j \stackrel{\eqref{eq:DS8}}{=} 
\int_{\Omega} d \bx \, \bV^T(\bx) \bu^{\rm Gal}_j(\bx), \quad j \ge 0.
\label{eq:ROM4}
\end{equation}
Thus, the ROM snapshots are the orthogonal projection of the Galerkin approximation of the 
snapshots on the space $\mathscr{S}$ defined in~\eqref{eq:Sn6}. Since this approximation is exact at the first $n$ time instants, we deduce from this equation and~\eqref{eq:defV4} that 
\begin{equation}
\buR_j = \int_{\Omega} d \bx \, \bV^T(\bx) \bu_j(\bx) \quad \mbox{and} \quad \bu_j(\bx) = \bV(\bx) \buR_j, \quad j = 0, \ldots, n-1.
\label{eq:ROM5}
\end{equation}

The ROM propagator matrix defined in~\eqref{eq:ROM3} equals the orthogonal projection of the propagator operator~\eqref{eq:Sn5} on $\mathscr{S}$, 
\begin{align}
\hspace{-0.1in}\cPR = \bR^{-T} \mathbb{S} \bR^{-1} &\stackrel{\eqref{eq:defstiff}}{=} 
\int_{\Omega} d \bx \, \bR^{-T} \bU^T(\bx) \cP \bU(\bx) \bR^{-1}  \stackrel{\eqref{eq:defV4}}{=}
\int_{\Omega} d \bx \, \bV^T(\bx) \cP \bV(\bx).
\label{eq:ROM6}
\end{align}
It is a symmetric $2nm \times 2nm$ matrix, with block tridiagonal structure.  One way to deduce this structure is to follow the proof in \cite[Appendix C]{borcea2020reduced}. Alternatively, the result follows by iterating equation~\eqref{eq:ROM2} for $j = 0, \ldots, n-2$ 
and using~\eqref{eq:ROM1}: If we count the $2m \times 2m$ blocks of $\cPR$ as $\cPR_{j,l}$, with $j, l = 0, \ldots, n-1$, and we let $\bR_{j,l}$ be the blocks of $\bR$, which are non-zero for $0 \le j \le l \le n-1$, 
we obtain from 
\eqref{eq:ROM2} evaluated at $j = 0$ that 
\[
2 \cPR \bR \bi_0 = 2 \begin{pmatrix} \cPR_{0,0} \bR_{0,0}\\ \cPR_{1,0} \bR_{0,0}\\ \vdots \\
\cPR_{n-1,0} \bR_{0,0} \end{pmatrix}  
= 2 \bR \bi_1 = 2 \begin{pmatrix} \bR_{0,1} \\ \bR_{1,1} \\ {\bf 0} \\ \vdots \\ {\bf 0} \end{pmatrix}.
\]
The block $\bR_{0,0}$ is invertible, because $\bR$ is invertible, so we must have 
$
\cPR_{j,0} = \cPR_{0,j} = {\bf 0}$ for $j  = 2, \ldots, n-1.$ 
The next step, obtained from~\eqref{eq:ROM2} evaluated at $j = 1$ gives 
\[
2 \cPR \bR \bi_1 = \bR \bi_2 + \bR \bi_0 \quad \Longrightarrow \quad 
\cPR_{j,1} \bR_{1,1} = {\bf 0}, \quad j = 3, \ldots, n-1,
\]
which means, since $\bR_{1,1}$ is invertible, that 
$
\cPR_{j,1} = \cPR_{1,j} = {\bf 0},$ for $j  = 3, \ldots, n-1.
$
Proceeding this way, until we reach $j = n-2$ in equation~\eqref{eq:ROM2}, we 
obtain that $\cPR$ is indeed block tridiagonal.

The ROM interpolates the time dependent data matrix $\mathbb{D}$ from which it is computed. Specifically, 
for $j = 0, \ldots, n-1$, we have 
\begin{align}
\mathbb{D}(t_j) &\stackrel{\eqref{eq:Dinn}}{=} \int_{\Omega} d \bx \, \bu_0^T (\bx) \bu_j(\bx) 
\stackrel{\eqref{eq:ROM5}}{=} (\buR_0)^T \int_{\Omega} d \bx \bV^T(\bx) \bV(\bx) \buR_j \stackrel{\eqref{eq:defV2}}{=} (\buR_0)^T \buR_j,  \label{eq:ROM7}
\end{align}
and for $j = 1, \ldots, n-1$ we have 
\begin{align}
\mathbb{D}(t_{n-1+j}) &\stackrel{\eqref{eq:calcGram}}{=} 2 \mathbb{M}_{n-1,j} - \mathbb{D}(t_{n-1-j}).
\label{eq:ROM7.1}
\end{align}
The second term in the right hand side of this equation is matched as in~\eqref{eq:ROM7}
and the first term is 
\begin{align}
\mathbb{M}_{n-1,j} &= \int_{\Omega} d \bx \, \bu_{n-1}^T(\bx) \bu_j(\bx) \hspace{-0.06in}\stackrel{\eqref{eq:ROM5}}{=}\hspace{-0.06in}
(\buR_{n-1})^T \hspace{-0.06in}\int_{\Omega} d \bx \bV^T(\bx) \bV(\bx) \buR_j  \hspace{-0.04in}\stackrel{\eqref{eq:defV2}}{=} \hspace{-0.06in}(\buR_{n-1})^T \buR_j.
\label{eq:ROM8}
\end{align}
Equations~\eqref{eq:ROM7}--\eqref{eq:ROM8} show that the ROM snapshots give exactly the 
matrices $\mathbb{D}(t_j)$ at instants $t_j = j \tau$, for $j = 0, \ldots, 2n-2.$ In fact, with a bit more work, 
one can show that this result extends one more step, to $j = 2n-1$ \cite[Appendix B]{borcea2020reduced}.

\subsection{Regularization of the ROM computation}
\label{sect:ROMnoise}
There are two critical and nonlinear steps in the computation of the ROM: The Cholesky factorization~\eqref{eq:Chol} of 
$\mathbb{M}$ and the inversion of the block Cholesky square root $\bR$. Both steps require a symmetric positive definite 
$\mathbb{M}$ computed from the data as in equation~\eqref{eq:calcGram}. 

In the absence of noise, the time dependent 
matrix $\mathbb{D}$, which is obtained from the array response matrix as in equation~\eqref{eq:newD}, is symmetric due to 
source/receiver reciprocity. This ensures the symmetry of $\mathbb{M}$. The matrix $\mathbb{D}$ is not symmetric for noisy measurements, but once we compute 
$\mathbb{M}$ from equation~\eqref{eq:calcGram}, we can just take $(\mathbb{M} + \mathbb{M}^T)/2$ to obtain a symmetric 
mass matrix.

If the time step $\tau$  and/or the separation between the antennas are chosen too small, then $\mathbb{M}$ is singular in finite precision arithmetic. For noisy measurements, even after the symmetrization, $\mathbb{M}$ is likely indefinite. Thus, the ROM construction requires a regularization procedure, which maps $\mathbb{M}$ to a symmetric, positive definite $\mathbb{M}^{{\rm reg}}$. 

There are various ways to carry out the regularization. Each one must ensure the correct causal structure of the ROM, 
which is manifested algebraically in the block upper triangular square root $\bR$ of $\mathbb{M}$ and the  tridiagonal structure of the ROM propagator matrix $\cPR$. We have tried the following two regularization approaches: 

\vspace{0.05in}
\begin{enumerate} 
\itemsep 0.05in
\item Replace $\mathbb{D}(0)$ by $(1+ 2 \alpha)\mathbb{D}(0)$, which according to~\eqref{eq:calcGram}, amounts to boosting the diagonal block of 
$\mathbb{M}$ by $\alpha \mathbb{D}(0)$, where $0 < \alpha \ll 1$.
\item Project $\mathbb{M}$ on the space spanned by the eigenvectors corresponding to the eigenvalues 
that exceed a user defined positive threshold.
\end{enumerate}

\vspace{0.05in} The advantage of the first regularization approach is that it is simple and preserves the causal algebraic 
structure of the ROM.  The matrix $\mathbb{D}(0)$ can be easily computed because 
it depends only on the medium near the array, which is known, isotropic and homogeneous. It is a positive definite 
$2m \times 2m$ matrix, so by adding $\alpha \mathbb{D}(0)$ to the diagonal of $\mathbb{M}$, we can get a positive definite 
$2nm \times 2nm$ regularized $\mathbb{M}^{{\rm reg}}$. The disadvantage is that the eigenvectors
of $\mathbb{M}^{{\rm reg}}$ corresponding to the smallest eigenvalues are significantly affected by noise, which can cause inversion artifacts.

The second regularization approach is more complicated, but it has the advantage that it removes the subspace  spanned by the 
``noisy eigenvectors". 
Let $\{\by_j\}_{j=1}^{2nm}$ be the orthonormal eigenfunctions of $\mathbb{M}$, corresponding to the eigenvalues counted in descending order. For regularization, we wish to keep the largest $2rm$ eigenvalues and associated eigenvectors, for $1 \le r < n$. The multiple $2m$ is 
needed here to carry out the block algebra calculations, with $2m \times 2m$ blocks. The projection of $\mathbb{M}$ on the 
space spanned by the leading eigenvectors is 
$
\boldsymbol{\Lambda} = \bY^T \mathbb{M} \bY,$ where $\bY = (\by_1, \ldots, \by_{2rm}).
$
The matrix $\boldsymbol{\Lambda}$  is positive definite, with spectrum that is weakly affected by noise. However, $\boldsymbol{\Lambda}$  is  diagonal, while the mass matrix should have block Hankel plus Toeplitz structure (recall~\eqref{eq:calcGram}), in order to get  a ROM that preserves causality at the algebraic level.  This causality is manifested in the block tridiagonal structure of the ROM propagator matrix.
To obtain the desired  $\mathbb{M}^{{\rm reg}}$,  we compute first
\begin{equation}
\Pi = \boldsymbol{\Lambda}^{-1/2}  \bY^T \mathbb{S} \bY \boldsymbol{\Lambda}^{-1/2},
\end{equation} 
and then use the block Lanczos algorithm~\cite{golubVanLoan} to obtain the orthogonal matrix $\mathbb{Q} \in \RR^{2rm \times 2rm}$ 
that gives a block tridiagonal $\mathbb{Q}^T \Pi \mathbb{Q}$. This is the regularized ROM propagator. The regularized mass matrix is 
\begin{equation}
\mathbb{M}^{{\rm reg}} = \mathbb{Q}^T \boldsymbol{\Lambda} \mathbb{Q}.
\end{equation}

\section{Estimation of the internal wave}
\label{sect:INT_WAVE}
It is known that knowledge of the internal wave field i.e., the snapshots $\bu_j$ at all points in $\Omega$, would simplify considerably the inverse problem. This is the whole point of   multi-physics ``hybrid" imaging modalities  
 like photo-acoustic tomography, transient elastography, etc.~\cite{nachman2009recovering,bal2013hybrid}. These modalities  are limited to medical applications because they involve delicate and accurate apparatus for transmitting several types of waves and measuring all around the body. It was shown recently in~\cite{borcea2022internal,borcea2023data}, in the context of inverse scattering with acoustic waves, that the ROM can be used to  estimate the internal wave. Here we extend those results to the vectorial problem, governed by equation~\eqref{eq:Sn1}.
 
 We begin in section~\ref{sect:INT_WAVE1} with the definition of the estimated snapshots  $\bu_j^{\rm est}$, followed by a numerical illustration in section~\ref{sect:INT_WAVE_Num}. Then, we describe in section~\ref{sect:INT_WAVE2} the relation between $\bu_j^{\rm est}$ and the dyadic Green's function. This relation sheds  light on the approximation of $\bu_j$ by  $\bu_j^{\rm est}$ and we use it in the analysis of the imaging method introduced in 
 section~\ref{sect:IMAGE}.
 
 \subsection{Estimated wave snapshots}
 \label{sect:INT_WAVE1}
  
The estimation of the snapshots $\bu_j$ is based on equation~\eqref{eq:ROM5}, which has two factors: The ROM snapshots 
$\{\buR_j = \bR \bi_j\}_{j=0}^{n-1}$,  which are computed from the data, and the $2 \times 2nm$ field $\bV$ that stores the orthonormal basis of the space $\mathscr{S}$. This field cannot be computed because $\mathscr{S}$ is not known, so we estimate the snapshots using the basis computed with the guess wave speed $\tilde \bc$,
\begin{equation}
\bu_j^{\rm est}(\bx;\tilde \bc) =  \bV(\bx;\tilde \bc)  \buR_j = \bV(\bx;\tilde \bc) \bR \bi_j, \quad j = 0, \ldots, n-1, ~~ \bx \in \Omega.
\label{eq:EstWave}
\end{equation}
We call the medium with wave speed $\tilde \bc$ the ``reference medium", but note that $\tilde \bc$ will change during our iterative approach to inversion, 
introduced in section~\ref{sect:INVERSE}.

A very popular imaging method, known as reverse time migration in the geophysics community 
\cite{biondi20063d,bleistein2001multidimensional,symes2008migration} or backprojection in radar~\cite{curlander1991synthetic,cheney2009fundamentals}, is based on the linearization of the forward map, known as the Born approximation. Imaging is carried out by backpropagating the measurements to the 
imaging points in $\Omega$, using the  internal wave computed in the reference medium. The snapshots of this wave satisfy \begin{equation}
\bu_j(\bx;\tilde \bc) =  \bV(\bx;\tilde \bc) \bR (\tilde \bc) \bi_j, \quad j = 0, \ldots, n-1, ~~ \bx \in \Omega.
\label{eq:BornWave}
\end{equation}

Equation~\eqref{eq:BornWave}
may look similar to~\eqref{eq:EstWave}, but there is a big difference: The matrix $\bR$, which stores the first $n$ ROM snapshots, 
is for the true and unknown medium, whereas in~\eqref{eq:BornWave} we have $\bR(\tilde \bc)$ computed in the reference medium. This difference is important, because the estimated snapshots~\eqref{eq:EstWave} are consistent with the measurements, 
whereas those in~\eqref{eq:BornWave} are not. Indeed, if we substitute the snapshots $\bu_j$ in equations~\eqref{eq:ROM7}--\eqref{eq:ROM8} with the estimated ones in~\eqref{eq:EstWave}, we obtain that the data interpolation still holds
\begin{align*}
 \int_{\Omega} d \bx \, \left[ \bu_0^{\rm est}(\bx;\tilde \bc) \right]^T \bu_j^{\rm est}(\bx;\tilde \bc) 
= (\buR_0)^T \underbrace{\int_{\Omega} d \bx \bV^T(\bx;\tilde \bc ) \bV(\bx;\tilde \bc )}_{\bI_{2nm}} \buR_j \stackrel{\eqref{eq:ROM7}}{=} \mathbb{D}(t_j), 
\end{align*}
for $j = 0, \ldots, n-1$. Moreover, we have 
\begin{align*}
 \int_{\Omega} d \bx \,\left[ \bu^{\rm est}_{n-1}(\bx;\tilde \bc)\right]^T \bu^{\rm est}_j(\bx;\tilde \bc) = (\buR_{n-1})^T \underbrace{\int_{\Omega} d \bx \bV^T(\bx;\tilde \bc) \bV(\bx;\tilde \bc)}_{\bI_{2nm}} \buR_j  = \mathbb{M}_{n-1,j},
\end{align*}
which in light of~\eqref{eq:ROM7.1} gives the interpolation of $\mathbb{D}(t_{n-1+j})$, for $j = 1, \ldots, n-1$.
If we replace $\bu_j^{\rm est}$ by ~\eqref{eq:BornWave} in the left hand side of these equations, we deduce the interpolation of $\mathbb{D}(t;\tilde \bc)$, not $\mathbb{D}(t)$. 

\subsection{Numerical illustration of the estimated internal wave}
\label{sect:INT_WAVE_Num}

The exact data fit equations above show  that  all the information about the waves recorded at the array is contained in $\bR$. The orthonormal basis stored in $\bV$ plays no role in the data interpolation. Its purpose is to map the information in $\bR$ from the algebraic space, to the physical space,  at points $\bx \in \Omega$. For acoustic waves, extensive numerical simulations~\cite{borcea2022internal,borcea2022waveform,druskin2018nonlinear,DtB,druskin2016direct}
and analysis in some special media 
\cite[Appendix A]{borcea2021reduced}  have shown that the orthonormal basis stored in $\bV$
depends mostly on the smooth part of the medium, which controls the kinematics of wave propagator. 
The analysis and conclusion in \cite[Appendix A]{borcea2021reduced} extend to the vector valued wave field $\bu$.
For brevity, we do not include it here.  Instead, we show  the results of a numerical simulation that illustrate the typical behavior 
of  the estimated internal wave.

\begin{figure}[h]
\centering
\begin{tabular}{cc}
Wave speed  & True Wave at point "+" \\
\includegraphics[width = 0.4\textwidth]{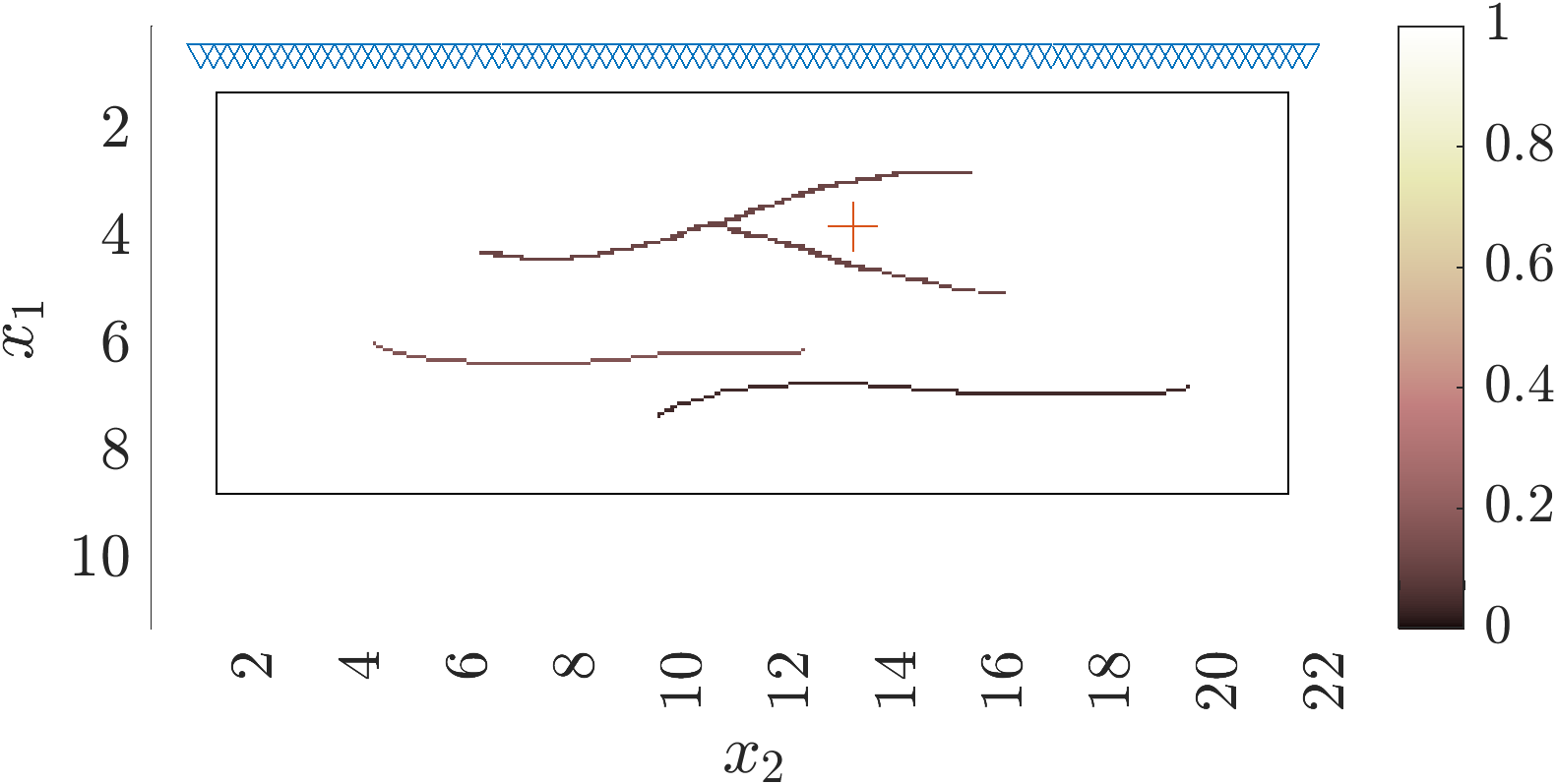} &\includegraphics[width = 0.3\textwidth]{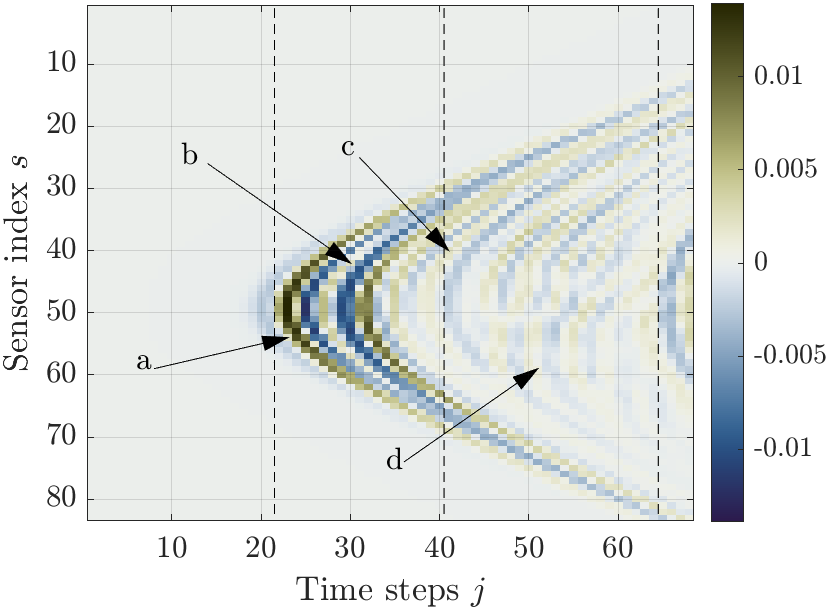} 
\\
Born Approximation & Estimated Wave\\
\includegraphics[width = 0.3\textwidth]{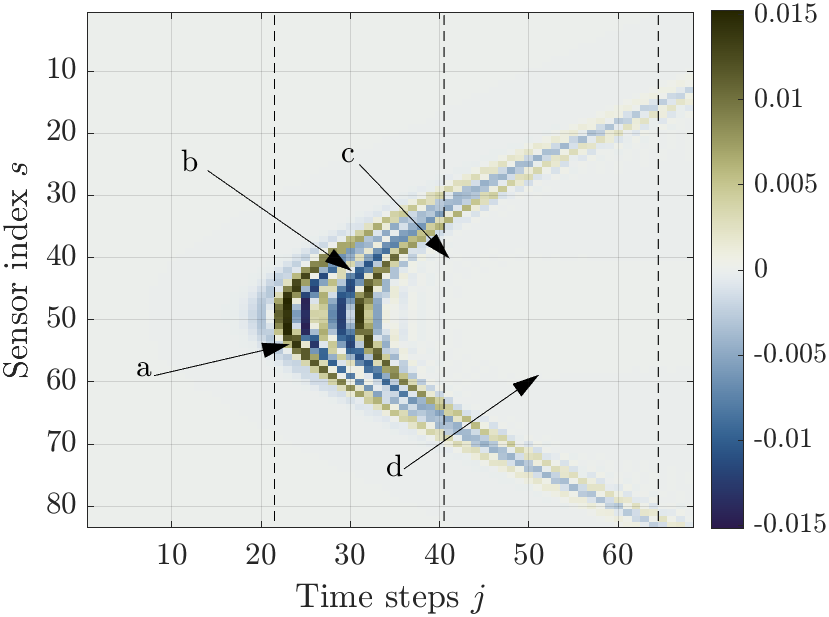}&

\includegraphics[width = 0.3\textwidth]{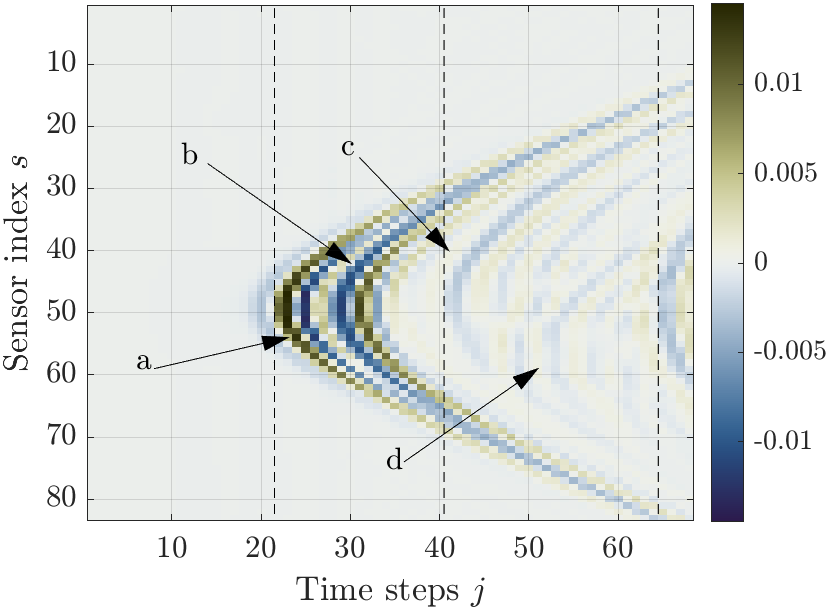} 
\end{tabular}

\caption{Top, from left to right: Plot of $c(\bx)/c_o$ and the wave $\be_2^T \bu_j(\bx)$ at the point $\bx$ indicated with a red cross. 
Bottom, from left to right: The waves $\be_2^T\bu_j(\bx,c_o \tensor{\bI})$ and $\be_2^T\bu_j^{\rm est}(\bx,c_o \tensor{\bI})$. The source is polarized along $\be_2$.}
\label{fig:Internal_Wave}
\end{figure}

In Fig. \ref{fig:Internal_Wave} we compare the components along $\be_2$ of the true wave $\bu_j$, the estimated wave $\bu_j^{\rm est}(\cdot; \tilde \bc)$ and the 
wave $\bu_j(\cdot;\tilde{\bc})$ calculated at the reference wave speed $\tilde \bc = c_o \tensor{\bI}$, at the point $\bx$ indicated with a red cross. The true medium is isotropic, with wave speed $c \tensor{\bI}$ displayed in the top left plot. The excitation is along $\be_2$. We indicate with arrows a  few arrival events: The event denoted by  $a$ is the direct arrival from the source. The event denoted by $b$ is  the wave that traveled from the source to the boundary, scattered there and then arrived at $\bx$. These two events are seen in all the plots. The events $c$ and $d$ are arrivals of waves that scattered in the medium. These are absent in the reference snapshots $\bu_j(\cdot;\tilde{\bc})$ but are seen in both the true wave 
snapshots and the estimated ones. Since in this example the kinematics is perturbed slightly, the estimated snapshots are a good 
approximation of the true ones.

\subsection{Estimated internal wave in terms of the dyadic Green's function}
\label{sect:INT_WAVE2}

The causal dyadic Green's function for the wave equation~\eqref{eq:We1} is the $2\times 2$ matrix valued solution $\bG$ of 
\begin{align}
\hspace{-0.13in} -\nabla^\perp \big[ \nabla^\perp \cdot \bG(t,\bx,\by) \big]+ \bc^{-2}(\bx) \partial_t^2 \bG(t,\bx,\by) &= \delta(t) \underline{\underline{\bI}} \delta(\bx-\by), \quad t \in \RR,~~ \bx \in \Omega,
\label{eq:GF3} \\
\bG(t,\bx,\by)  &\equiv \tensor{\bf 0}, \quad t < 0, ~~ \bx \in \Omega,
\label{eq:GF4}\\
\bn^\perp(\bx) \cdot \bG(t,\bx,\by)  &= {\bf 0}, \quad t \in \RR, ~~ \bx \in \partial \Omega, \label{eq:GF4.1}
\end{align}
where the operators are understood to act columnwise. In light of our transformation of the wave operator in~\eqref{eq:GF3} to  $A + \partial_t^2$, and the fact that in the right hand side of~\eqref{eq:u1} we have a time derivative, 
we also use $\bcG$,  satisfying 
\begin{align}
A\bcG(t,\bx,\by) + \partial_t^2\bcG(t,\bx,\by) &= \delta'(t) \underline{\underline{\bI}} \delta(\bx-\by), \quad t \in \RR, ~~ \bx \in \Omega,
\label{eq:GF1} \\
\bcG(t,\bx,\by)  &\equiv \tensor{\bf 0}, \quad t < 0, ~~ \bx \in \Omega,
\label{eq:GF2} \\
\bn^\perp \cdot \bcG(t,\bx,\by)  &= {\bf 0}, \quad t \in \RR, ~~ \bx \in \Omega.\label{eq:GF2.1}
\end{align}
The next lemma,  proved in Appendix \ref{ap:A}, gives the connection between these dyadics:

\vspace{0.05in}
\begin{lemma}
\label{lem.dyads}
The Green dyadics $\bcG$ and $\bG$ are related by the identity
\begin{equation}
\partial_t  \bG(t,\bx,\by) = \bc(\bx) \bcG(t,\bx,\by)\bc(\by), \quad t \in \RR, ~~ \bx, \by \in \Omega.
\label{eq:GF5}
\end{equation}
\end{lemma}

\vspace{0.05in} 
The approximation of  the forcing $\underline{\underline{\bI}} \delta(\bx-\by)$ in  our space $\mathscr{S}$, with orthonormal basis stored in $\bV$, is 
\begin{equation}
\bde(\bx,\by) = \sum_{j=0}^{n-1} \bv_j(\bx) \bv_j^T( \by) = \bV(\bx) \bV^T (\by).
\label{eq:deltaRange}
\end{equation}
The next lemma, proved in Appendix \ref{ap:B}, shows that the columns of the wave snapshots are wave fields with source 
determined by~\eqref{eq:deltaRange}. The components of the estimated internal wave~\eqref{eq:EstWave}
have a similar expression, except that the source is determined by
\begin{equation}
\bde^{\rm est}(\bx,\by;\tilde \bc) = \sum_{j=1}^\infty \bv_j(\bx) \bv_j^T( \by;\tilde \bc) = \bV(\bx) \bV^T (\by;\tilde \bc).
\label{eq:deltaRangeEst}
\end{equation}

\begin{lemma}
\label{lem.2}
Introduce the $2\times 2$ matrix valued fields 
\begin{equation}
\tensor{\bg}(t,\bx,\by) = \cos (t \sqrt{\cA}) \left| \hat f (\sqrt{\cA}) \right| \bde(\bx,\by),
\label{eq:trueg1}
\end{equation}
and 
\begin{equation}
\tensor{\bg}^{\rm est}(t,\bx,\by;\tilde \bc) = \cos (t \sqrt{\cA}) \left| \hat f (\sqrt{\cA}) \right| \bde^{\rm est}(\bx,\by;\tilde \bc).
\label{eq:estg1}
\end{equation}
Let $j = 0, \ldots, n-1$. 
The wave snapshots  can be written as 
\begin{equation}
\bu_j(\by) \approx \left(\tensor{\bg}(t_j,\bx_1,\by), \ldots, \tensor{\bg}(t_j,\bx_m,\by) \right), \quad  \by \in \Omega,
\label{eq:trueg}
\end{equation}
and the estimated snaphots satisfy
\begin{equation}
 \bu_j^{\rm est}(\by;\tilde \bc)  \approx \left(\tensor{\bg}^{\rm est}(t_j,\bx_1,\by;\tilde \bc), \ldots, \tensor{\bg}^{\rm est}(t_j,\bx_m,\by;\tilde \bc) \right), \quad  \by \in \Omega.
\label{eq:estg}
\end{equation}
\end{lemma}

\vspace{0.05in} The expression~\eqref{eq:trueg} of the wave snapshots says that its $(s,p)$ columns are wave vector fields  that 
originate from the vicinity of the location $\bx_s$ of the $s^{\rm th}$ antenna, with polarization along $\be_p$, for $s = 1, \ldots, m$ and $p = 1,2$. Vicinity means  within the support\footnote{By support we mean the set of points $\by$ around $\bx_s$, where $\bde(\bx_s,\by)$ is significant.}
of  $\bde$ defined in~\eqref{eq:deltaRange}. The estimated snapshots~\eqref{eq:estg} have a similar interpretation, except that their initial support is dictated by $\bde^{\rm est}$ defined in ~\eqref{eq:deltaRangeEst}.  The support of $\bde$ depends on 
how rich the approximation space $\mathscr{S}$ is, while the support of 
 $\bde^{\rm est}$ also depends on the approximation of $\bV$ by $\bV(\cdot; \tilde \bc)$.
This approximation seems to be mostly dependent on the kinematics i.e., the smooth part of  $\bc$.

We can now state the connection between the estimated internal wave and the Green dyadic. The proof is in appendix \ref{ap:C}
and the result is used in the analysis of the imaging function introduced in the next section.

\vspace{0.05in}
\begin{proposition}
\label{prop.1}
Let $\check{f}$ be the even time signal with Fourier transform $|\hat f|$,
\begin{equation}
\check{f}(t) = \int_{-\infty}^\infty \frac{d \om}{2 \pi} |\hat f(\om)| \cos(\om t).
\label{eq:checkf}
\end{equation}
Suppose that $\bde^{\rm est}(\bx,\by;\tilde \bc)$ defined in~\eqref{eq:deltaRangeEst} peaks near $\by$. Then, the 
components of the estimated snapshots in equation~\eqref{eq:estg}, for points $\by \in \Omega$ satisfying $|\by -\bx_s| > c_ot_F$, are approximated by 
\begin{equation}
\tensor{\bg}^{\rm est}(t,\bx_s,\by;\tilde \bc) \approx c_o^{-1} \check{f}'(t) \star_t 
\int_{\Omega} d \bxi \, {\bG}(t,\bx_s,\bxi) \bc^{-1}(\bxi) {\bde}^{\rm est}(\bxi,\by;\tilde \bc).
\label{eq:gvsdyad}
\end{equation}
\end{proposition}

\section{Imaging}
\label{sect:IMAGE}
The goal of imaging is to obtain a qualitative picture of the heterogeneous medium, which  localizes 
reflective structures, modeled by the ``rough features" of $\beps$ and thus $\bc$, like jump discontinuities~\cite{symes2008migration}. The smooth part of the medium, which determines the kinematics of wave propagation, is assumed known. In this section we model this smooth part by the constant and isotropic reference wave speed $
\tilde \bc = c_o \tensor{\bI}$.
 This is typical in applications like radar imaging or nondestructive evaluation
\cite{cheney2009fundamentals,curlander1991synthetic}, although  in seismic imaging the smooth part is variable
\cite{biondi20063d,symes2008migration}. 

The most popular imaging method is reverse time migration~\cite{biondi20063d}, also known as backprojection~\cite{cheney2009fundamentals}. We call it henceforth the ``traditional method" and its imaging function is defined by 
\begin{equation}
\hspace{-0.05in}\cI^\RTM(\by) = \hspace{-0.05in} \sum_{s,s'=1}^m \sum_{p,p'=1}^2 \hspace{-0.05in}\be_{p'}^T 
\bG(-t,\bx_{s'},\by;\tilde \bc) \star_t \bG(-t,\by,\bx_{s};\tilde \bc)  \be_p \star_t \mathcal{W}^{(s',p'),(s,p)}(t) 
\Big|_{t=0}.
\label{eq:defIRTMpp}
\end{equation}
This definition  is based on the linearization of the forward mapping about the reference medium. The linearization, known as the Born approximation, 
assumes that the waves propagate from the source at $\bx_s$ to a point  $\by \in \Omega$, where they scatter once, and 
then return to the array, at the receiver location $\bx_{s'}$. The image is formed by back-propagating the recorded waves 
to search points $\by$ in the imaging domain $\Omega_{\rm im} \subset \Omega$.
The backpropagation uses 
the time reversibility of the wave equation and it is modeled in~\eqref{eq:defIRTMpp} by the convolution with the 
time reversed Green's dyadic in the reference medium.  Traditional imaging works well if the medium is a slight perturbation of the reference one. This holds
when the reflectivity is weak and/or the support of the unknown features is small. If this is true, then the imaging functions
\eqref{eq:defIRTMpp} peaks near the reflective structures. If the reference medium has the wrong kinematics, the 
images are unfocused and if the reflectivity is strong, the images display ghost features, which are multiple scattering artifacts.

We introduce next an imaging method that uses the estimated internal wave~\eqref{eq:EstWave}. As we have discussed 
in the previous section, this wave contains all the scattering events recorded at the array. However, these events are mapped 
to wrong locations in $\Omega$ if the kinematics is strongly perturbed. Thus, it is not surprising 
that the imaging is successful when the reference medium gives a good approximation of the kinematics. The interesting property of the new imaging function is that it mitigates multiple scattering artifacts. 
\subsection{Imaging with the estimated internal wave}
\label{sect:normg}
Similar to equation~\eqref{eq:Sn3}, let us write the estimated internal wave~\eqref{eq:EstWave} in terms of its components
\begin{equation}
\bu_j^{\rm est}(\by;\tilde \bc) = \left(\bu_j^{{\rm est} (1,1)}(\by;\tilde \bc), 
\bu_j^{{\rm est} (1,2)}(\by;\tilde \bc), \ldots, \bu_j^{{\rm est} (m,2)}(\by;\tilde \bc)\right).
\label{eq:decEstWave}
\end{equation}
The imaging function for incident polarization along the direction $\be_p$ and measurements along $\be_{p'}$ is 
\begin{equation}
\cI^{(p',p)}(\by) = \sum_{j=0}^{n-1} \sum_{s = 1}^m \left| \be_{p'}^T \bu_j^{{\rm est} (s,p)}(\by;\tilde \bc) \right|^2, \qquad 
\by \in \Omega_{\rm im} \subset \Omega.
\label{eq:Jgpp}
\end{equation}
This is an extension of the imaging function introduced in~\cite{borcea2021reduced} for  acoustic waves. As we show next, its expression is connected to the time reversal point spread function~\cite{fink99}, which describes mathematically the following experiment: Suppose that there is a point source at $\by \in \Omega$, which emits a wave. If we record this wave at the array, time reverse the recordings and re-emit the wave back into the medium, then the wave will focus near $\by$, due to the time reversibility of the wave equation. The sharpness of the focusing depends on the frequency content of the wave, the aperture of the array and the medium. The striking experiments in~\cite{fink99} show that the more scattering in the medium, the better the focusing. 

The analysis below shows that our imaging function~\eqref{eq:Jgpp} is determined by the time reversed point spread function averaged around $\by$, in the support of $\bde^{\rm est}$ defined in~\eqref{eq:deltaRangeEst}. Thus, it is designed to be sensitive to changes in the medium around $\by$.

\subsubsection{Connection to the time reversal experiment}
\label{sect:AnalNormg}
To analyze~\eqref{eq:Jgpp}, we recall the result~\eqref{eq:estg} in Lemma \ref{lem.2} and approximate the 
sum over $j$ by a time integral
\begin{equation}
\cI^{(p',p)}(\by) \approx \frac{1}{\tau} \int_0^{n \tau} d t \sum_{s=1}^m \left| \be_{p'}^T \bg^{\rm est}(t,\bx_s,\by) \be_p \right|^2.
\label{eq:TR1}
\end{equation}
Now use the result~\eqref{eq:gvsdyad} of Proposition \ref{prop.1} to write 
\begin{equation}
\cI^{(p',p)}(\by) \approx \int_{\Omega} d \bxi \int_{\Omega} d \bxi' \, \left[\bet^{(p)}(\bxi,\by;\tilde \bc)\right]^T 
\tensor{\bK}^{(p')}(\bxi,\bxi')  \bet^{(p)}(\bxi',\by;\tilde \bc),
\label{eq:TR2}
\end{equation}
in terms of the vector 
\begin{equation}
\bet^{(p)}(\bxi,\by;\tilde \bc) =  \bc^{-1}(\bxi) \bde^{\rm est}(\bxi,\by; \tilde \bc) \be_p,
\label{eq:TR3}
\end{equation}
and the matrix $\tensor{\bK}^{(p')} = \left(K^{(p')}_{j,j'}\right)_{j,j' = 1,2}, $ with entries 
\begin{align}
K^{(p')}_{j,j'}(\bxi,\bxi') &= \sum_{s = 1}^m \int_0^{n \tau}  \frac{dt}{c_o^2 \tau} \, 
\check{f}'(t) \star_t G_{p',j}(t,\bx_s,\bxi) \check{f}'(t) \star_t  G_{p',j'}(t,\bx_s,\bxi') \nonumber \\
&= \sum_{s = 1}^m \int_0^{n \tau} \frac{dt}{c_o^2 \tau} \,
\check{f}'(t) \star_t G_{j,p'}(t,\bxi,\bx_s) \check{f}'(t) \star_t  G_{p',j'}(t,\bx_s,\bxi').
\label{eq:TR4}
\end{align}
The second equality is due to the reciprocity of Green's dyadic~\cite{tai1992complementary,strunc2007constitutive}
\begin{equation}
\bG(t,\bx_s,\bxi) = \bG^T(t,\bxi,\bx_s).
\end{equation}

We relate next the expression~\eqref{eq:TR2} to the time reversal experiment, using the following two steps:

\vspace{0.05in}
\noindent \textbf{Step 1:} Let $\bpsi^{(p)}$ be the wave field satisfying the wave equation
\begin{align*}
-\nabla^\perp \left[ \nabla^\perp \cdot \bpsi^{(p)}(t,\bx,\by) \right] + \bc^{-2}(\bx) \partial_t^2 \bpsi^{(p)}(t,\bx,\by) &=\check{f}'(t) \bet^{(p)}(\bx,\by;\tilde \bc), ~ t \in \RR, ~ \bx \in \Omega, \\
\bpsi^{(p)}(t,\bx,\by) &= {\bf 0}, \quad t \ll 0, ~ \bx \in \Omega, \\
\bn^\perp(\bx) \cdot \bpsi^{(p)}(t,\bx,\by) &= 0 \quad t \in \RR, ~ \bx \in \partial \Omega.
\end{align*}
This wave can be written in terms of Green's dyadic $\bG$, the solution of~\eqref{eq:GF3}-\eqref{eq:GF4.1}, using linear superposition
\begin{equation*}
\bpsi^{(p)}(t,\bx,\by) = \check{f}'(t) \star_t  \int_{\Omega} d \bxi' \, \bG(t,\bx,\bxi') \bet^{(p)}(\bxi',\by;\tilde \bc).
\end{equation*}
We are interested in the $p'$ component of this wave, evaluated at $\bx = \bx_s$.  Substituting it in
equations ~\eqref{eq:TR2} and~\eqref{eq:TR4}, we get 
\begin{align*}
 \hspace{-0.08in} \cI^{(p',p)}(\by) \approx \int_{\Omega} d \bxi  \left[\bet^{(p)}(\bxi,\by;\tilde \bc)\right]^T \hspace{-0.05in} \int_0^{n \tau}  \hspace{-0.05in} \frac{d t}{c_o^2 \tau} 
\check{f}'(t) \star_t \sum_{s=1}^m \bG(t,\bxi,\bx_s) \be_{p'}  \left[\be_{p'}^T \bpsi^{(p)}(t,\bx_s,\by)\right].
\end{align*}

\vspace{0.05in}
\noindent \textbf{Step 2:} Now change variables in the time integral from $t$ to $-t$ and recall that $\check{f}$ is even, 
which means that $\check{f}'$ is odd. Write also explicitly the convolution to get
\begin{align}
\cI^{(p',p)}(\by) \approx -\frac{1}{c_o^2 \tau} \int_{\Omega} d \bxi \left[\bet^{(p)}(\bxi,\by;\tilde \bc)\right]^T  \bbeta^{(p',p)}(\bxi,\by),
\label{eq:TR6}
\end{align}
in terms of the vector field
\begin{equation}
\bbeta^{(p',p)}(\bxi,\by) =  \check{f'}(t') \star_t \int_{-n \tau}^0 dt \sum_{s=1}^m \bG(t'-t,\bxi,\bx_s) \be_{p'} \left[\be_{p'}^T \bpsi^{(p)}(-t,\bx_s,\by)\right] \Big|_{t'=0}.
\label{eq:TR7}
\end{equation}
The right hand side in~\eqref{eq:TR7} is obtained by emitting from all the antennas, over the time interval $t \in [0,n\tau]$, the time reversed wave $\be_{p'}^T  \bpsi^{(p)}$, polarized along $\be_{p'}$.  The result is then convolved with $\check f'$ and it is evaluated at 
$t = 0$, when we expect from the time reversibility of the wave equation the refocusing to occur, at points $\bxi$ near $\by$~\cite{bal2003time}.

\subsubsection{Expected focusing of the imaging function}
\label{sect:Expected}
Recall definition~\eqref{eq:TR3} of  $\bet^{(p)}$ and write the expression~\eqref{eq:TR6}  more explicitly
\begin{equation}
\cI^{(p',p)}(\by) \approx -\frac{1}{c_o^2 \tau} \int_{\Omega} d \bxi \, 
\be_p^T \left[  \bde^{\rm est}(\bxi,\by; \tilde \bc)\right]^T \bc^{-1}(\bxi) \bbeta^{(p',p)}(\bxi,\by).
\label{eq:TR8}
\end{equation}
We know that the time reversal experiment gives a field $\bbeta^{(p,p')}$ that is peaked at $\bxi \approx \by$, 
with resolution that depends on the signal $\check f$, the aperture of the array and the medium through which the waves propagate
\cite{bal2003time}.
The shorter the signal and the larger the aperture, the better the refocusing. Our imaging function averages 
$\bbeta^{(p,p')}$ in the support of $ \bde^{\rm est}(\cdot, \by;\tilde \bc)\be_p$ and it is sensitive to the changes of 
the wave speed $\bc$ there. Thus, as long as the components in the $p^{\rm th}$ column of  $\bde^{\rm est}(\cdot, \by;\tilde \bc)$
are peaked near $\by$, the imaging function computed from the internal wave as in  equation~\eqref{eq:Jgpp} gives good estimates 
of the support of the reflectivity of the medium.

\subsection{Numerical results}
We refer to Appendix \ref{ap:num} for the setup of the numerical simulations and some comments on the computational cost. The simulations are carried out in a rectangular domain 
with perfectly conducting boundary. The array of antennas  is near the top boundary and the system of coordinates is $\bx = (x_1,x_2)$, with the ``range coordinate" $x_1$ pointing downwards and the ``cross-range" coordinate $x_2$ orthogonal to it. 
The length unit in the plots is $\lambda_c$, the wavelength calculated at the highest frequency $\om_c$ in the bandwidth of the probing signal $f(t)$ (Appendix \ref{ap:num}).

We show imaging results for isotropic and anisotropic media with crack like features, where the kinematics is slightly perturbed, but the reflectivity is strong enough to cause visible multiple scattering effects in the traditional imaging function~\eqref{eq:defIRTMpp}.
To explore the role of polarization in the measurements, we display the imaging function~\eqref{eq:Jgpp} for  $p,p' = 1,2$, instead 
of showing the sum over all polarizations. Note that we plot the range derivatives of the imaging functions, in order to emphasize the
large changes which are related to the jump discontinuities of the wave speed. 
The array aperture, antenna separation and time step $\tau$ are given in the captions.

\subsubsection{Imaging in isotropic media}
The first example considers an isotropic medium, with dielectric permittivity displayed in the left plot of Fig. \ref{fig:Crack1}. 
The traditional imaging function ~\eqref{eq:defIRTMpp} is shown in the middle plot. It does show the 
crack, but it has ghost features below it,  induced by multiple scattering. The ideal image, the analogue of $\cI^{(2,2)}$, defined for the uncomputable true internal wave, is shown for reference in the right plot. It has no ghost features and the resolution is excellent,
as expected from the expression~\eqref{eq:TR8} with $\bde^{\rm est}$ replaced by $ \bde$. 

\begin{figure}[t]
\centering
\hspace{-0.13in}\includegraphics[width = 0.33\textwidth]{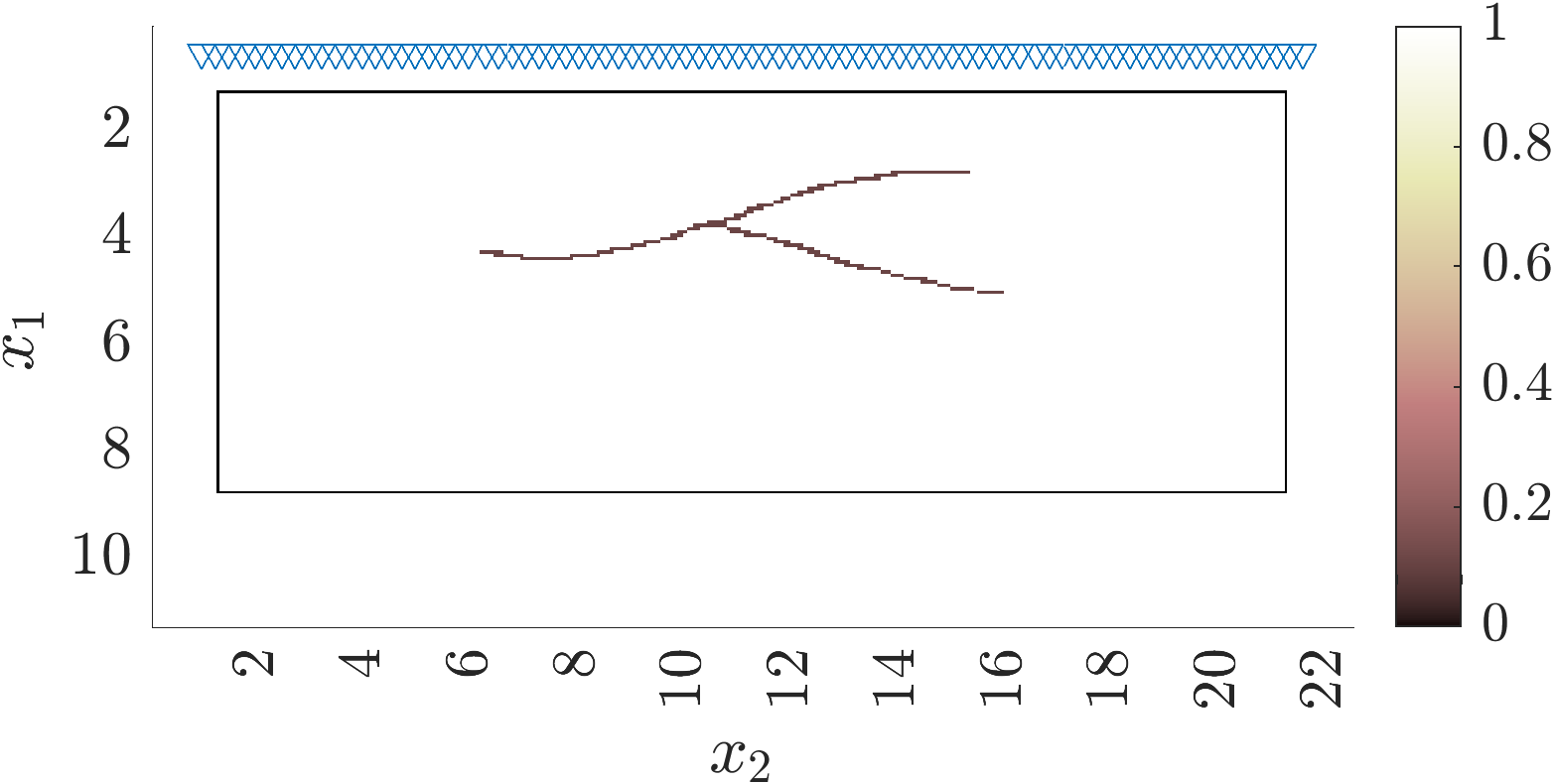}
\hspace{-0.05in}
\includegraphics[width = 0.33\textwidth]{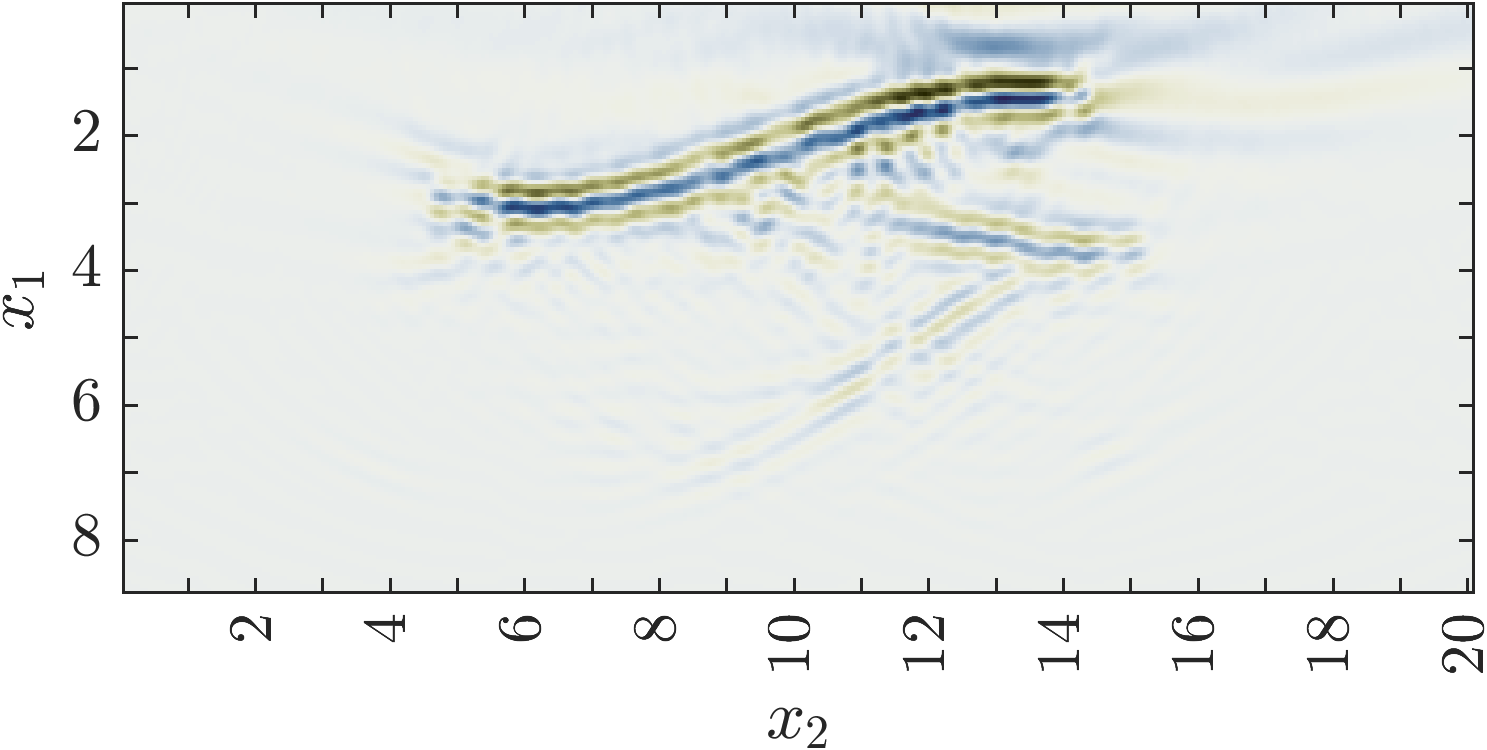}
\hspace{-0.07in}
\includegraphics[width = 0.33\textwidth]{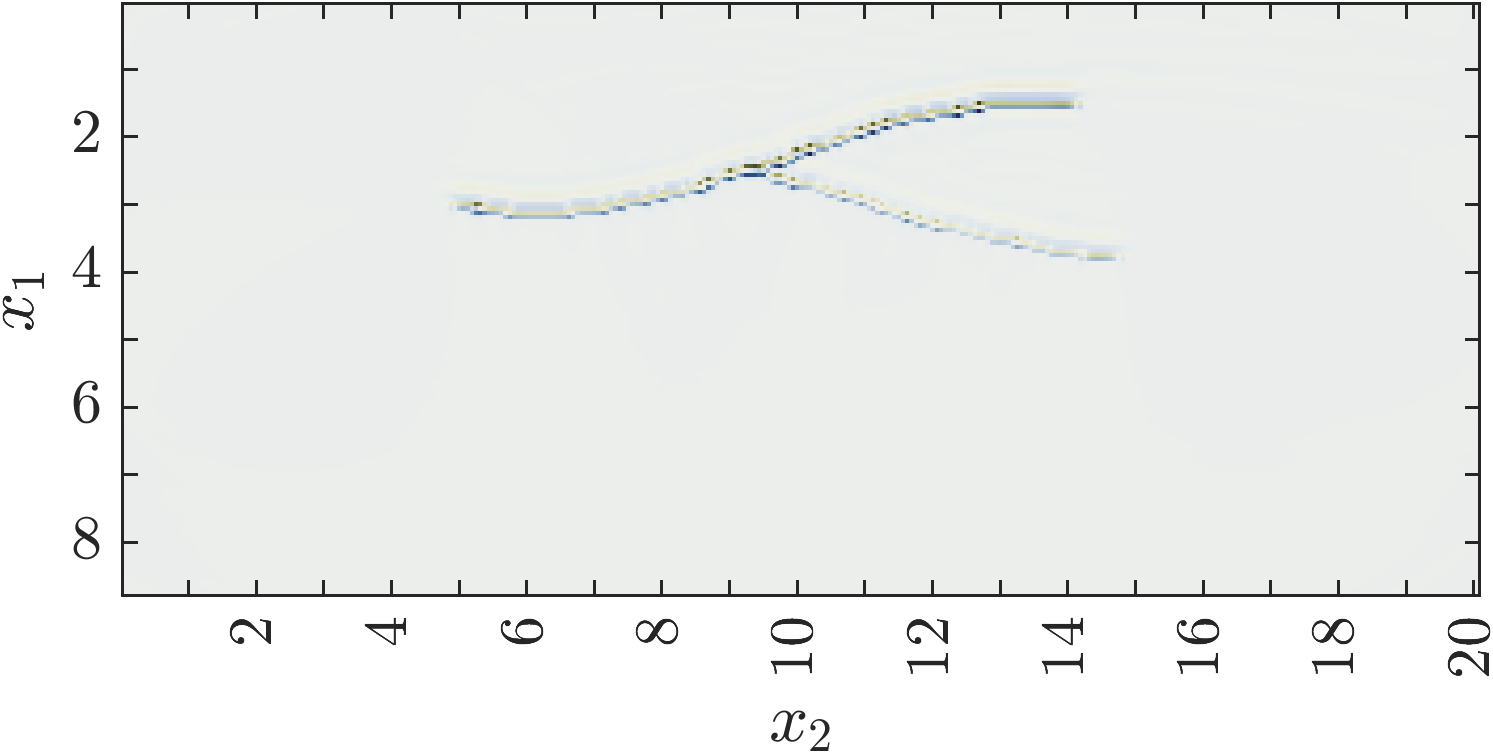}
\vspace{-0.1in}
\caption{Left: Display of the isotropic medium with a crack feature. The medium is modeled by the dielectric permittivity 
$\eps(\bx) \underline{\underline{\bI}}$ and the colorbar shows the contrast $\eps(\bx)/\eps_o$. The antennas are drawn as  triangles near the top boundary and the imaging domain $\Omega_{\rm im }$ is inside the black rectangle. Middle: The traditional imaging 
function~\eqref{eq:defIRTMpp}. Right: The ideal imaging function. The axes are in units of $\lambda_c$. The array aperture is $20 \lambda_c$ and there are 
$80$ antennas. The time step is $\tau = 0.3 \pi/w_c$.}
\label{fig:Crack1}
\end{figure}

\begin{figure}[h]
\centering
\begin{tabular}{cc}
$\cI^{(1,1)}$ & $\cI^{(2,2)}$ \\
\includegraphics[width = 0.35\textwidth]{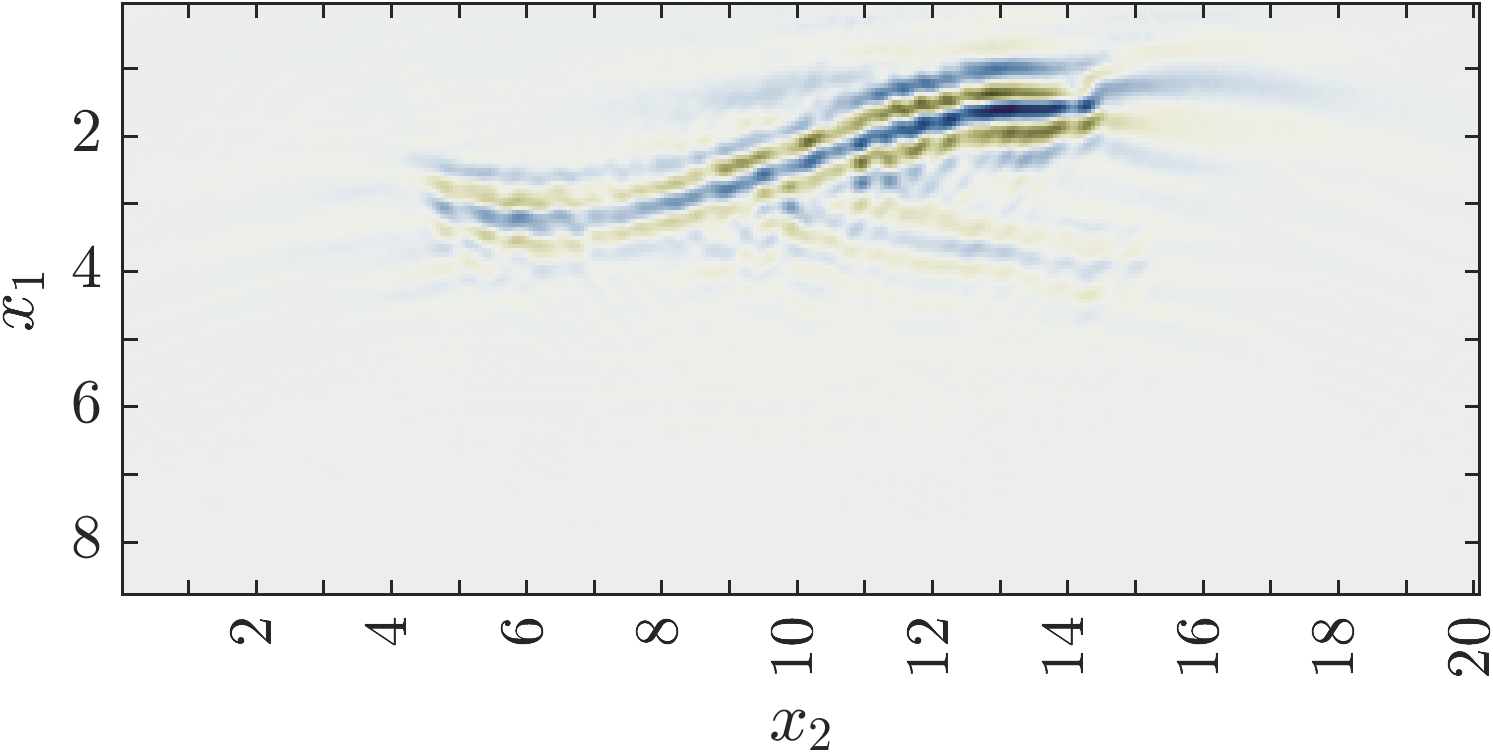} &\includegraphics[width = 0.35\textwidth]{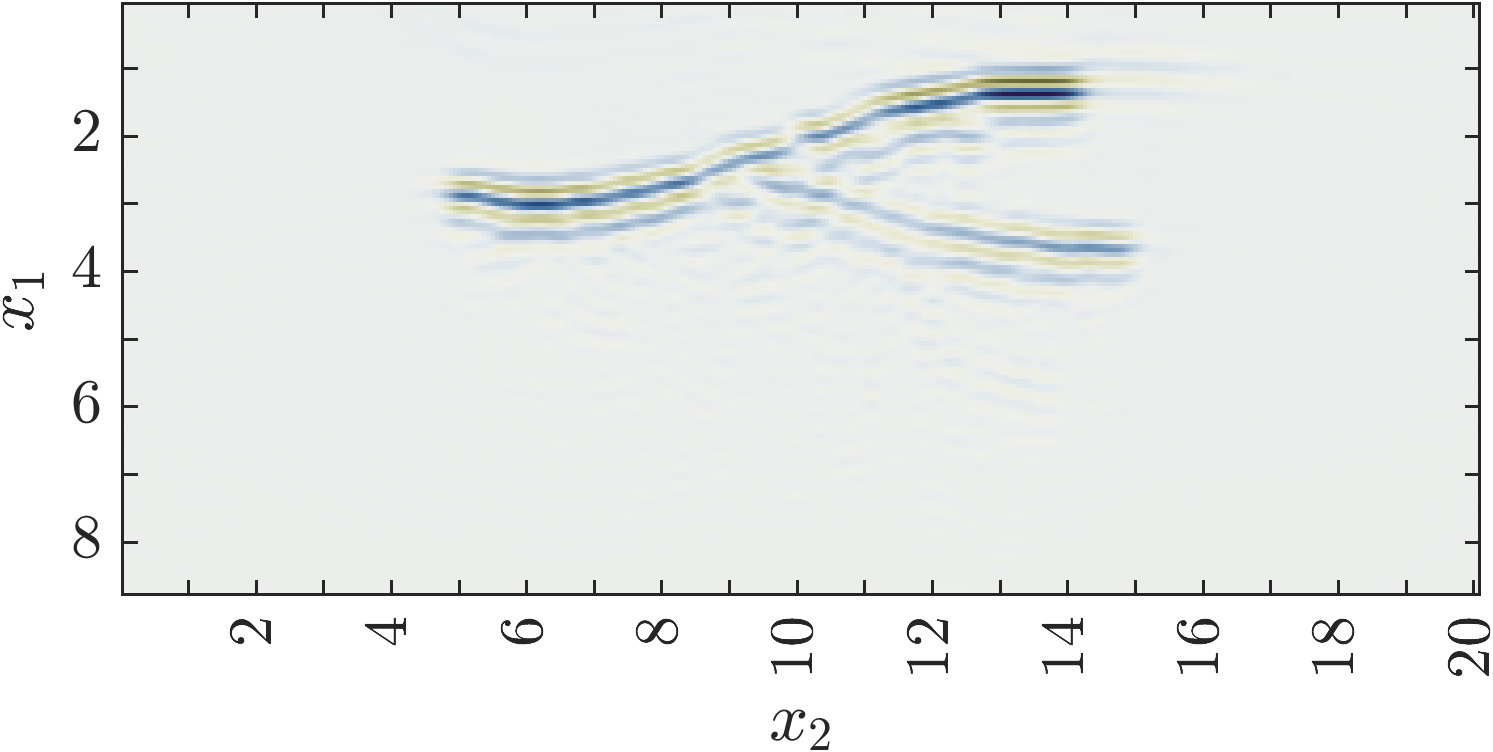} 
\\
$\cI^{(1,2)}$ & $\cI^{(2,1)}$\\
\includegraphics[width = 0.35\textwidth]{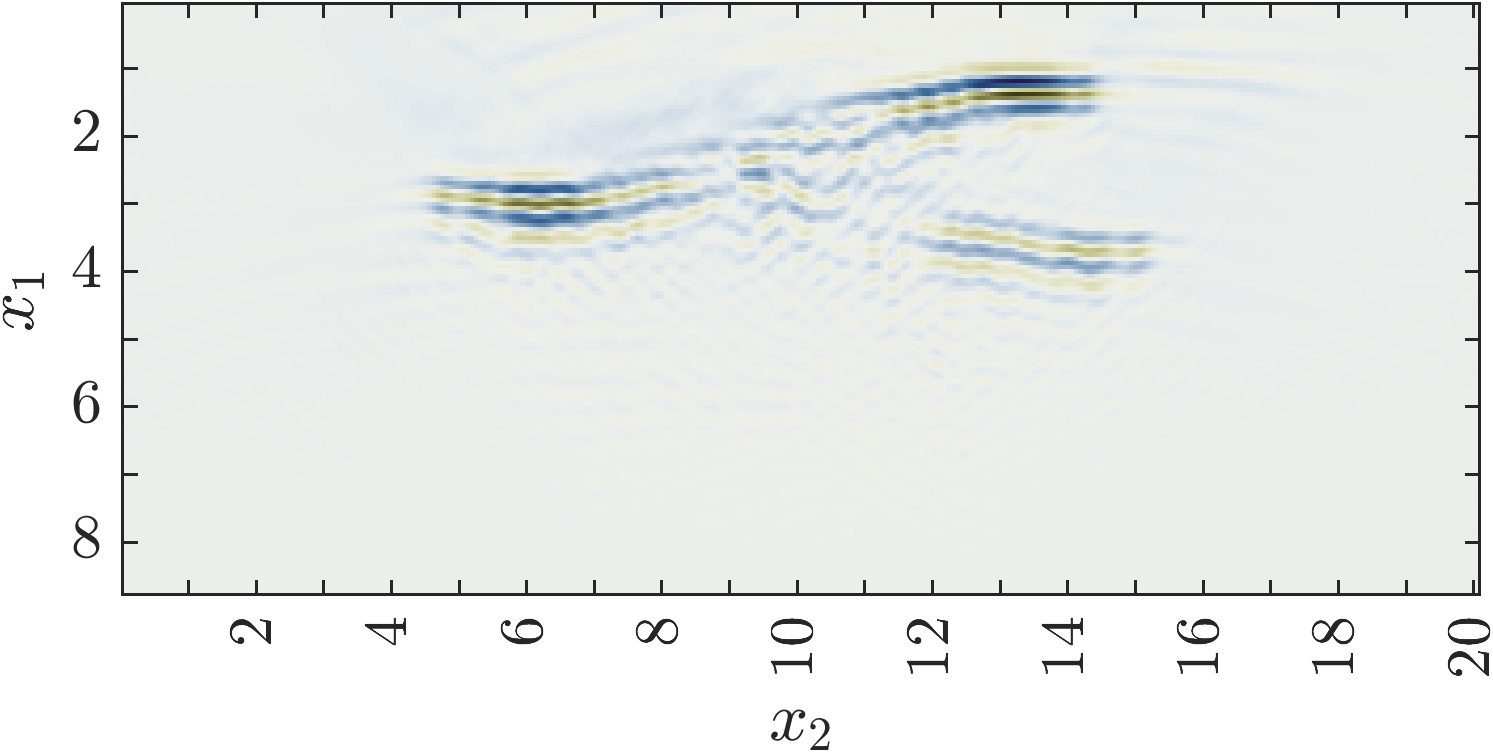}&\includegraphics[width = 0.35\textwidth]{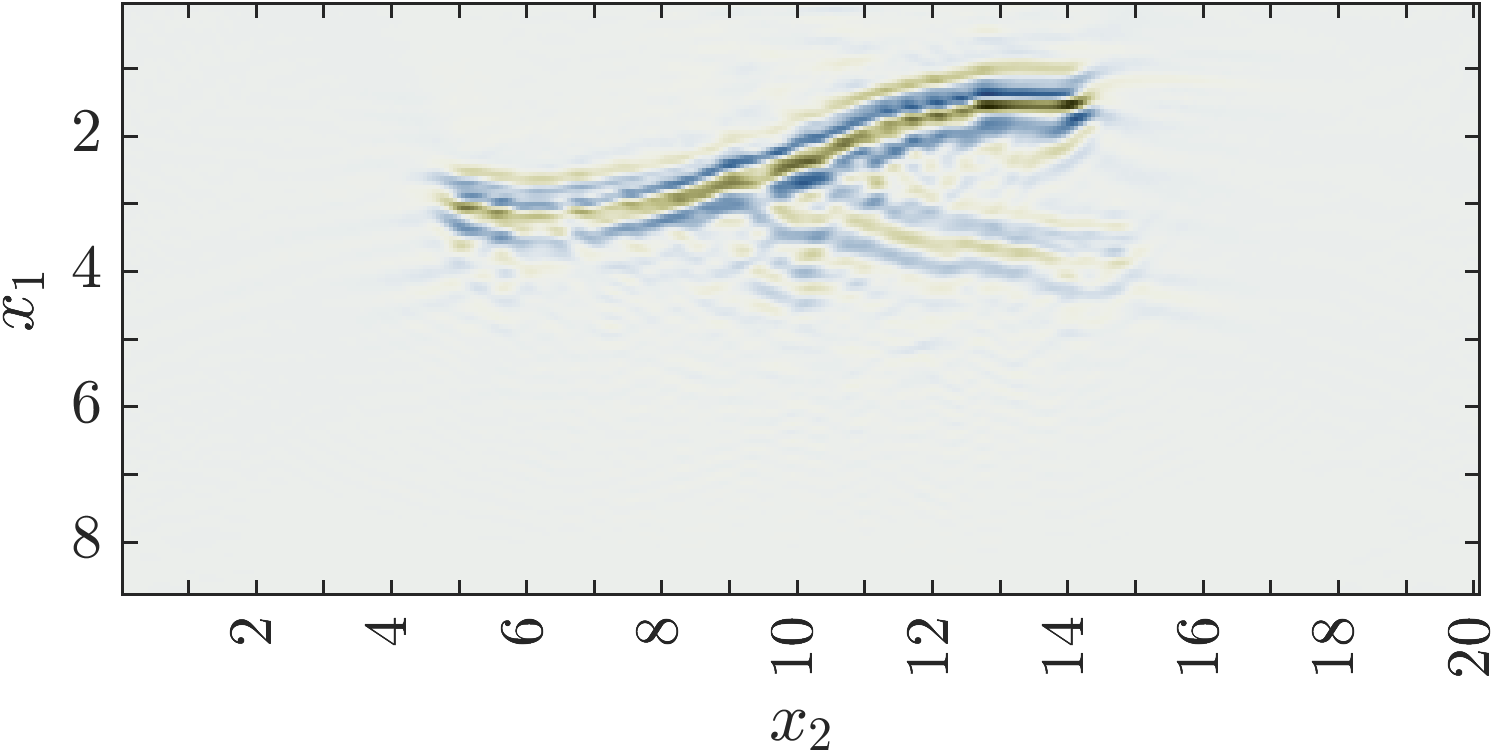} 
\end{tabular}
\vspace{-0.1in}
\caption{Imaging function~\eqref{eq:Jgpp}. The setup and axes are as in Fig. \ref{fig:Crack2}}
\label{fig:Crack2}
\end{figure}

The computable ROM based imaging functions~\eqref{eq:Jgpp} are displayed in Fig. \ref{fig:Crack2} for all $p,p' = 1,2$. 
The images identify the crack  better than the traditional approach (top middle plot in Fig. \ref{fig:Crack1}). The best image is $\cI^{(2,2)}$. This is expected  because we have a transversal  wave that propagates mainly along the range direction $\be_1$. The dominant component of this wave  is therefore along the cross-range axis  $\be_2$.  

We illustrate in Fig. \ref{fig:Crack3} the effect of $\tau$, the array aperture and sensor separation on the imaging function 
$\cI^{(2,2)}$. For each plot we perturb only one parameter about the reference values used in Fig. \ref{fig:Crack1}-\ref{fig:Crack2}:
$\tau = 0.3\pi/\om_c$, aperture $20 \la_c$ and sensor separation $\la_c/4$. Note how under sampling in time and space causes unwanted
ripples in the image. The aperture size is known to affect the cross-range resolution of traditional images. For our imaging function
the reduced aperture leads to poor imaging of the ends of the crack.

\begin{figure}[h]
\centering
\hspace{-0.15in}\includegraphics[width = 0.33\textwidth]{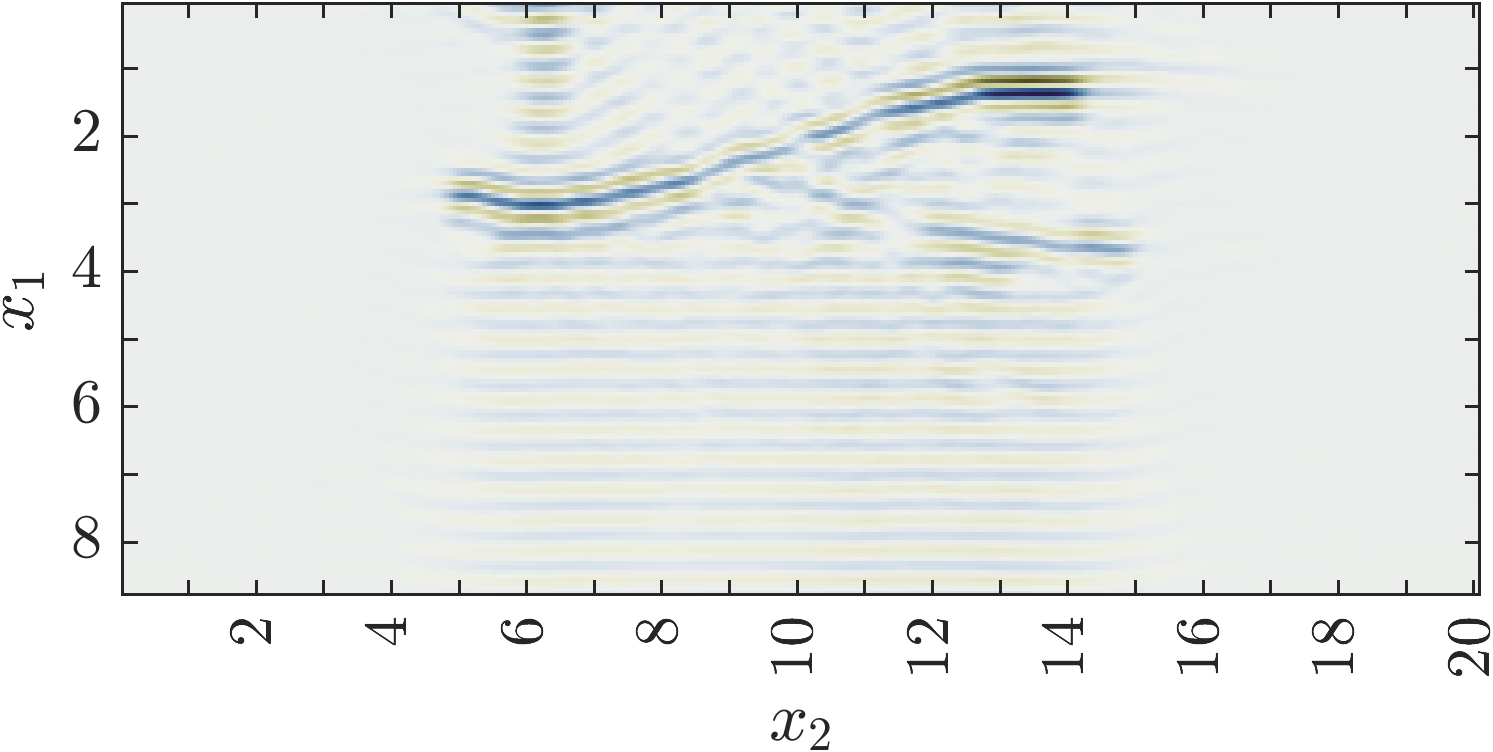}
\hspace{-0.05in}
\includegraphics[width = 0.33\textwidth]{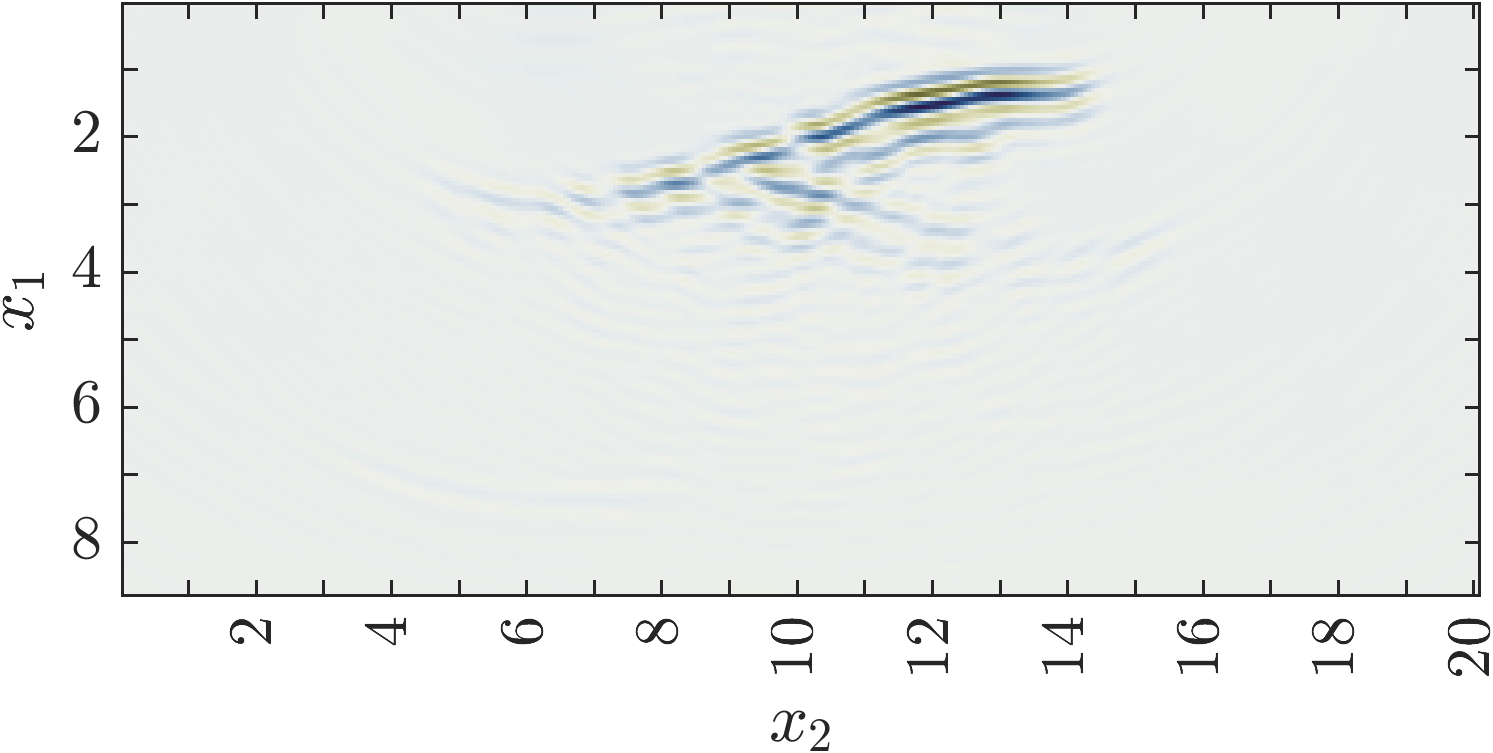}
\hspace{-0.05in}
\includegraphics[width = 0.33\textwidth]{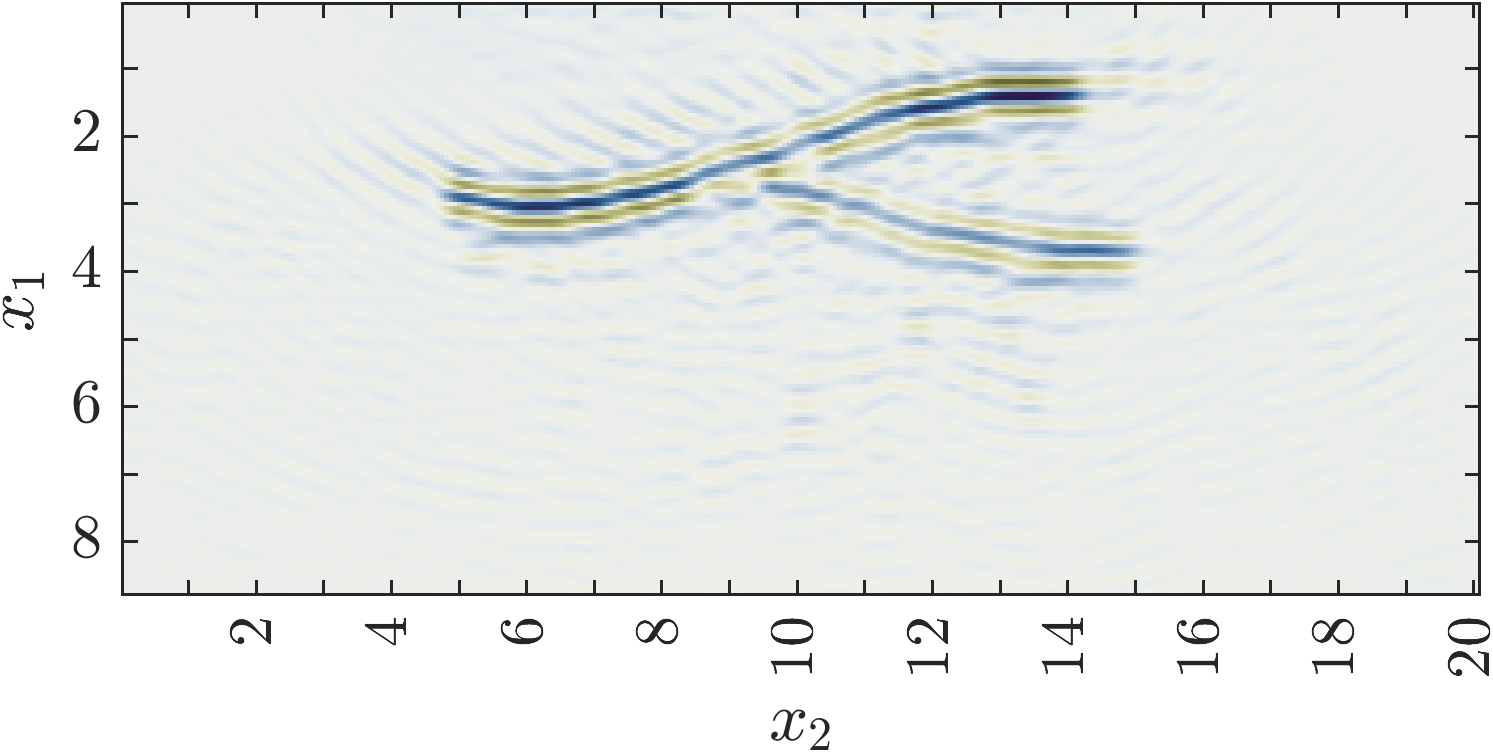}
\vspace{-0.1in}
\caption{Imaging function $\cI^{(2,2)}$ for the following perturbed parameters:  $\tau = 0.9 \pi/\om_c$ (left), array aperture  $5 \la_c$ (middle)
and sensor separation $\la_c$ (right). Compare these with the top right plot in Fig. \ref{fig:Crack2}.}
\label{fig:Crack3}
\end{figure}

The next example, shown in Fig. \ref{fig:Crack4} is for a more complicated medium, with multiple cracks. Note the ghost feature 
induced by the multiple scattering in the traditional image (middle plot) that makes it difficult to tell if there is a single long crack or 
two disconnected ones. Our imaging function (right plot) is clearly superior.

\begin{figure}[h]
\centering
\hspace{-0.15in}\includegraphics[width = 0.33\textwidth]{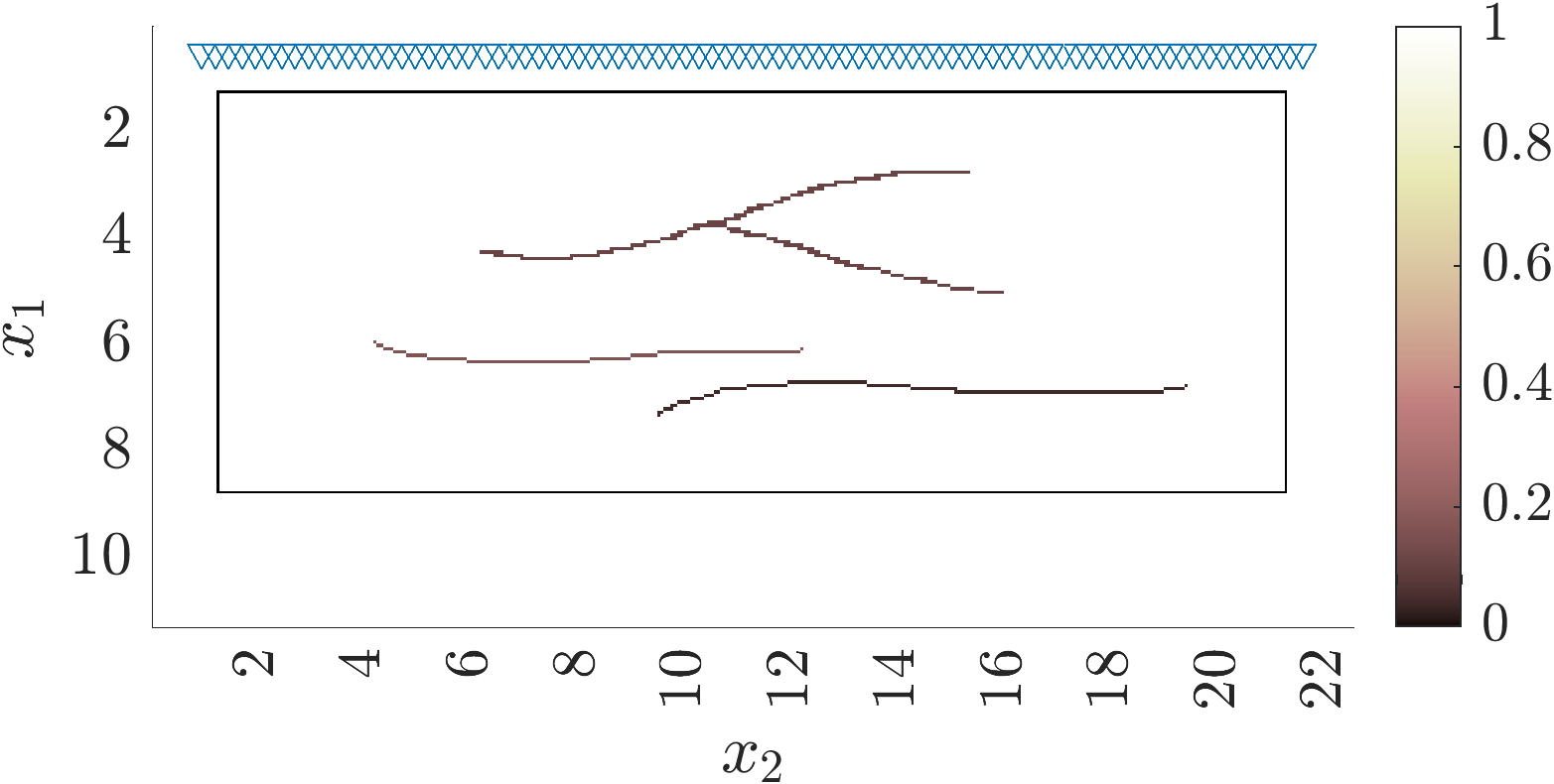}
\hspace{-0.05in}
\includegraphics[width = 0.33\textwidth]{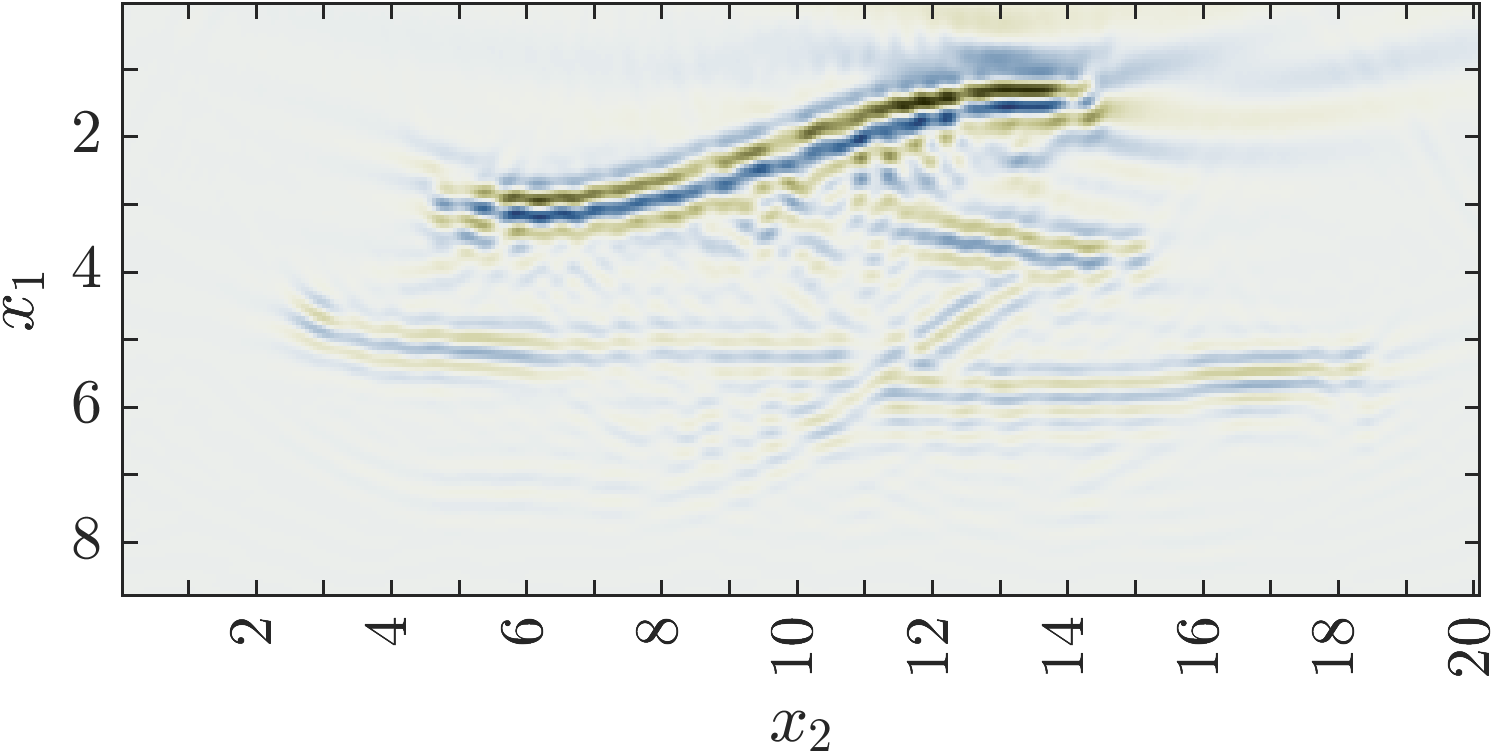}
\hspace{-0.05in}
\includegraphics[width = 0.33\textwidth]{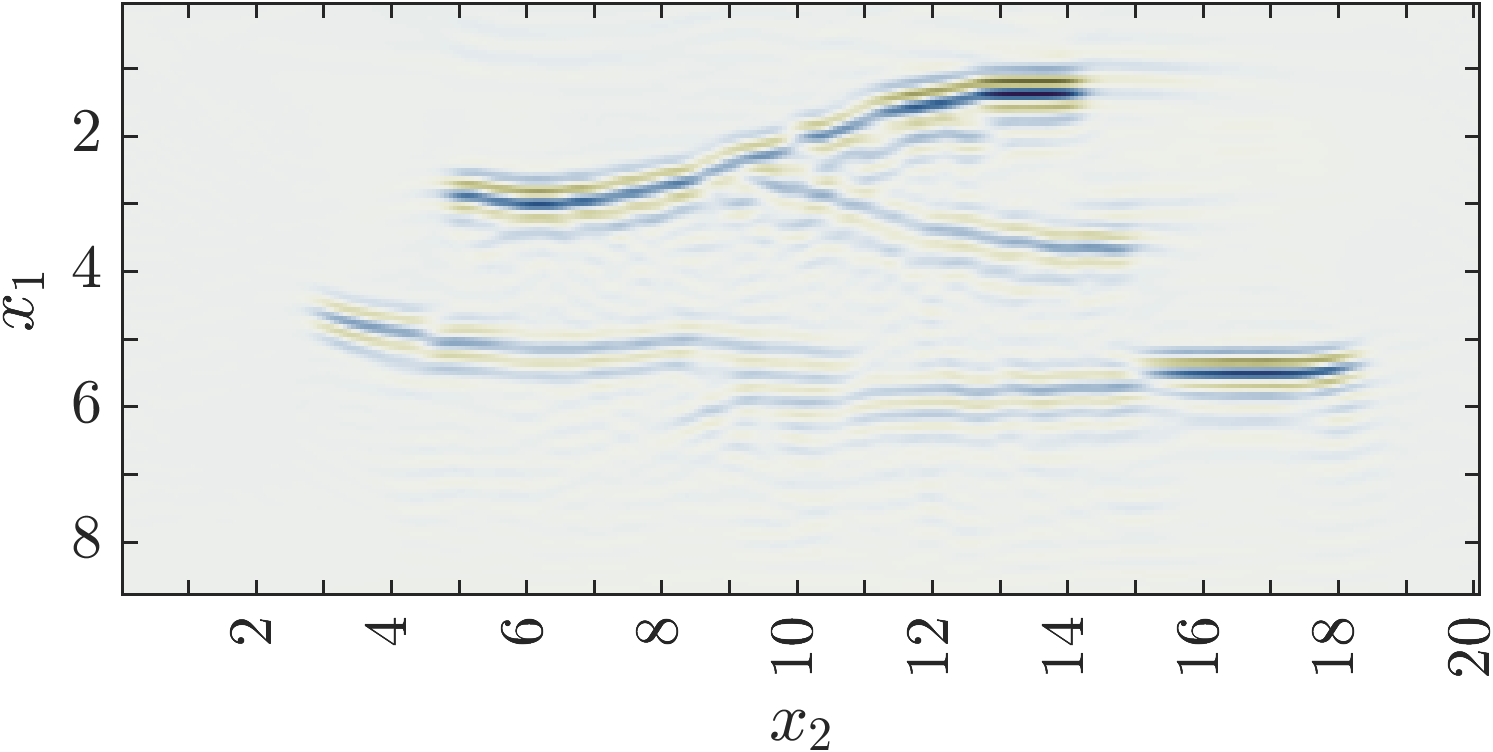}
\vspace{-0.1in}
\caption{Left: Display of the isotropic medium with multiple cracks. The medium is modeled by the dielectric permittivity 
$\eps(\bx) \underline{\underline{\bI}}$ and the colorbar shows the contrast $\eps(\bx)/\eps_o$. The antennas are drawn as  triangles near the top boundary and the imaging domain is inside the black rectangle. Middle: The traditional imaging 
function. Right: The imaging function $\cI^{(2,2)}$. The axes are in units of $\lambda_c$. The array aperture is $20 \lambda_c$ and there are $80$ antennas. The time step is $\tau = 0.3 \pi/w_c$.}
\label{fig:Crack4}
\end{figure}

\subsubsection{Imaging in anisotropic media} 
Imaging methods are not designed to give quantitative estimates of the components of $\bc$. 
The definition of the initial state ~\eqref{eq:TR3} in the time reversal experiment shows that 
even if we have a jump in a single component of $\bc$, that will be mapped to variations in $\cI^{(p,p')}$ for all $p,p' = 1,2$. 
Thus, we expect our images to indicate the location of all the jumps of $\bc$. Indeed, this is the case, as illustrated in 
Fig. \ref{fig:Crack6}, for the medium shown in Fig. \ref{fig:Crack5}. We do not display the traditional image for this case, 
because it does not bring additional insight from what is shown in the figures above.

\begin{figure}[h]
\centering
\begin{tabular}{ccc}
$c_{1,1}/c_o$ & $c_{2,2}/c_o$  & $c_{(1,2)}/c_o$\\ 
\hspace{-0.13in}\includegraphics[width = 0.32\textwidth]{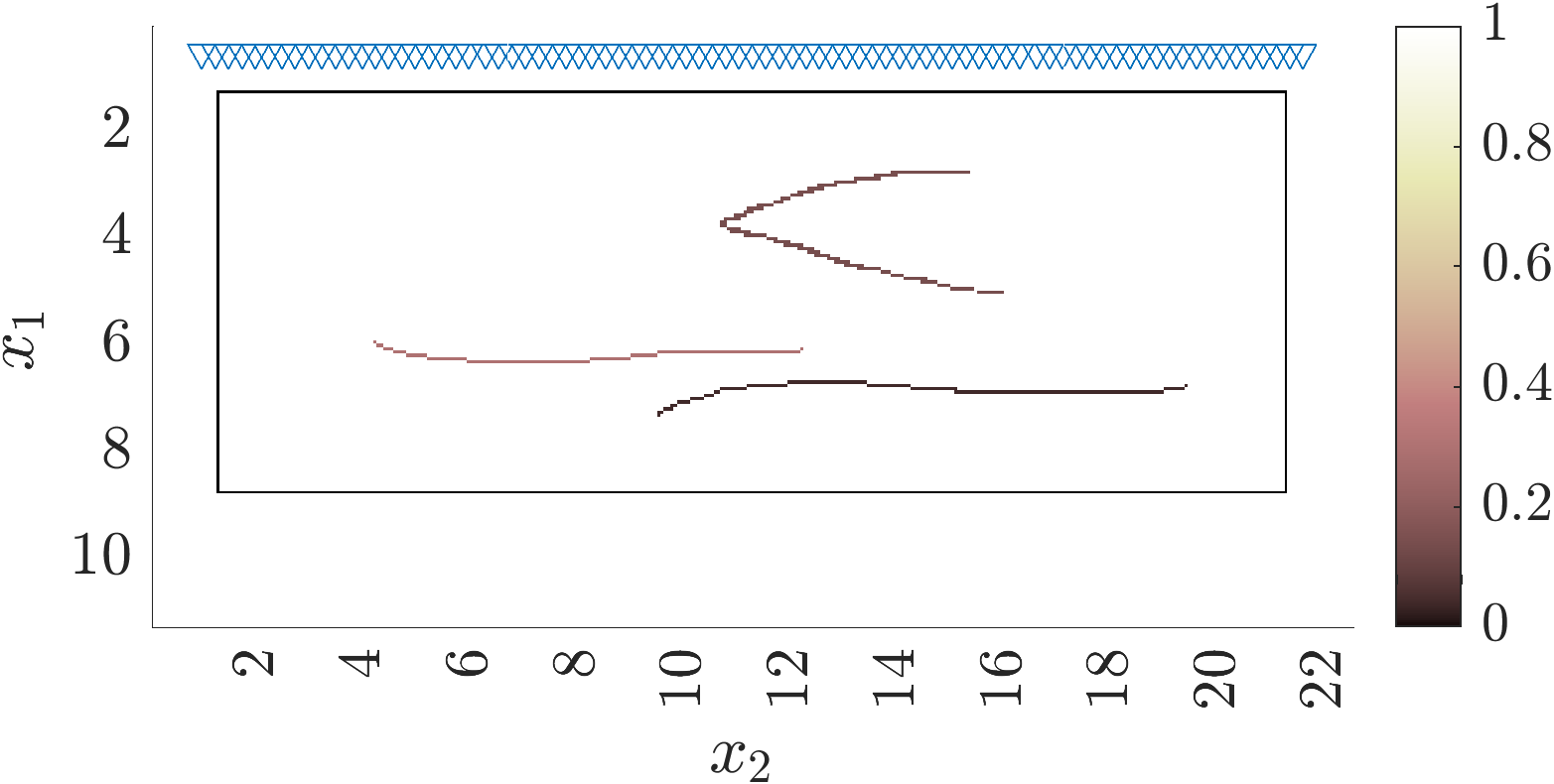}&\hspace{-0.13in} \includegraphics[width = 0.32\textwidth]{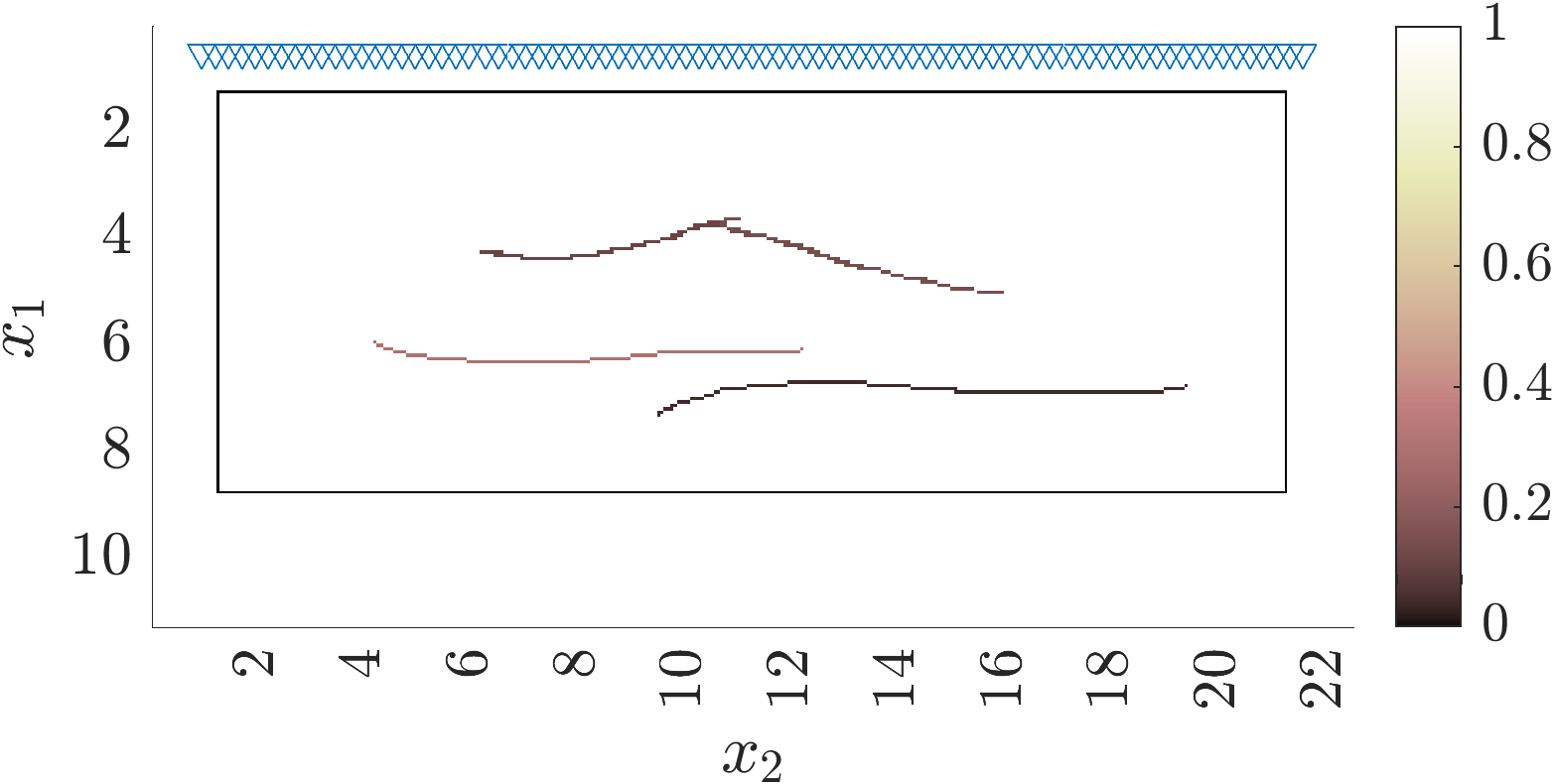} &\hspace{-0.13in}
\includegraphics[width = 0.32\textwidth]{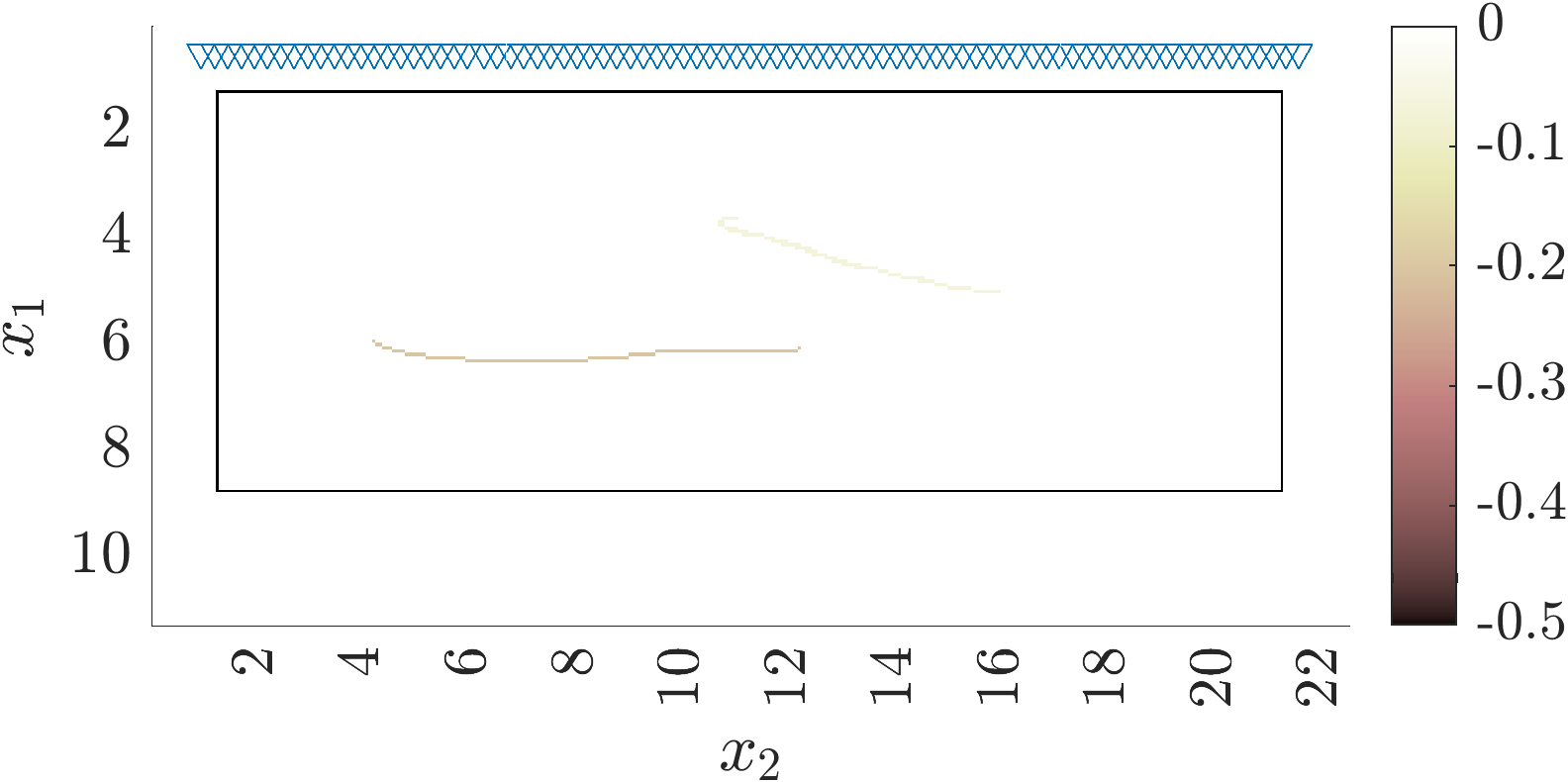}
\end{tabular}
\vspace{-0.1in}
\caption{The anisotropic medium. The colorbar is the same in all the plots.}
\label{fig:Crack5}
\end{figure}

\begin{figure}[h]
\centering
\begin{tabular}{cc}
$\cI^{(1,1)}$ & $\cI^{(2,2)}$ \\
\includegraphics[width = 0.35\textwidth]{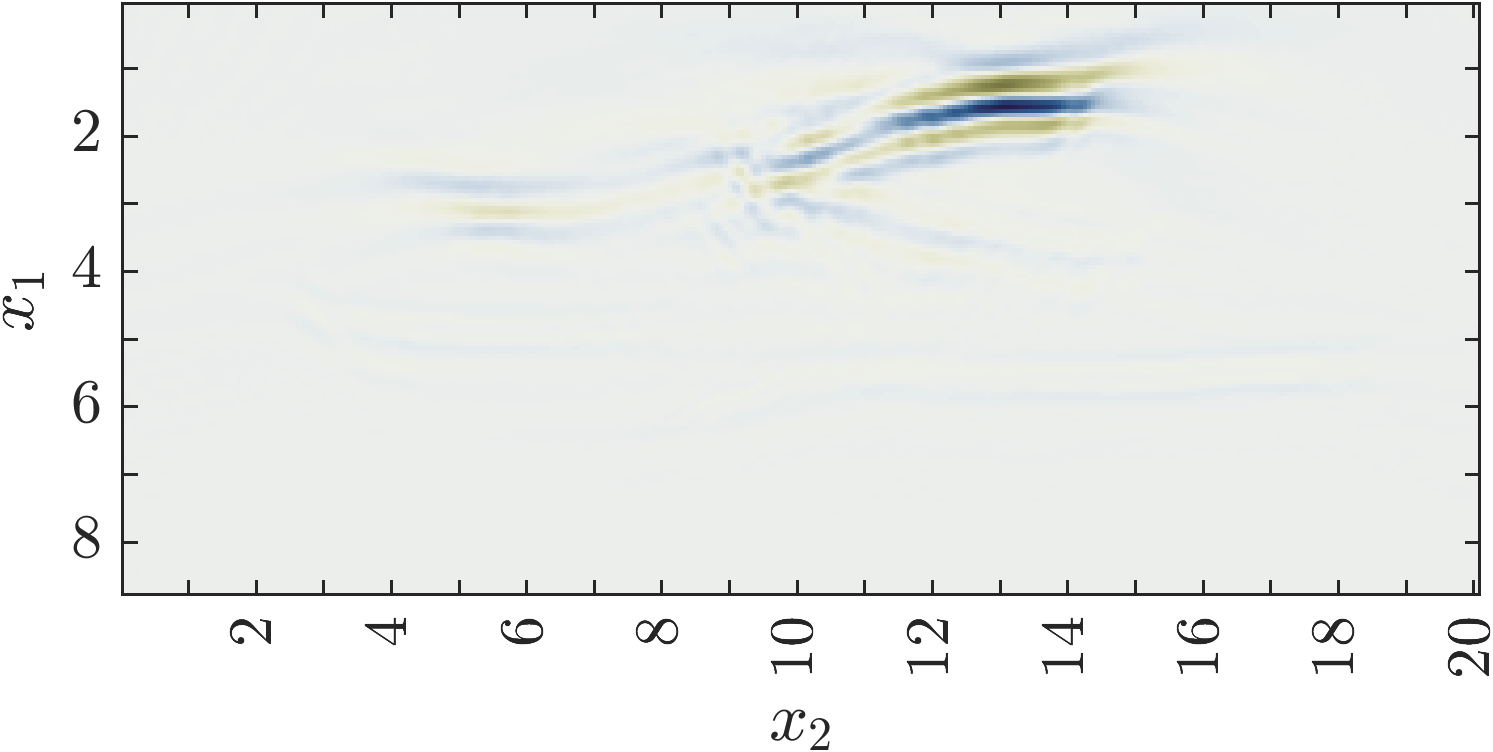} &\includegraphics[width = 0.35\textwidth]{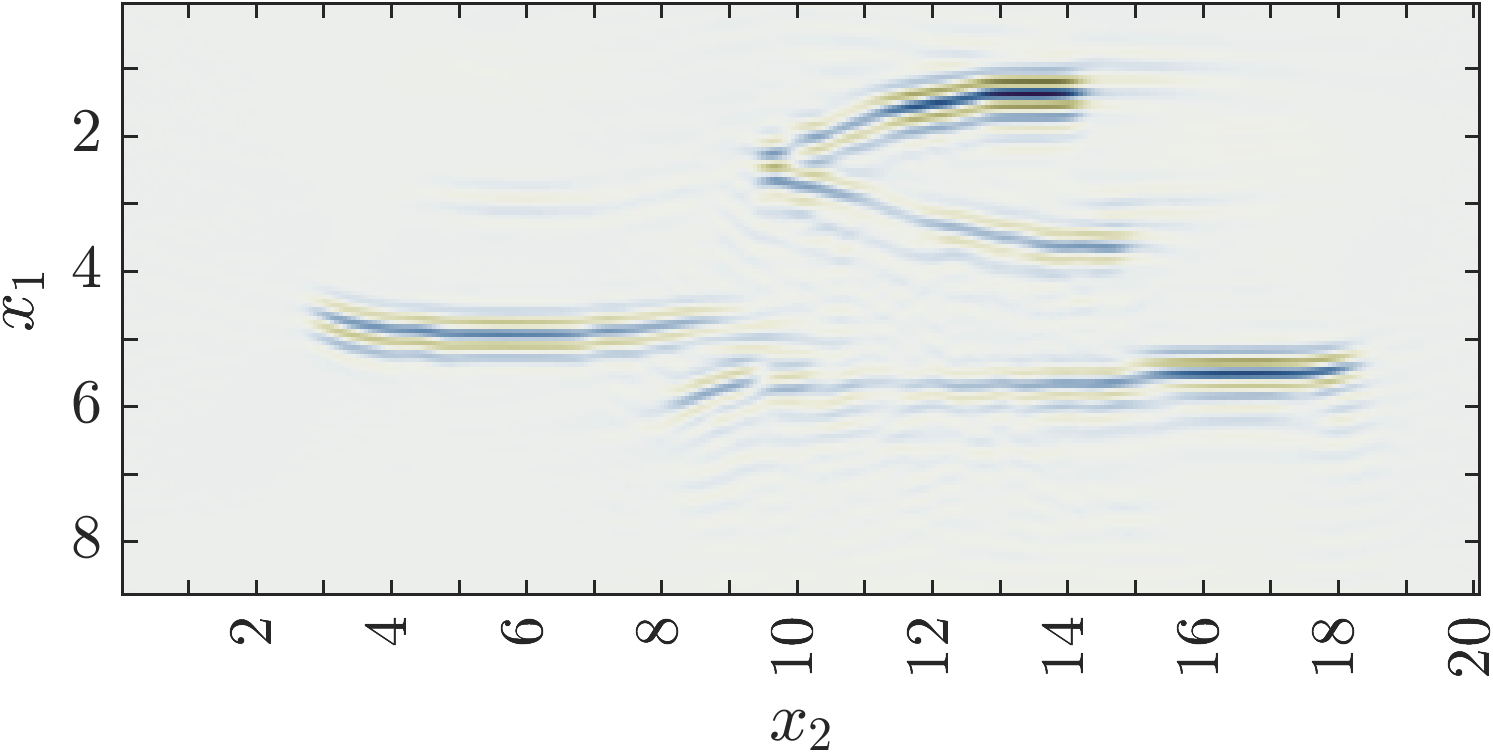} 
\\
$\cI^{(1,2)}$ & $\cI^{(2,1)}$\\
\includegraphics[width = 0.35\textwidth]{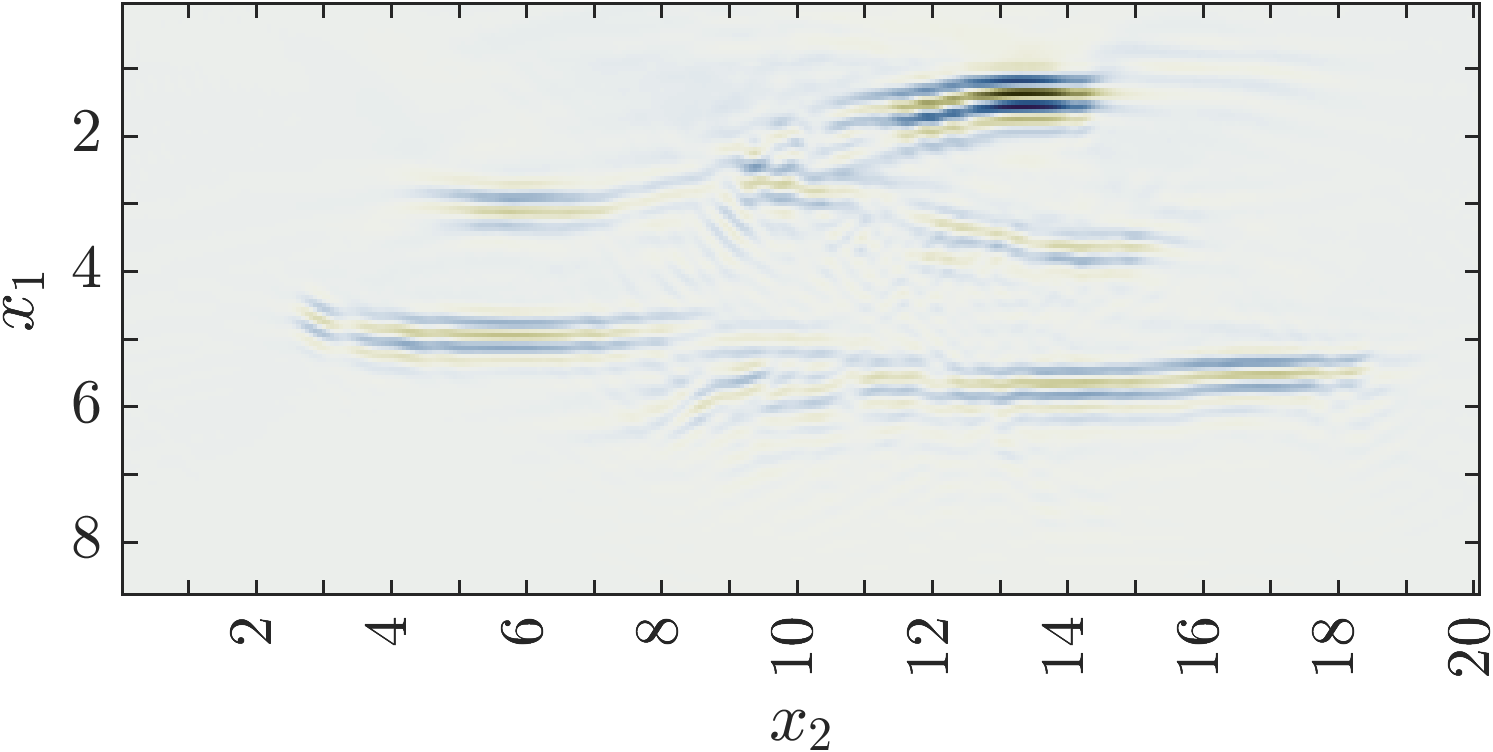}&\includegraphics[width = 0.35\textwidth]{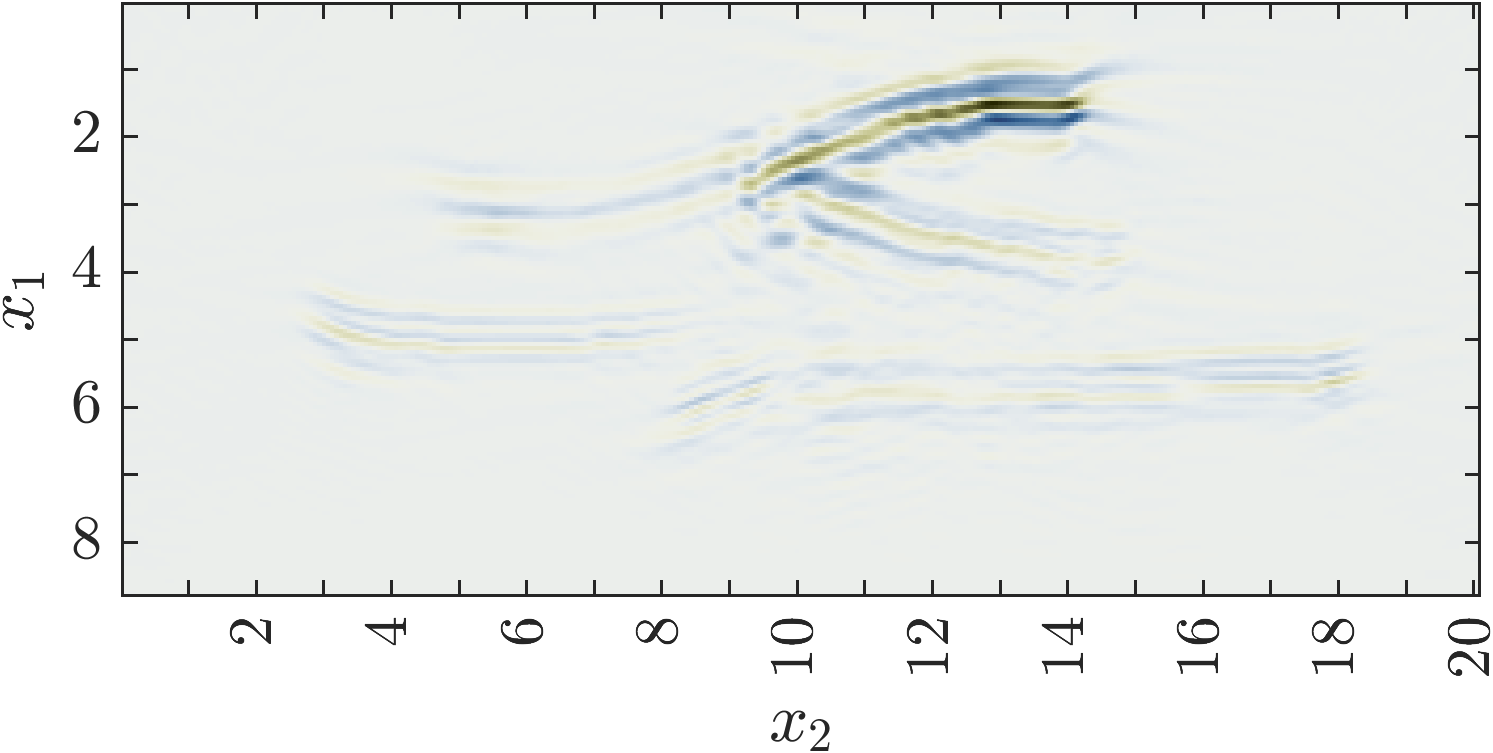} 
\end{tabular}
\vspace{-0.1in}
\caption{Imaging function~\eqref{eq:Jgpp} for all polarization pairs. The setup and axes are as in Fig. \ref{fig:Crack2} and the medium has the dielectric permittivity shown in Fig. \ref{fig:Crack5}.}
\label{fig:Crack6}
\end{figure}

\section{Inversion}
\label{sect:INVERSE}
In this section we consider the quantitative estimation of the matrix valued wave speed $\bc$. This is traditionally 
done via optimization, as described in the introduction and in equation~\eqref{eq:FWI2}. In light of our data mapping 
described in section~\ref{sect:waveDec}, we reformulate the traditional (FWI) approach as 
\begin{equation}
\min_{\tilde \bc \in \mathscr{C}} \mathcal{O}^{\FWI}(\tilde \bc) + \mbox{regularization}, ~~\mathcal{O}^{\FWI}(\tilde \bc) = \tau \sum_{j=0}^{2n-1} \left\| \bbD(j \tau) - \bbD(j \tau;\tilde \bc) \right\|_F^2.
\label{eq:Inv1}
\end{equation}
We introduce next a different approach to inversion, using our estimated internal wave~\eqref{eq:EstWave}. As explained in 
section~\ref{sect:INT_WAVE1}, this wave fits the data by construction. However, due to the wrong orthonormal basis 
stored in $\bV( \cdot; \tilde \bc)$, the  wave~\eqref{eq:EstWave} is not a solution of the wave equation.  If it was a solution, 
then $\tilde \bc$ would be close to the true wave speed, assuming uniqueness and stability of the inverse wave scattering problem, which holds for proper regularization and an appropriate search space $\mathscr{C}$. 

The solution of the wave equation with wave speed $\tilde \bc$ is given by~\eqref{eq:BornWave}. 
Thus, we can minimize the internal wave misfit
\begin{align}
\int_{\Omega} d \bx \, \sum_{j=0}^{n-1} \left\| \bu_j^{\rm est}(\bx;\tilde \bc) - \bu_j(\bx;\tilde \bc) \right\|_F^2 &\stackrel{\eqref{eq:EstWave},\eqref{eq:BornWave}}{=} 
\int_{\Omega} d \bx \, \left\| \bV(\bx,\tilde \bc) \left[\bR - \bR(\tilde \bc)\right] \right\|_F^2 \nonumber \\&= 
 \left\|\bR - \bR(\tilde \bc) \right\|_F^2,
\label{eq:Inv2}
\end{align}
where the last equality is due to the orthonormality of the basis in $\bV( \cdot; \tilde \bc)$. 
We estimate the wave speed using the optimization
\begin{equation}
\min_{\tilde \bc \in \mathscr{C}} \mathcal{O}(\tilde \bc) + \mbox{regularization}, ~~\mathcal{O}(\tilde \bc) = \left\|\bR(\tilde \bc) \bR^{-1} - \bI_{2nm} \right\|_F^2,
\label{eq:Inv3}
\end{equation}
where the objective function is a slight modification of~\eqref{eq:Inv2}, motivated by 
the following:  It is known that variations of $\tilde \bc$ that are far from the array can produce weak reflections that 
are masked from reflections from shallower depth. These weak reflections appear as small entries in the data matrices 
and, according to Theorem \ref{thm.1}, the Gramian $\bbM$.  These, in turn, result in small eigenvalues 
of the square root $\bR$ of $\bbM$. By using the inverse $\bR^{-1}$ in the definition of the objective function\footnote{Once we regularize the ROM construction, the mass matrix and its block Cholesky square root $\bR$ are well conditioned, so inverting $\bR$ is not an issue.}, 
we expect to recover the hidden reflections that are distinguishable from noise (recall the discussion in section 
\ref{sect:ROMnoise}) and thus image better.

\subsection{The inversion algorithm}
\label{sect:invAlg}
There are two user defined components of the inversion algorithm:  The first is the choice of the search space $\mathscr{C}$  
and the second is the choice of the regularization in the optimization~\ref{eq:Inv3}. 

We define $\mathscr{C}$ as follows:
First, we parametrize the Cholesky square root of the squared inverse wave speed tensor 
\begin{equation}
\mu_o \tilde{\beps}(\bx) \stackrel{\eqref{eq:defc}}{=} 
 \bc^{-2}(\bx) = \bth^T(\bx) \bth(\bx),  \qquad \bth(\bx) = \begin{pmatrix} \gamma_1(\bx) & \gamma_3(\bx) \\ 0 & \gamma_2(\bx) \end{pmatrix},
\label{eq:Inv4}
\end{equation}
to ensure that $\tilde{\beps}$ remains symmetric and positive definite. The parametrization is done using the basis functions 
$\{ \phi_j \}_{j=1}^N$ described in Appendix \ref{ap:num}, 
\begin{equation}
\gamma_l(\bx) = (1-\delta_{l,3}) c_o^{-1} + \sum_{j=1}^N \alpha_{l,j} \phi_j(\bx),
\label{eq:Inv5}
\end{equation}
with $
\boldsymbol{\alpha} = \left(\alpha_{1,1} , \ldots, \alpha_{3,N} \right) \in \RR^{3N}$ to be determined by optimization.

We use the simplest Tikhonov type regularization penalty $ \nu\|\boldsymbol{\alpha}\|^2$ in~\eqref{eq:Inv3},
with parameter $\nu$ chosen as explained in Appendix \ref{ap:num}. However, other choices of regularization, 
that incorporate prior information about $\bc$ can be used and would likely lead to sharper estimates.

\vspace{0.05in}
\begin{algorithm} \textbf{\emph{(Inversion algorithm)}}
\label{alg:AlgR}

\vspace{0.04in} \noindent \textbf{Input:} Data matrices $\bbD(j \tau)$, for $j = 0, \ldots, 2n-1$.

 \vspace{0.04in} \noindent 1. Compute $\bbM$ with block entries given in equation~\eqref{eq:calcGram}.
  If $\bbM$ is indefinite or poorly conditioned, use regularization as described in section
\ref{sect:ROMnoise}.

\vspace{0.04in} \noindent 2. Compute the block Cholesky square root $\bR$ of $\bbM$. 

  \vspace{0.04in} \noindent 3. Starting with  $\boldsymbol{\alpha}^{(0)} = {\bf 0}$  proceed:

 \begin{itemize}
 \itemsep 0.03in
\item  For update index $j \ge 1$ calculate $\bth(\bx)$ as in equation~\eqref{eq:Inv5}, with $\boldsymbol{\alpha} = \boldsymbol{\alpha}^{(j-1)}.$ Compute $\tilde \bc$ from equation~\eqref{eq:Inv4}. 
\item Calculate $\bR(\tilde \bc)$ following the same procedure of calculating $\bR$.
\item Compute $\boldsymbol{\alpha}^{(j)}$ as a Gauss-Newton update for minimizing the objective function~\eqref{eq:Inv3} with
the user's choice of the regularization parameter $\nu$.
\item Go to the next iteration or stop when the user defined convergence criterion has been met.
\end{itemize}

\vspace{0.04in} \noindent \textbf{Output:}  The estimate of $\bc$ given by~\eqref{eq:Inv4}-\eqref{eq:Inv5} with $\boldsymbol{\alpha}$ calculated at step 3.
\end{algorithm}

\vspace{0.05in}
Note that since the ROM construction is causal, the inversion can be carried out in a layer peeling fashion, by time windowing
the data. If we use measurements up to time instant $t_{n'}$, for $n' < n$, then the ROM contains information 
up to the distance of order $c_o t_{n'}/2$ from the array. Thus, we can estimate the medium near the array, and then increase $n'$ 
to obtain estimates at further distance. This can be very useful in speeding up the optimization. 

\subsection{Numerical results} We illustrate the performance of the inversion algorithm for the medium shown in 
Fig. \ref{fig:inv1}. There are three anisotropic inclusions, modeled by the piecewise constant wave speed 
$\bc$ with components $c_{j,j'}$ plotted in the figure, for $j,j' = 1,2$.

\begin{figure}[h]
\centering
\begin{tabular}{ccc}
$c_{1,1}/c_o$ & $c_{2,2}/c_o$  & $c_{1,2}/c_o$\\ 
\hspace{-0.05in}\includegraphics[width = 0.3\textwidth]{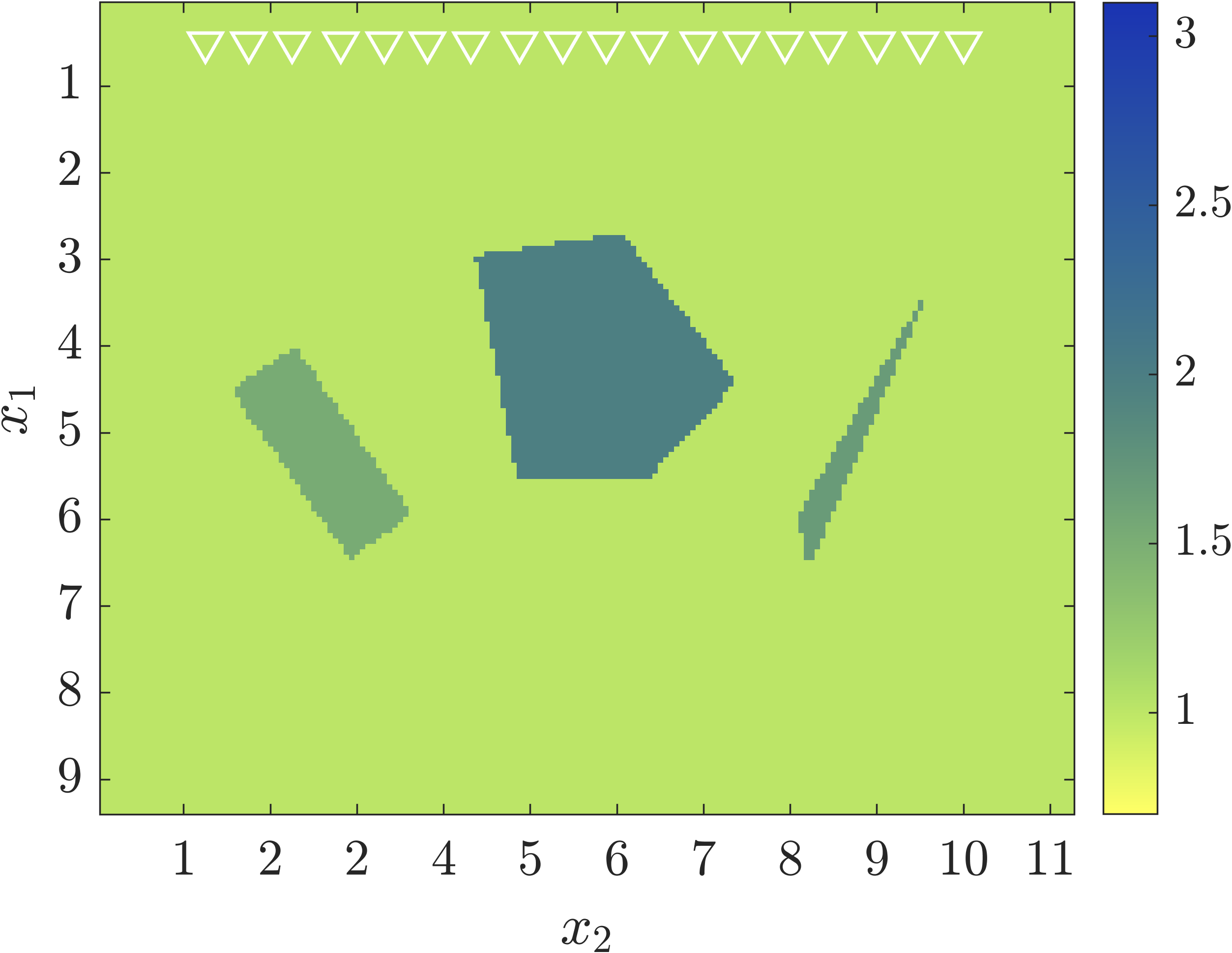}&\hspace{-0.05in} \includegraphics[width = 0.29\textwidth]{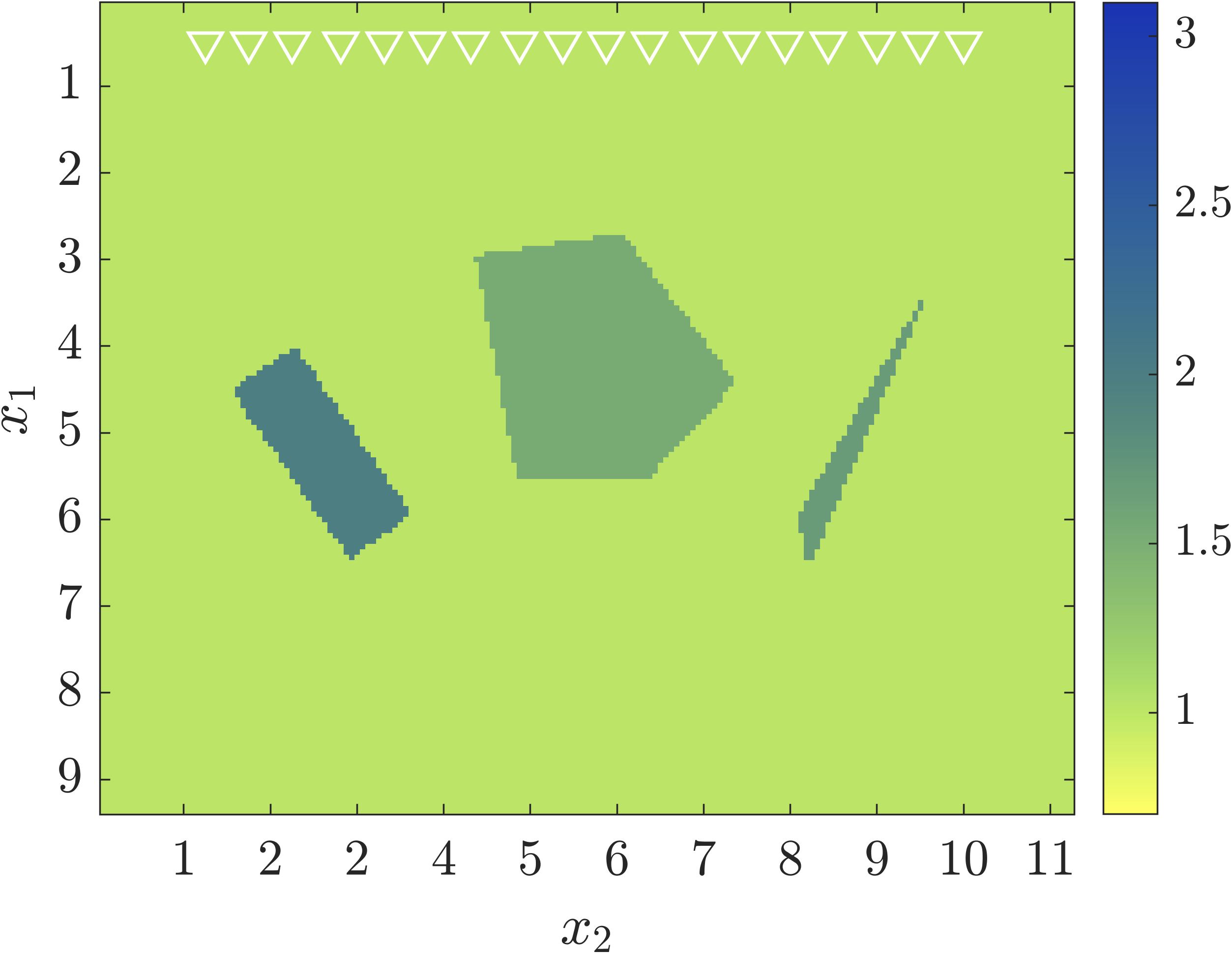} &\hspace{-0.05in}
\includegraphics[width = 0.3\textwidth]{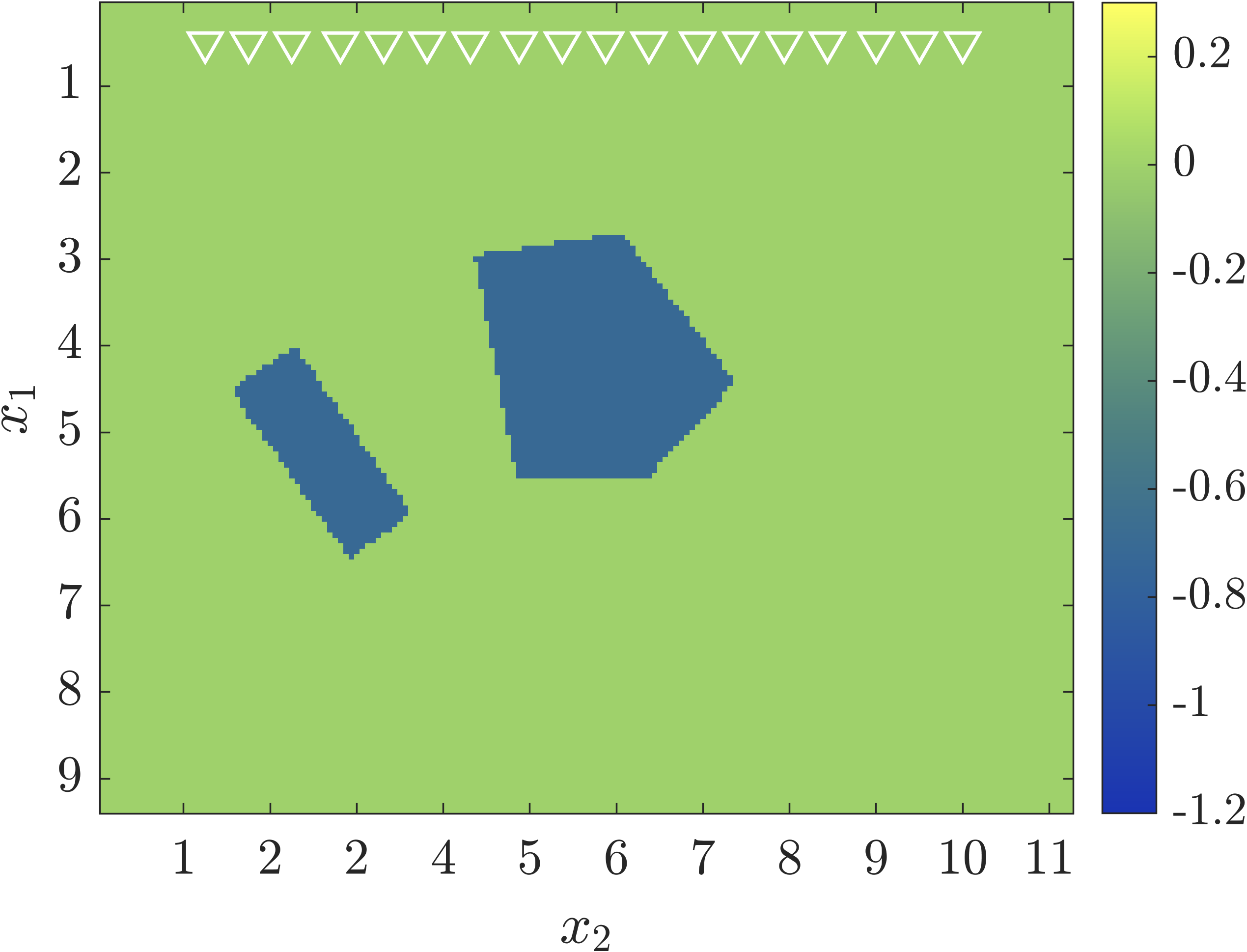}
\end{tabular}
\vspace{-0.1in}
\caption{The anisotropic medium with three inclusions. The axes are in units of $\la_c$. The contour of the inclusions is superposed on these plots.}
\label{fig:inv1}
\end{figure}

\begin{figure}[h]
\centering
\begin{tabular}{ccc}
$c_{1,1}/c_o$ & $c_{2,2}/c_o$  & $c_{1,2}/c_o$\\ 
\hspace{-0.05in}\includegraphics[width = 0.3\textwidth]{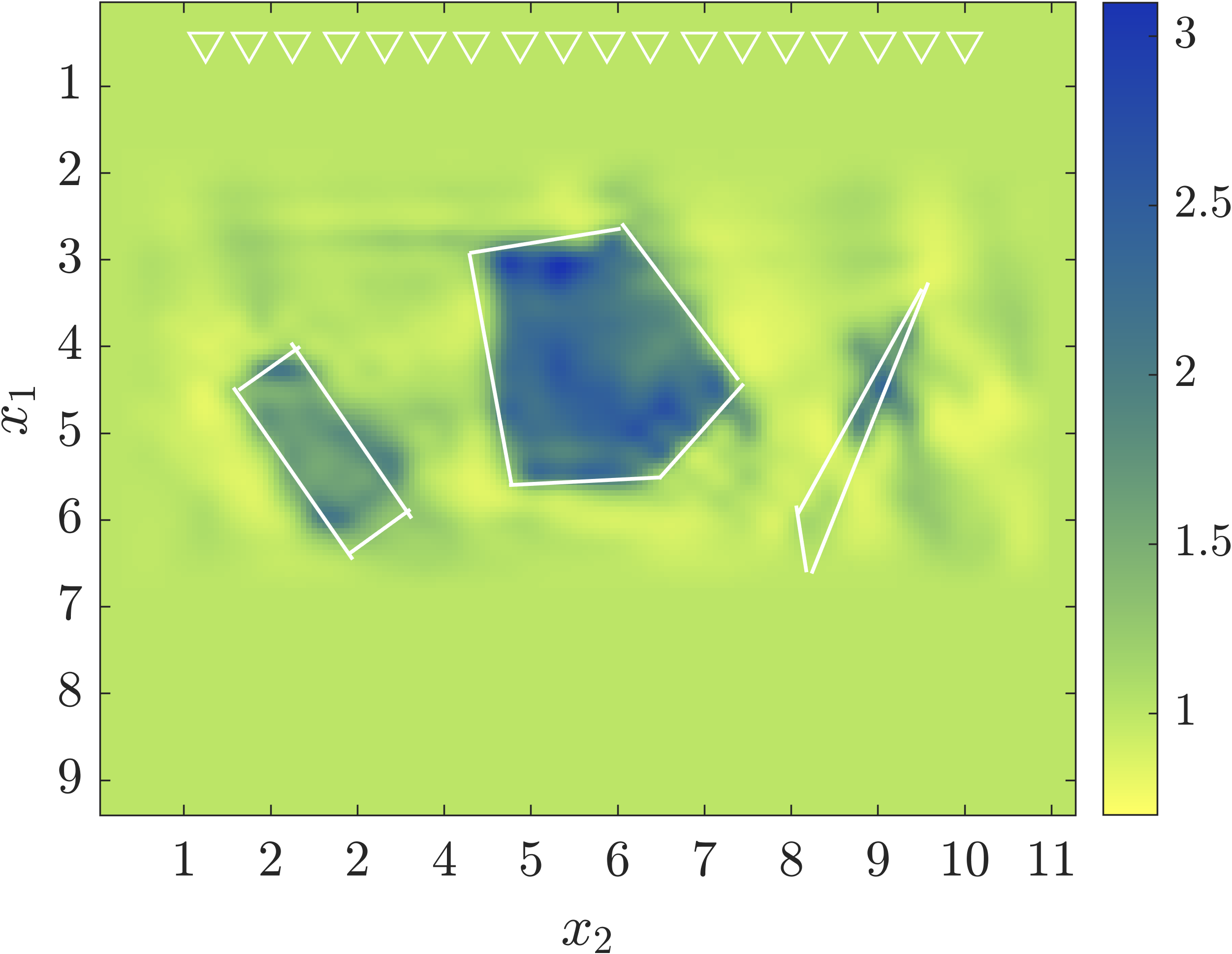}&\hspace{-0.05in} \includegraphics[width = 0.3\textwidth]{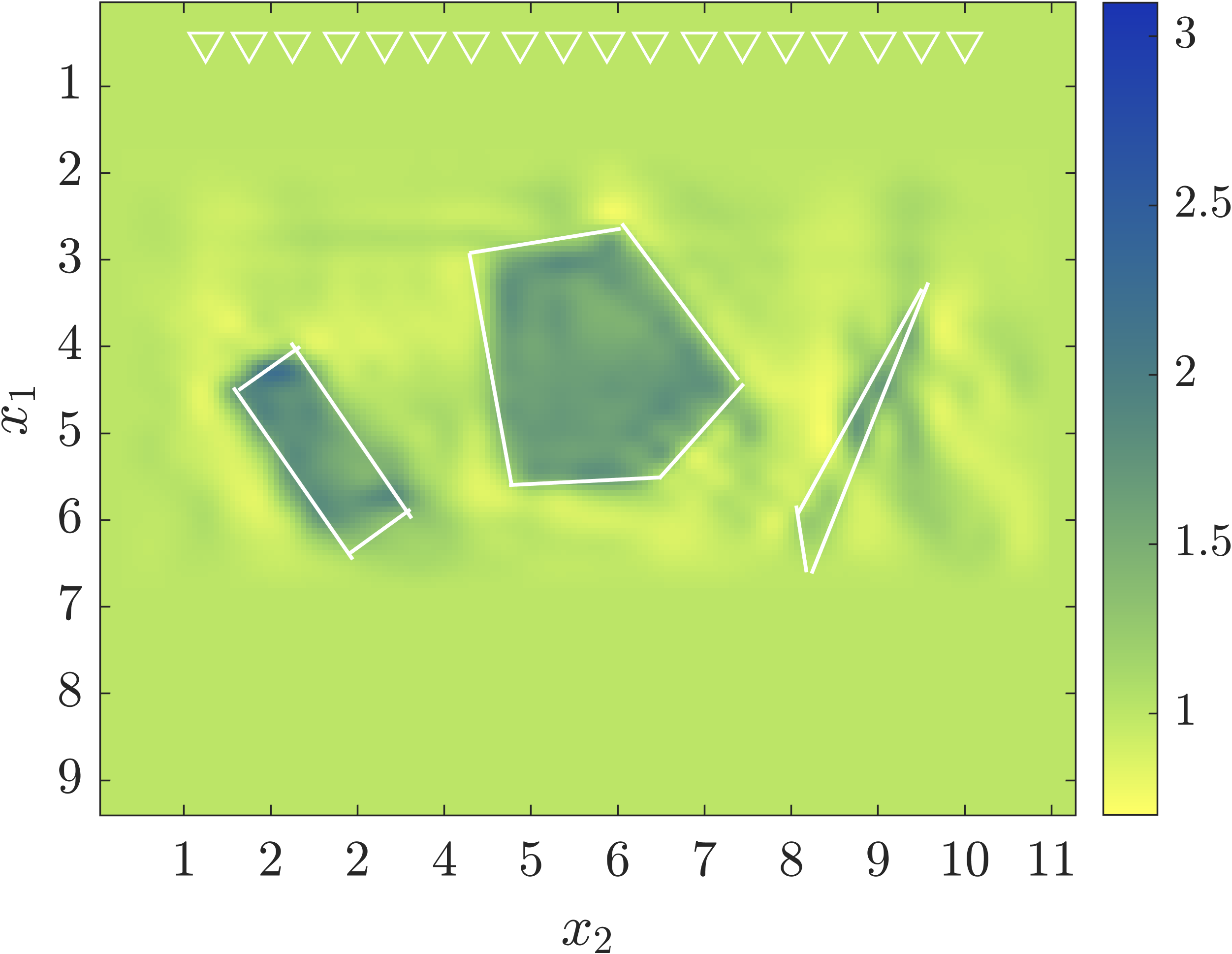} &\hspace{-0.05in}
\includegraphics[width = 0.3\textwidth]{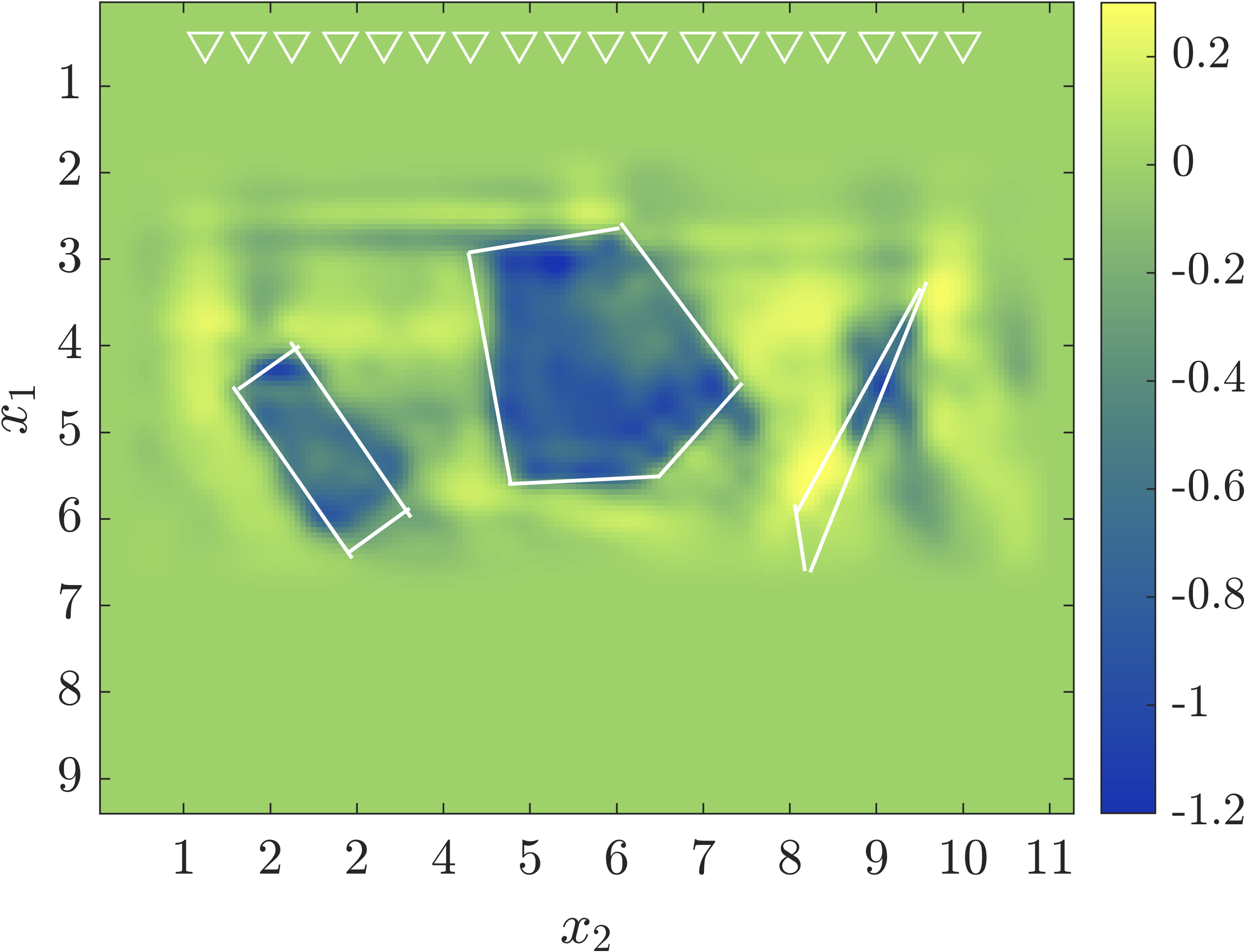} \\
\hspace{-0.05in}\includegraphics[width = 0.3\textwidth]{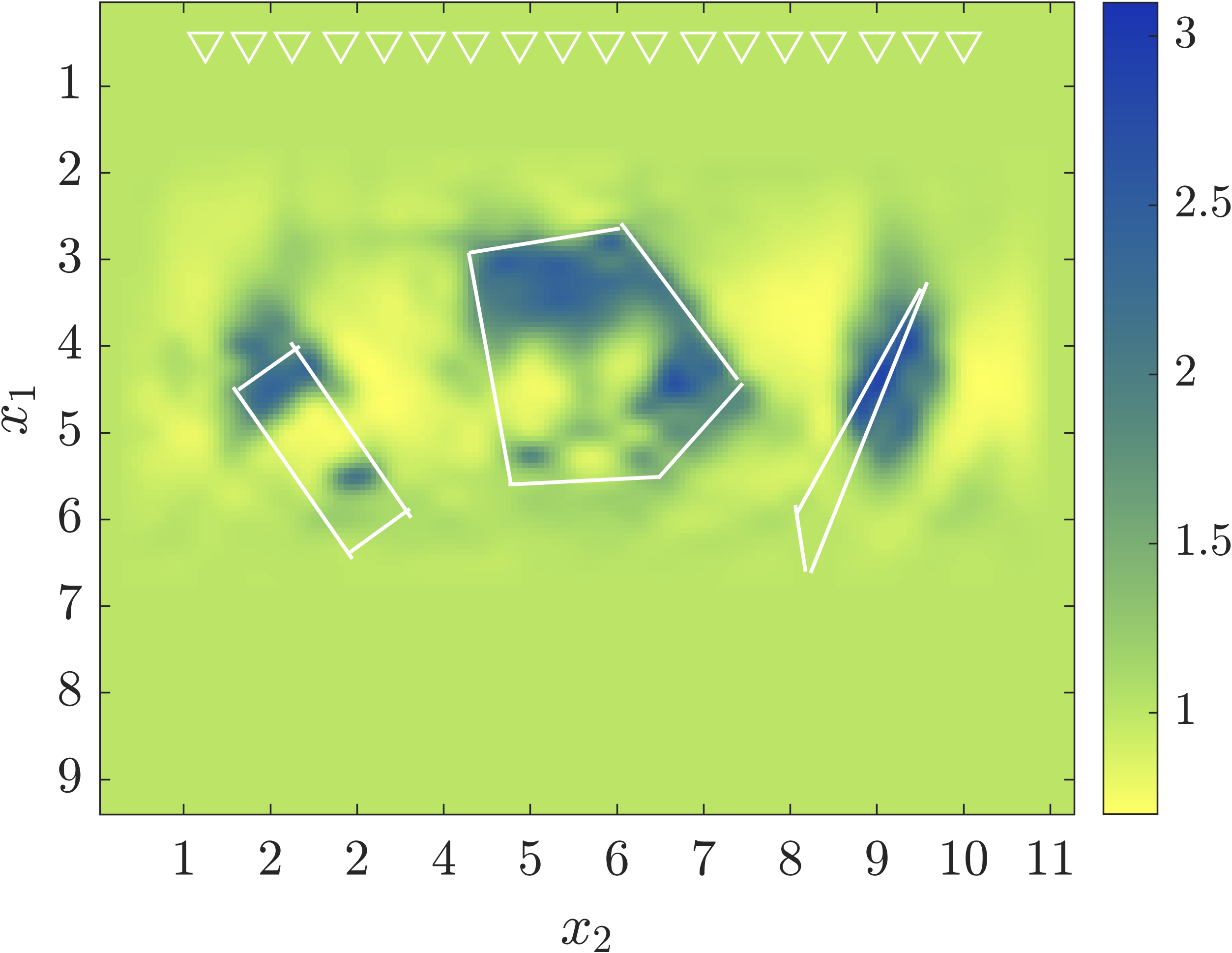}&\hspace{-0.05in} \includegraphics[width = 0.3\textwidth]{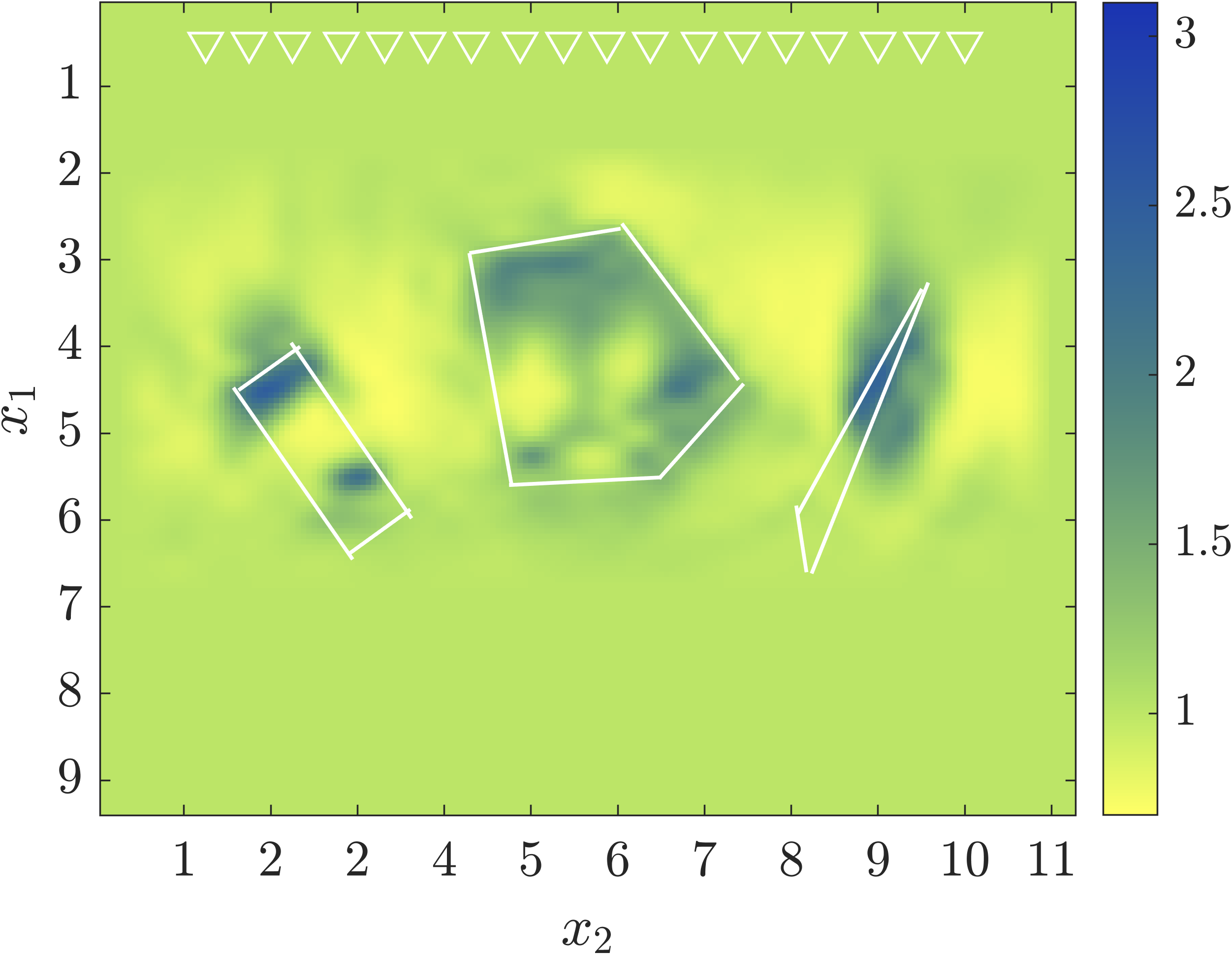} &\hspace{-0.05in}
\includegraphics[width = 0.3\textwidth]{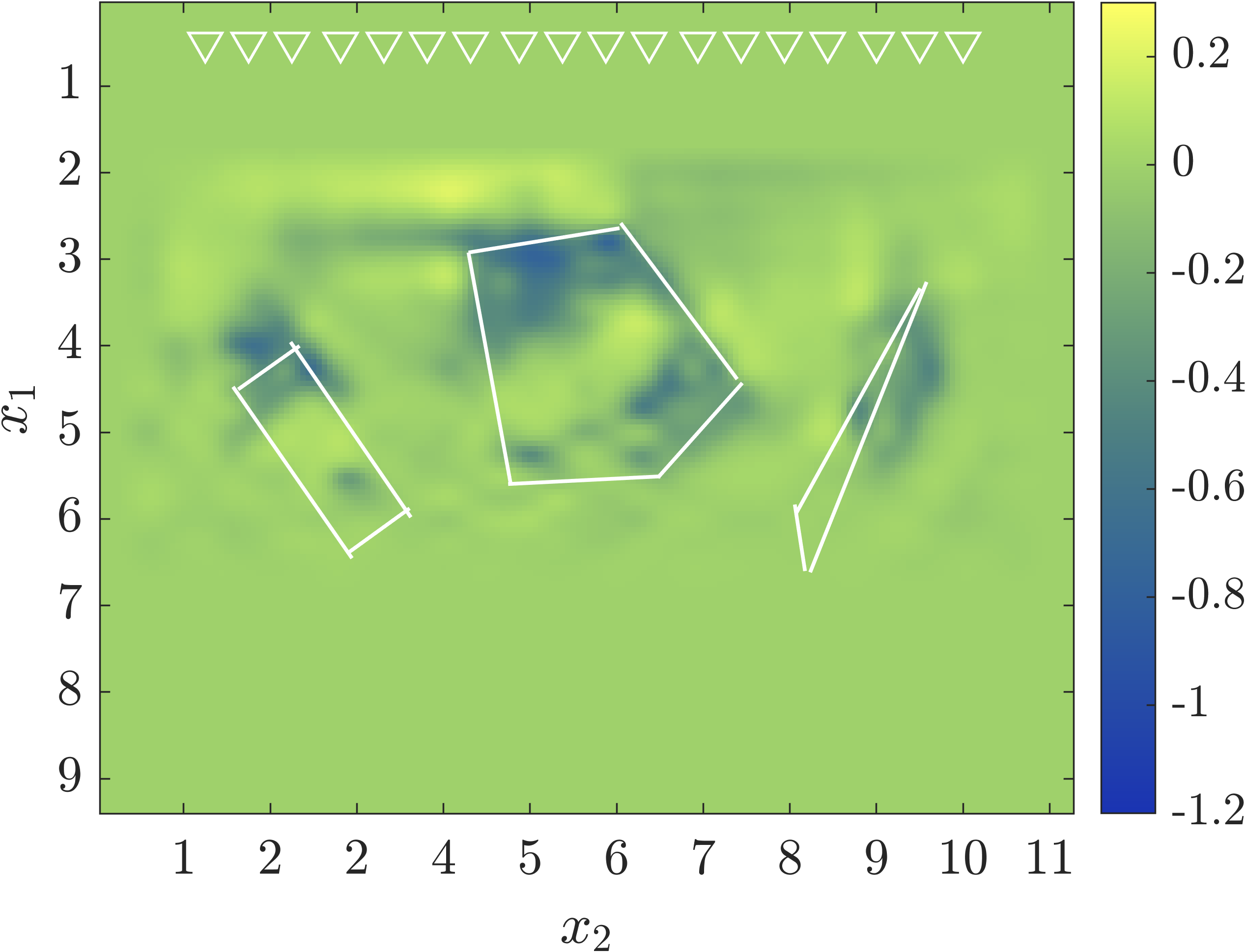}
\end{tabular}
\vspace{-0.1in}
\caption{The inversion results for the medium plotted in Fig.~\ref{fig:inv1}.  Top row: results obtained with Algorithm \ref{alg:AlgR}. Bottom row: Results given by the FWI approach. The axes are in units of $\la_c$. The contour of the inclusions is superposed on these plots.}
\label{fig:inv2}
\end{figure}

The inversion results are shown in Fig.~\ref{fig:inv2}. { The ROM approach (top plots) recovers well the shape of the inclusions, especially the 
left and middle ones.  The components of $\bc$ there are also well estimated.  There is some leakage  in the off-diagonal 
component of $\bc$ for the right, thin inclusion. This inclusion is very difficult to determine, because the data recorded at the array is dominated by the reflection at its top corner. 

As is typical of  FWI inversion \cite{virieux2009overview},  the results shown in the bottom plots of Fig.~\ref{fig:inv2} do not determine the shape of the inclusions and the wave speed there.  For instance, FWI sees the top and bottom of the leftmost inclusion, which cause two strong reflections registered at the arry.
However, the bottom is misplaced because the wave speed inside the inclusion is not correctly identified.  

Both methods tend to overshoot the value of the wave speed  near the edges of the inclusions. This is a normal Gibbs phenomenon seen  with Tikhonov regularization and it can be mitigated with a different regularization that is better suited for a piecewise constant $\bc$. }

\section{Imaging with the estimated kinematics}
\label{sect:CONJ}
In this section, we explain how one may combine the imaging and inversion methods described in sections~\ref{sect:normg} and~\ref{sect:invAlg}. This may be useful for distinguishing better the contour of inclusions buried in the medium, 
if the regularization of the inversion is not designed to do so.

We give the explanation using the example in Fig.~\ref{fig:Comb1}, 
where the medium contains a rectangular inclusion filled with an anisotropic medium. The inclusion is large enough to affect significantly the kinematics. Consequently, our imaging function~\eqref{eq:Jgpp}, displayed in the  left plot of 
Fig.~\ref{fig:Comb2}, gives the wrong estimate of the bottom right corner of the inclusion. The traditional image shown in 
the middle plot behaves similarly, although it shows a ghost feature at the bottom of the domain, due to multiple scattering.
It also has spurious oscillations near the top corner of the inclusion.

\begin{figure}[h]
\centering
\begin{tabular}{ccc}
$c_{1,1}/c_o$ & $c_{2,2}/c_o$  & $c_{1,2}/c_o$\\ 
\hspace{-0.05in}\includegraphics[width = 0.3\textwidth]{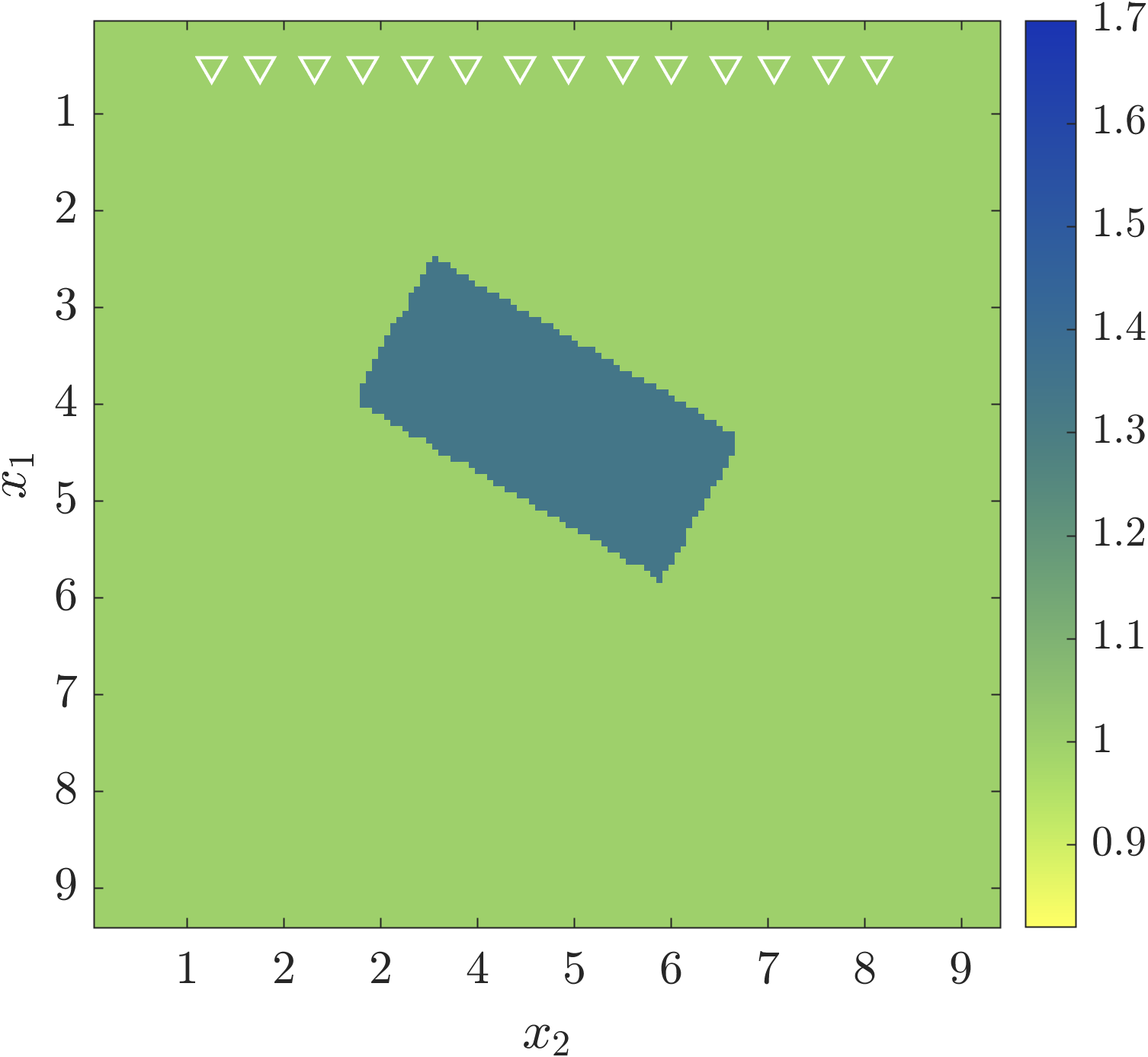}&\hspace{-0.05in} \includegraphics[width = 0.3\textwidth]{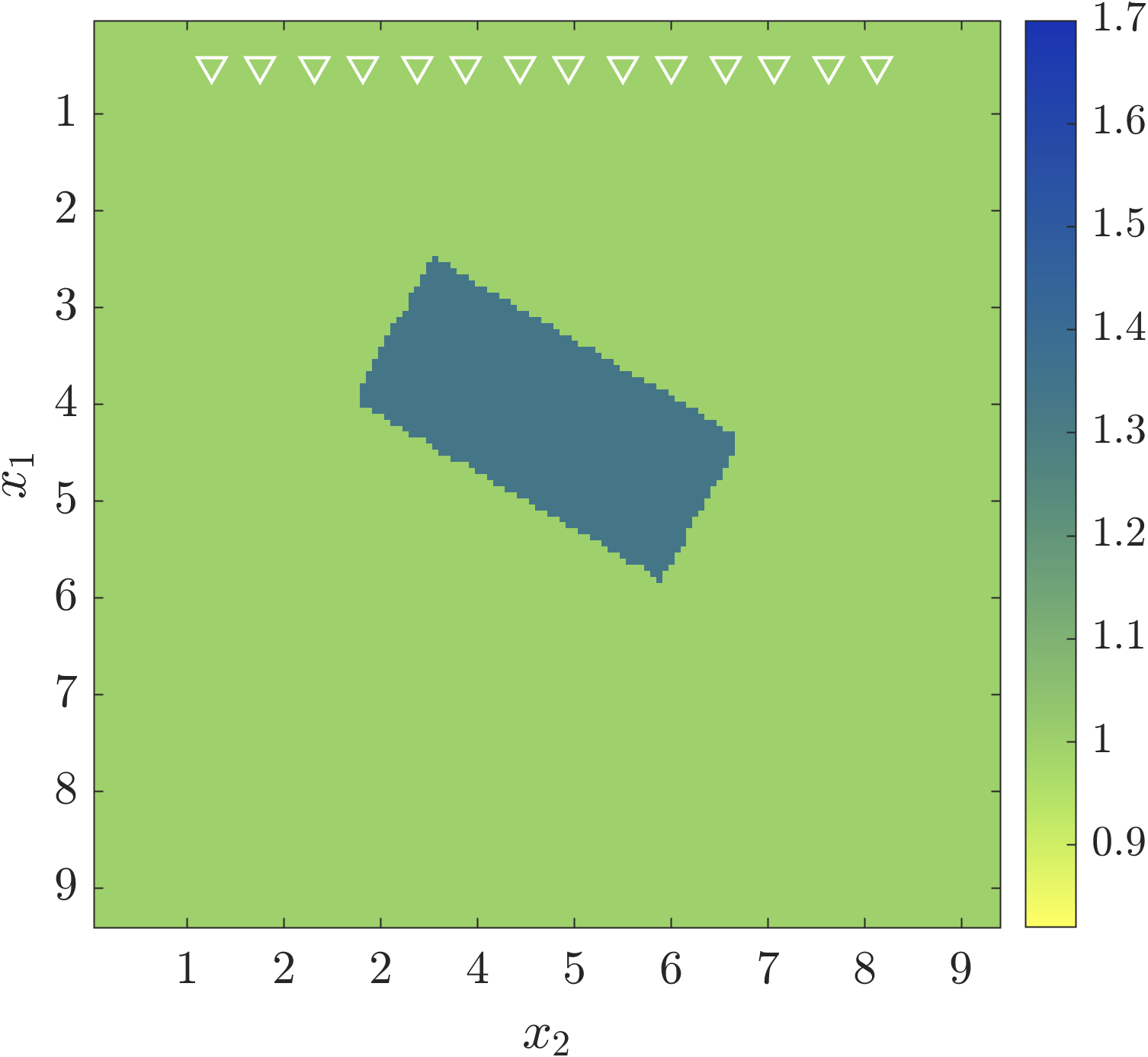} &\hspace{-0.05in}
\includegraphics[width = 0.3\textwidth]{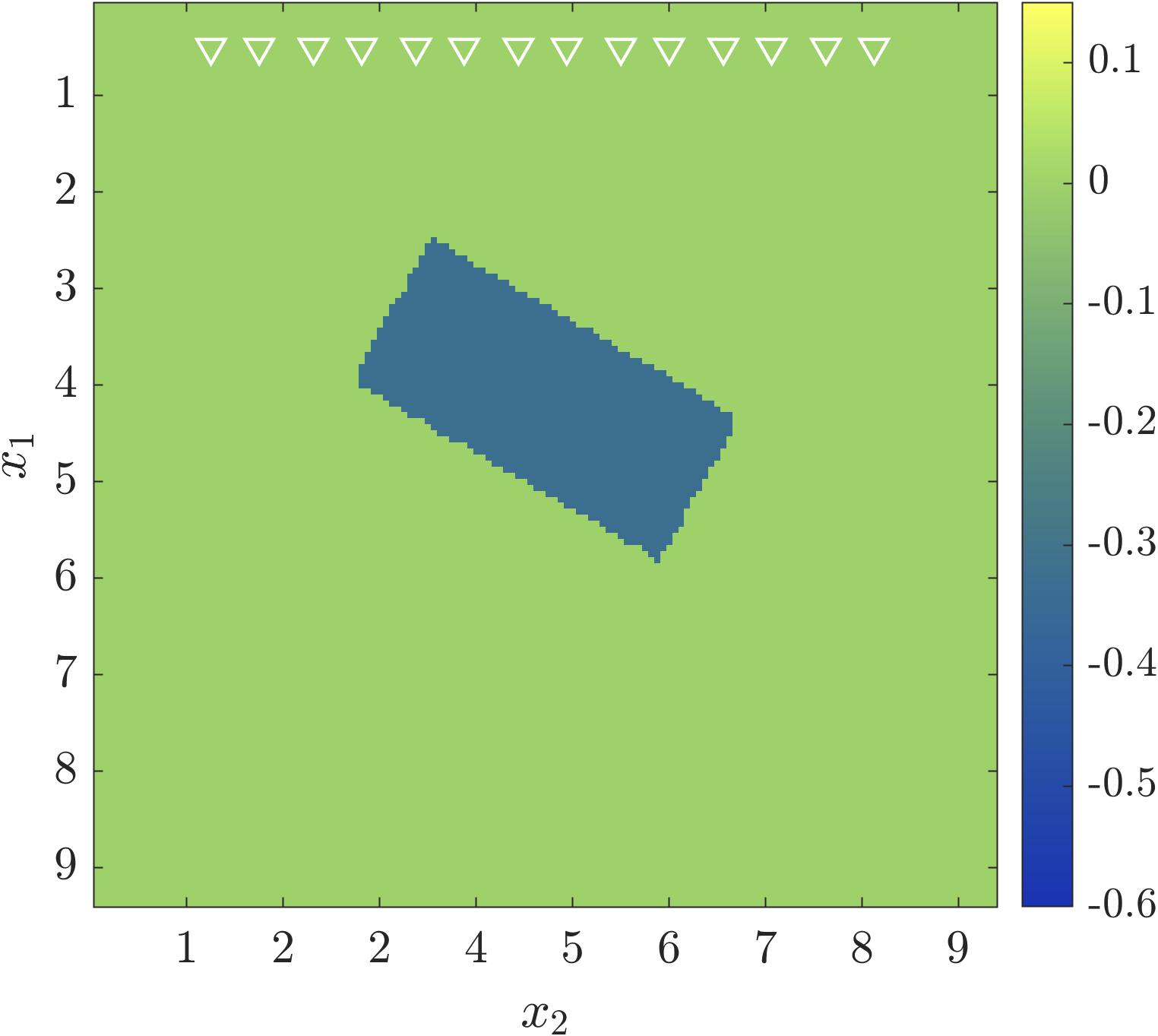}
\end{tabular}
\vspace{-0.1in}
\caption{Anisotropic medium with rectangle shaped inclusion. The axes are in units of $\la_c$. The contour of the inclusion is superposed on these plots.}
\label{fig:Comb1}
\end{figure}

\begin{figure}[h]
\centering
\includegraphics[width = 0.32\textwidth]{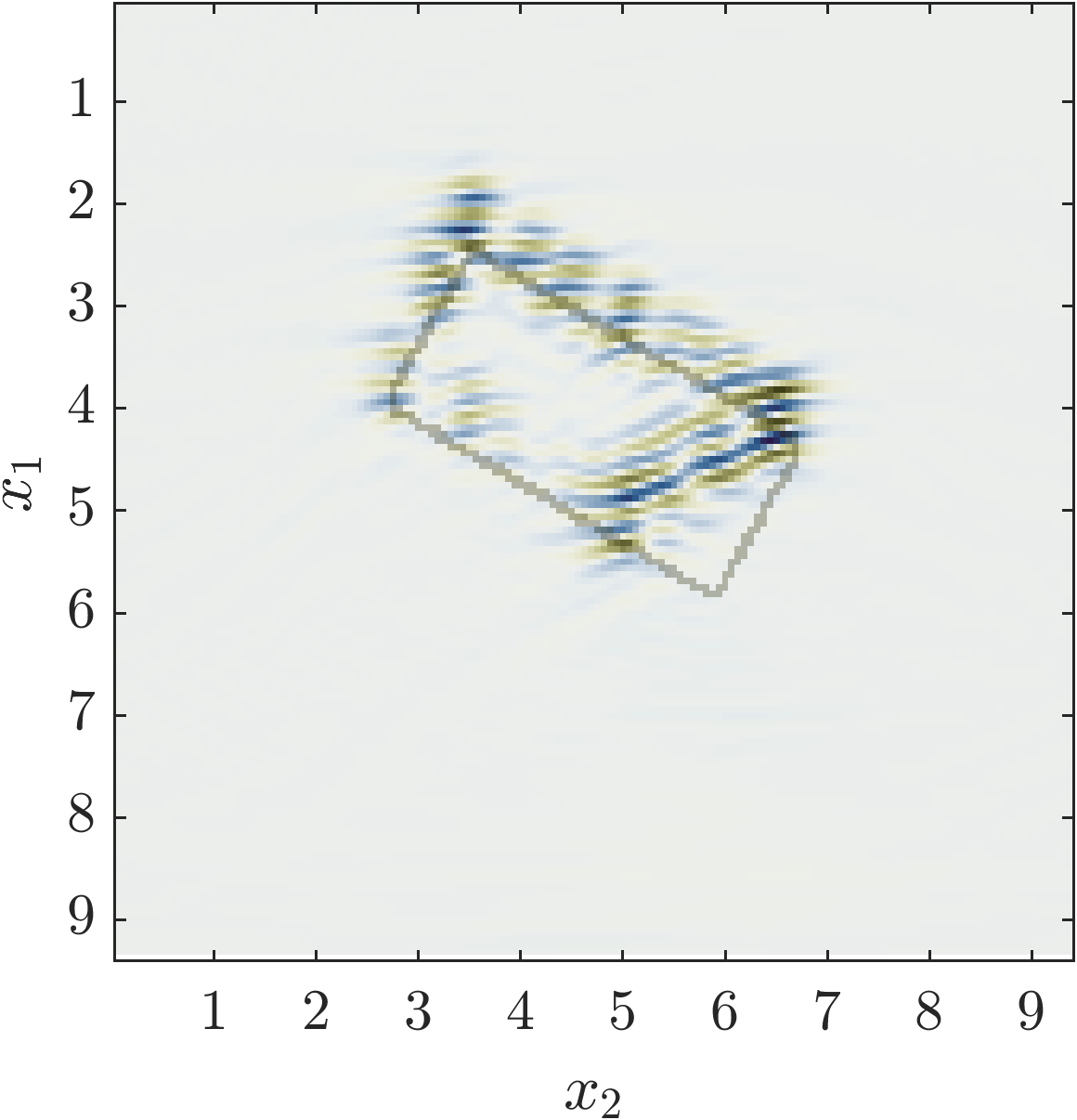}
 \includegraphics[width = 0.31\textwidth]{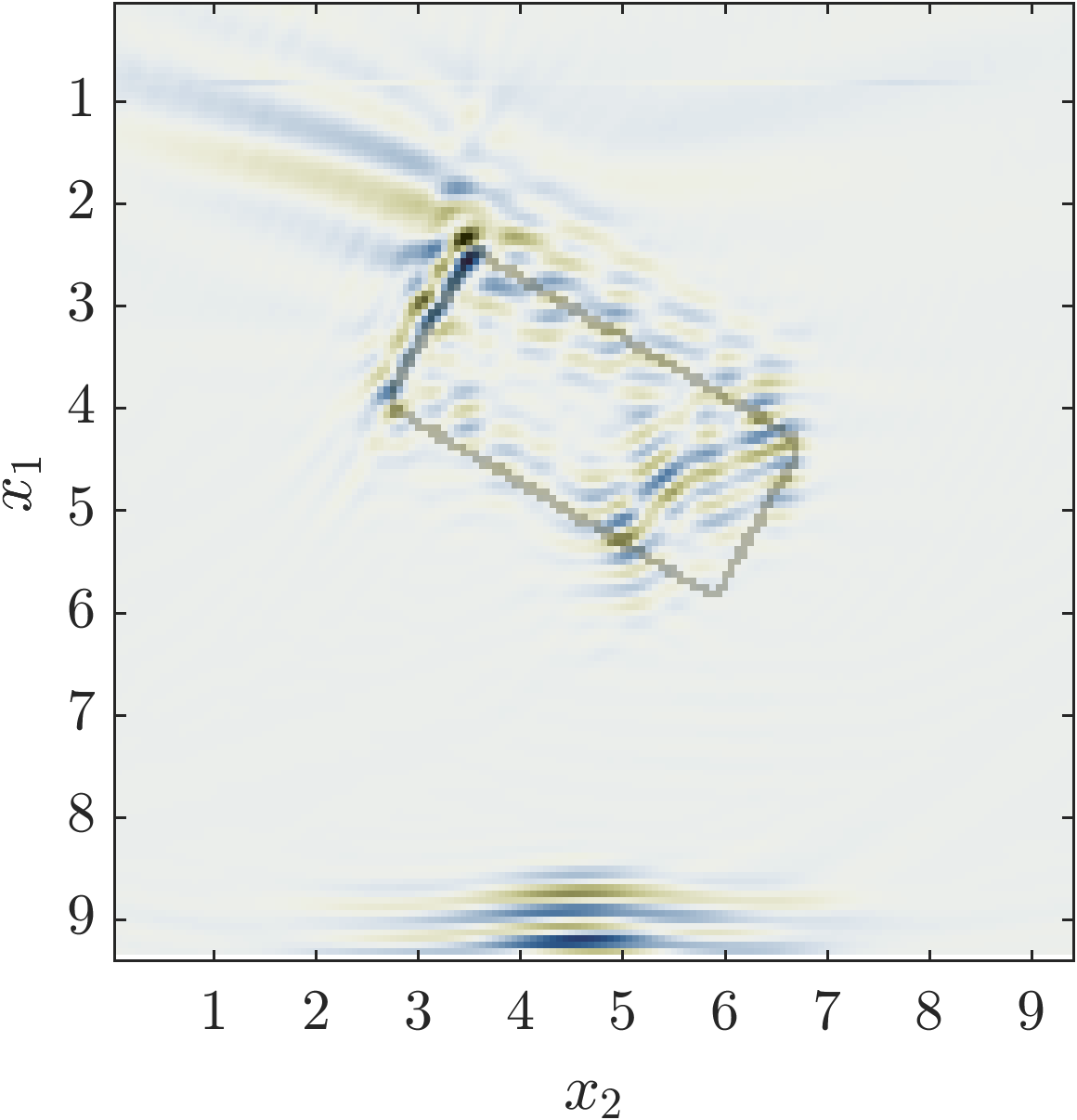}
  \includegraphics[width = 0.31\textwidth]{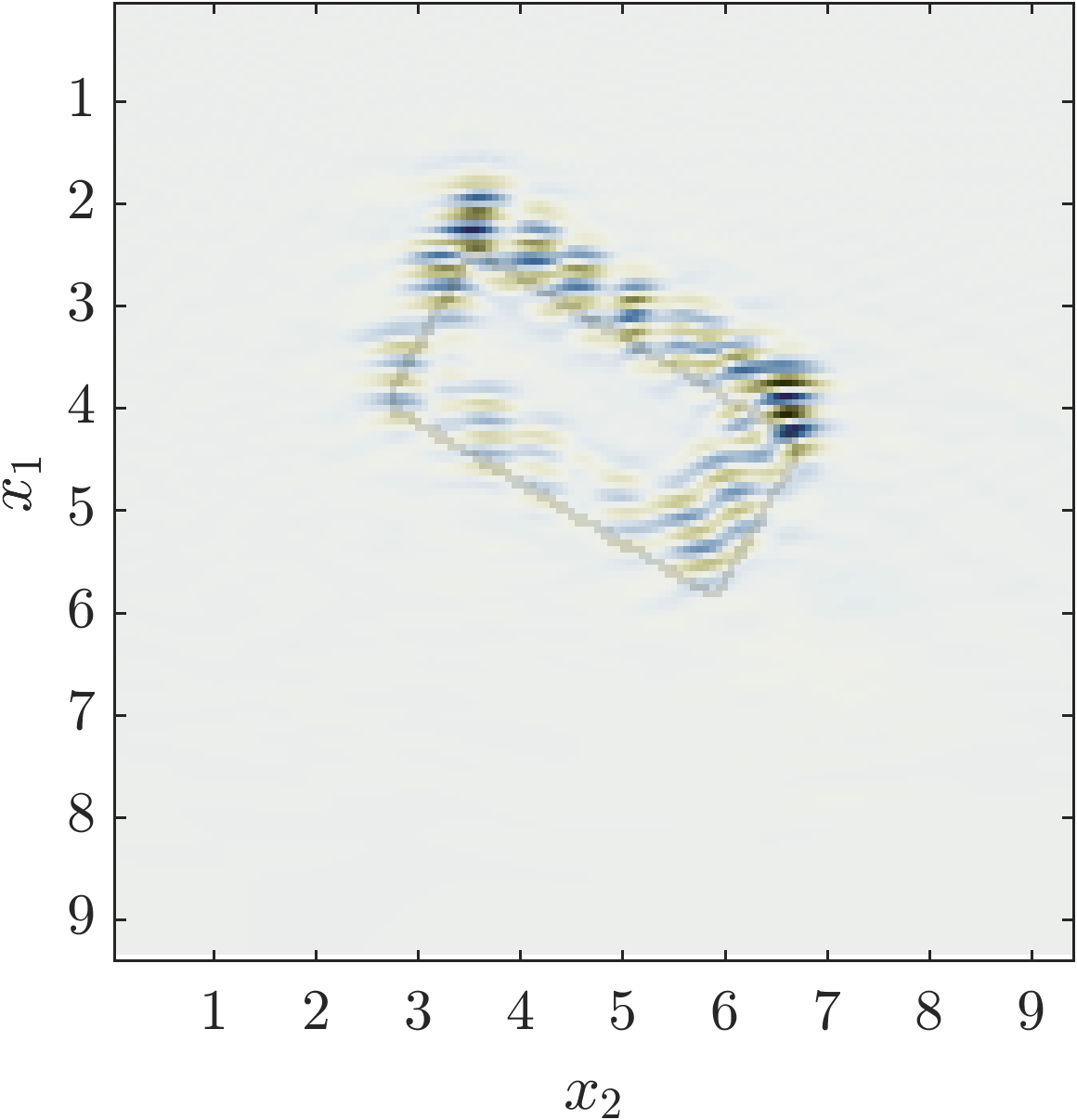}
\vspace{-0.1in}
\caption{First two plots: images of the medium in Fig.~\ref{fig:Comb1}, calculated using the constant reference speed $c_o \tensor{\bI}$. Left plot: the range derivative of the imaging function 
$\cI^{(2,2)}$. Middle plot the  traditional imaging function. 
The right plot:  the imaging function 
$\cI^{(2,2)}$ computed in the medium estimated by our inversion approach. 
The axes are in units of $\la_c$. The contour of the inclusion is superposed on these plots.}
\label{fig:Comb2}
\end{figure}

To correct the image in Fig.~\ref{fig:Comb2}, we need a better estimate of the kinematics. One way to get such an estimate is to use the result of the inversion with Algorithm~\ref{alg:AlgR}. The result of this inversion is shown in the top row of plots in 
Fig.~\ref{fig:Comb3}. For comparison, we also show in the bottom row of plots the result of the traditional (FWI) inversion. 
As we saw before, FWI identifies the top of the inclusion, but it does not give a good approximation of the smooth part 
of the wave speed, which determines the kinematics. 

\begin{figure}[h]
\centering
\begin{tabular}{ccc}
$c_{1,1}/c_o$ & $c_{2,2}/c_o$  & $c_{1,2}/c_o$\\ 
\hspace{-0.05in}\includegraphics[width = 0.3\textwidth]{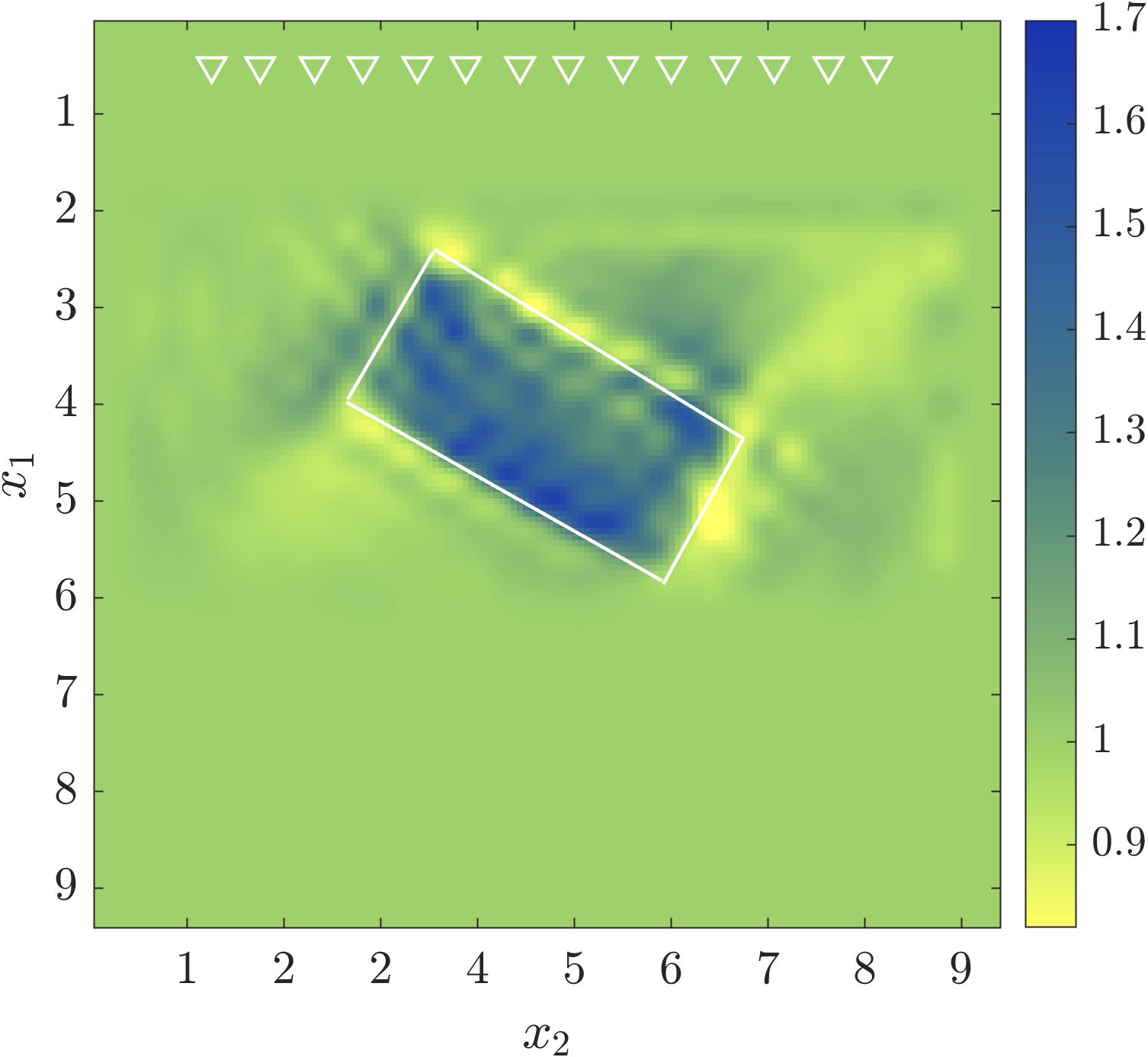}&\hspace{-0.05in} \includegraphics[width = 0.3\textwidth]{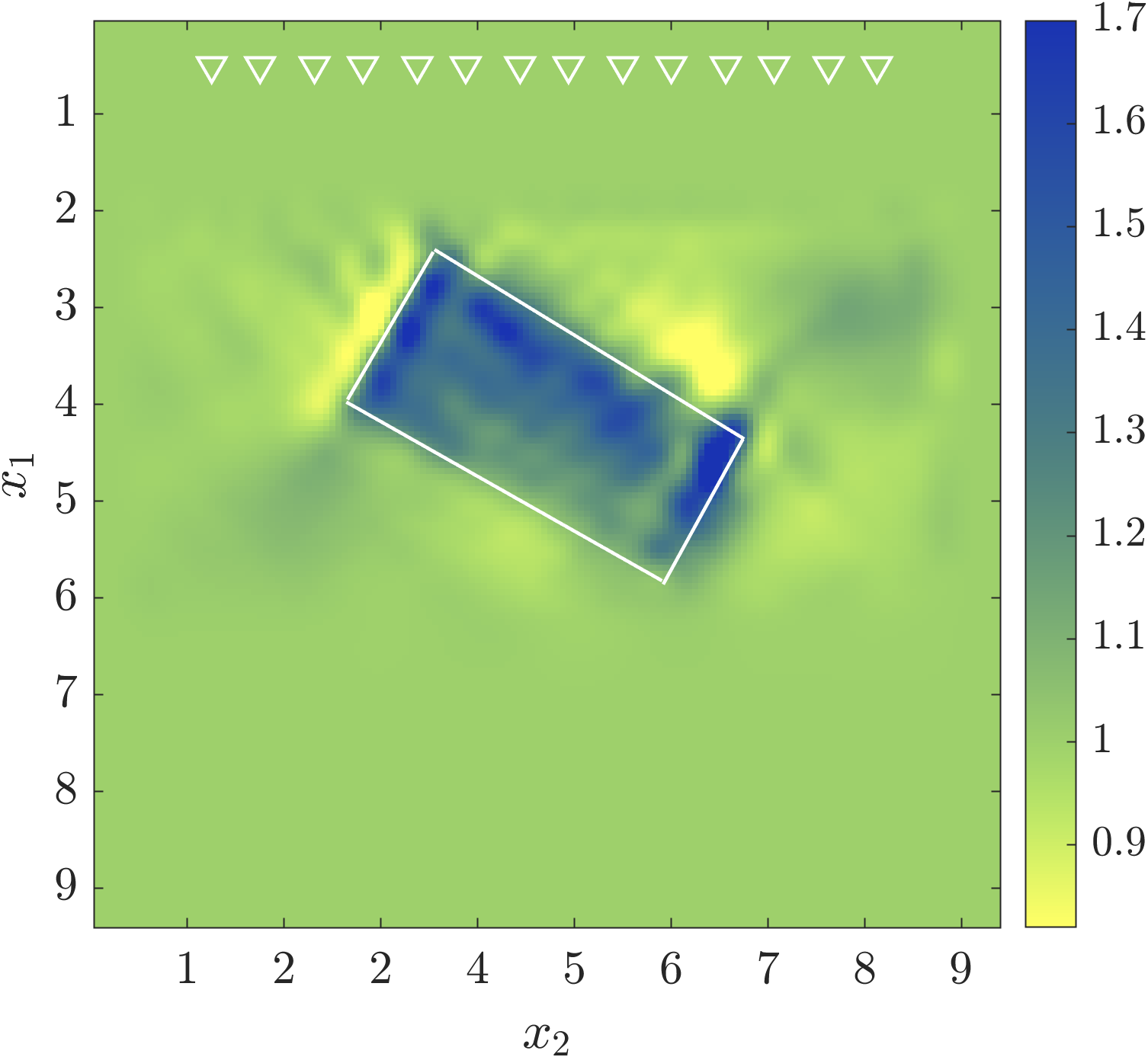} &\hspace{-0.05in}
\includegraphics[width = 0.3\textwidth]{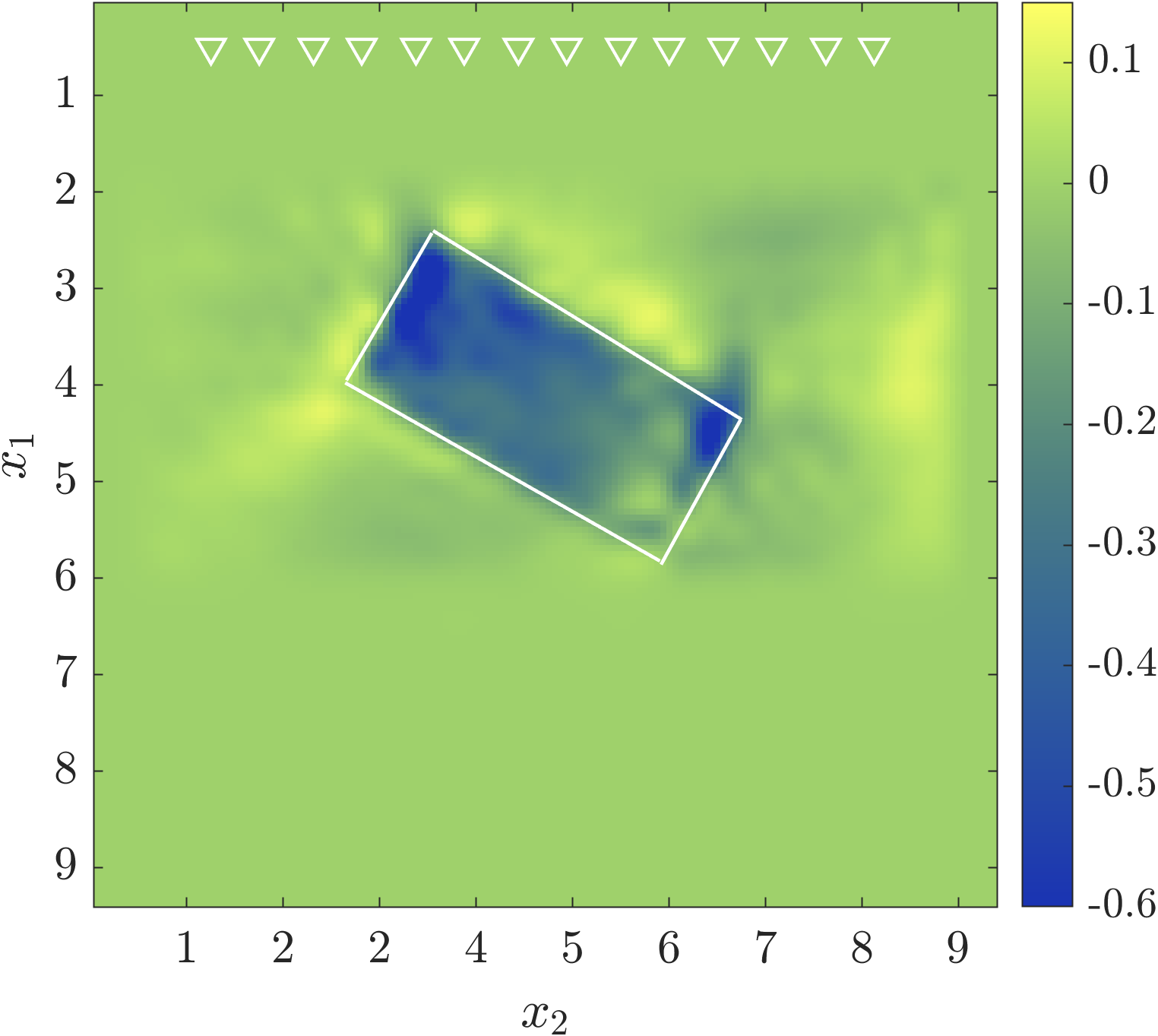} \\
\hspace{-0.05in}\includegraphics[width = 0.3\textwidth]{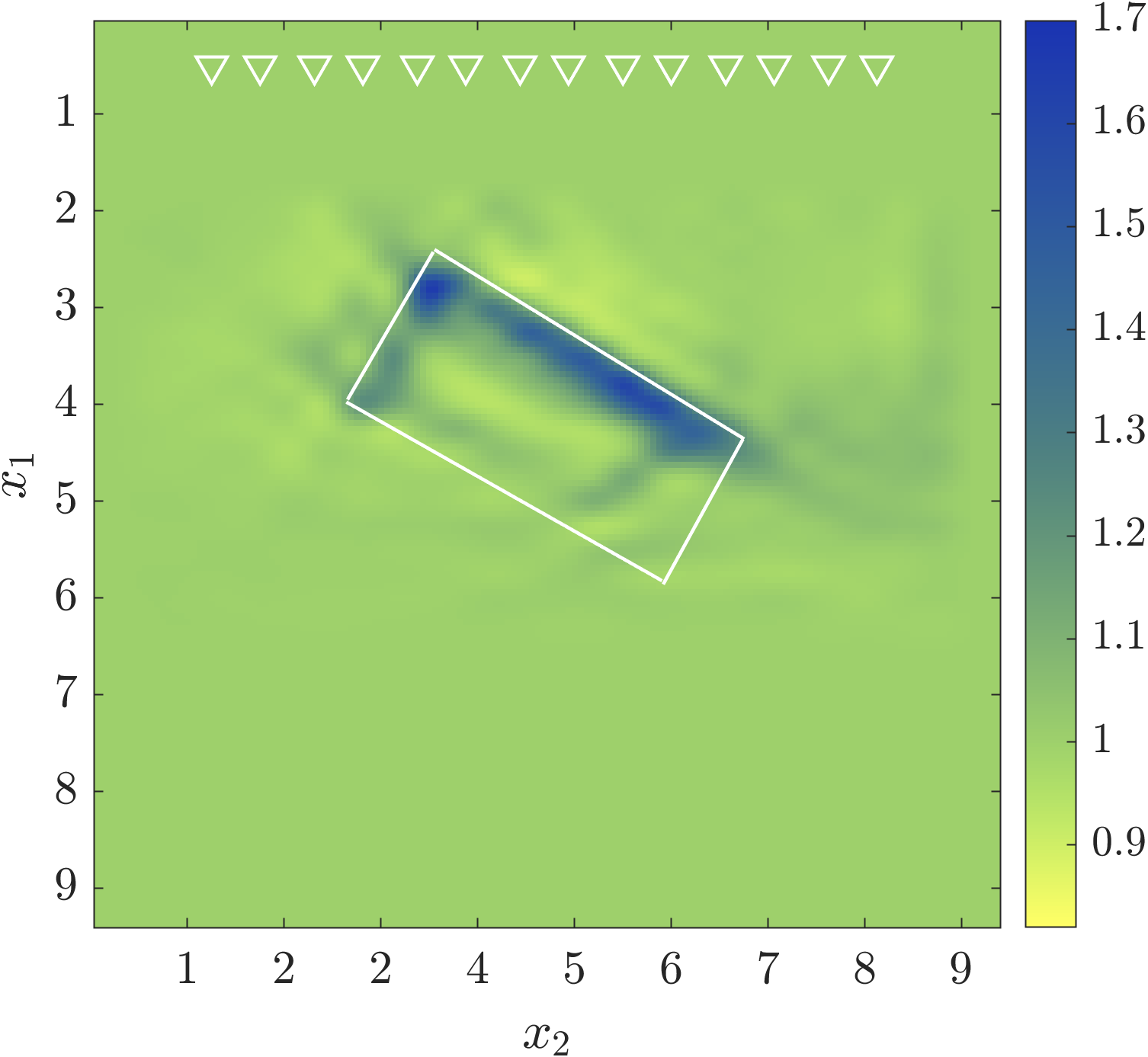}&\hspace{-0.05in} \includegraphics[width = 0.3\textwidth]{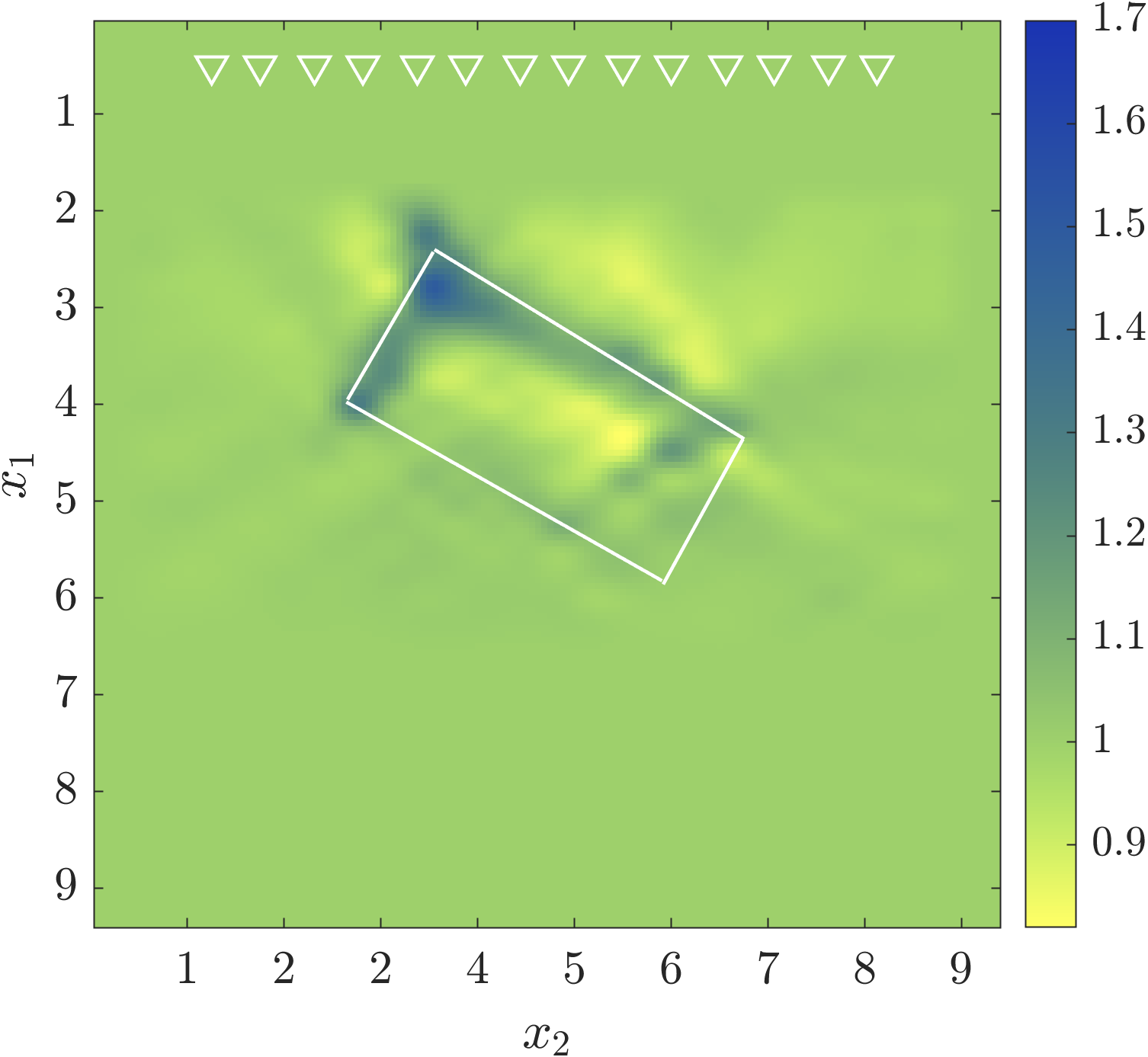} &\hspace{-0.05in}
\includegraphics[width = 0.3\textwidth]{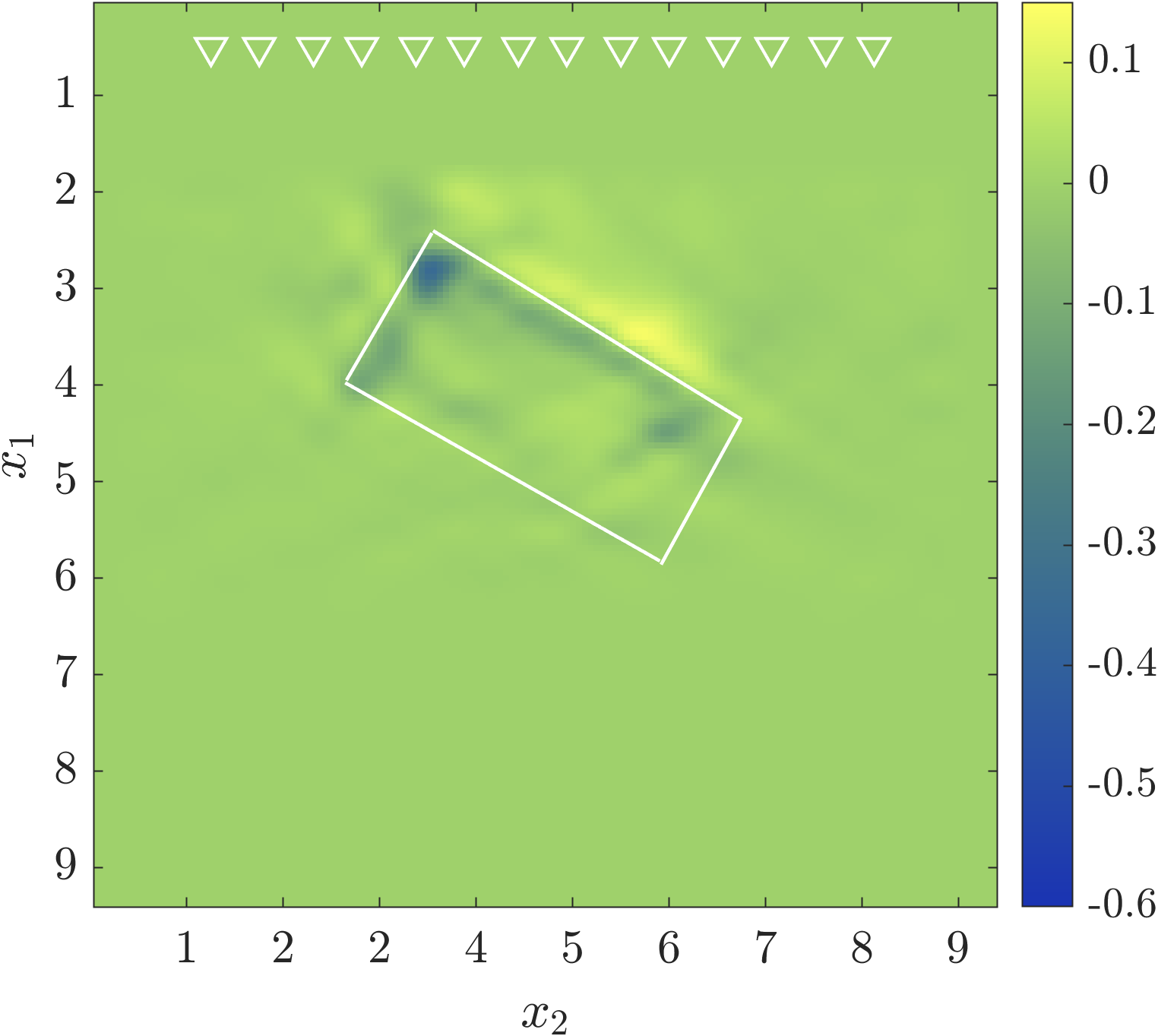}
\end{tabular}
\vspace{-0.1in}
\caption{The inversion results for the medium plotted in Fig.~\ref{fig:Comb1}.  Top row: results obtained with Algorithm \ref{alg:AlgR}. Bottom row: Results given by the FWI approach. The axes are in units of $\la_c$. The contour of the inclusion is superposed on these plots.}
\label{fig:Comb3}
\end{figure}

If we compute the imaging function~\eqref{eq:Jgpp} using the reference wave speed obtained by our inversion Algorithm~\ref{alg:AlgR} (top plots in Fig.~\ref{fig:Comb3}), we obtain the result shown in the right plot of Fig.~\ref{fig:Comb2}. This is an improvement because it identifies correctly the entire contour of the inclusion. It also complements the inversion results in Fig.~\ref{fig:Comb3}, which do 
not give a precise estimation of this contour and also display some small oscillations that are typical of the Tikhonov regularization procedure.
\section{Summary}
\label{sect:SUM}
We introduced a novel approach for inverse scattering with electromagnetic waves in lossless, anisotropic heterogeneous media. 
It  is an extension of  results obtained recently for the acoustic wave equation in lossless,  isotropic media. The approach uses a reduced order model (ROM)
which is an algebraic dynamical system that captures the propagation of the electric wave field inside the inaccessible medium, on a uniform time grid with appropriately chosen time step. The ROM is data driven, meaning that it can be computed just from measurements 
gathered by an active array of antennas that emit probing signals and measure the generated electric field. 
 We showed how to use the ROM to obtain an estimate of the electric field inside the medium, called the internal wave. This wave is consistent by construction with the data measured at the antennas. We use it to formulate a novel imaging method, designed to locate the rough part of the medium, the reflectivity, in a known smooth background. The imaging function is easy and inexpensive to compute and  it 
is superior to the traditional imaging approach because it does not display multiple scattering artifacts.  
We also introduced a novel inversion algorithm that uses the estimated internal wave to determine the matrix valued wave speed.
Both methods are assessed with numerical simulations and are shown to perform better than the traditional approaches. 

\section* {Acknowledgements}  This material is based upon research supported in part by the  AFOSR award number FA9550-22-1-0077. We thank Vladimir Druskin, Josselin Garnier, Alexander Mamonov  and Mikhail Zaslavsky for the feedback and their work on data driven ROM imaging and inversion with acoustic waves.
 \appendix 
 \section{Proof of Lemma \ref{lem.dyads}}
 \label{ap:A}
Take the time derivative in~\eqref{eq:GF3} and introduce the new $2\times 2$ field 
\begin{equation}
\tensor{\bH}(t,\bx,\by) = \bc^{-1}(\bx) \partial_t \bG(t,\bx,\by),
\end{equation}
which satisfies
\begin{align*}
-\nabla^\perp \big[ \nabla^\perp \cdot ( \bc(\bx) \underline{\underline{\bH}}(t,\bx,\by) )\big]+ \bc^{-1}(\bx) \partial_t^2 \underline{\underline{\bH}}(t,\bx,\by) &= \delta'(t) \underline{\underline{\bI}}  \delta(\bx-\by), ~ t \in \RR,~\bx \in \Omega,\\
\underline{\underline{\bH}}(t,\bx,\by)  &\equiv {\bf 0}, \quad t < 0,~~ \bx \in \Omega, \\
\bn^\perp(\bx) \cdot \underline{\underline{\bH}}(t,\bx,\by)  & =0, \quad t \in \RR, ~~ \bx \in \partial \Omega.
\end{align*}
Multiply this equation on the left by $\bc$ and recall the definition~\eqref{eq:u4} of the operator $A$. The equation
for $\tensor{\bH}$ becomes
\begin{align*}
A \underline{\underline{\bH}}(t,\bx,\by)+ \partial_t^2 \underline{\underline{\bH}}(t,\bx,\by) &= \delta'(t) \bc(\by) \delta(\bx-\by), \quad t \in \RR, ~~ \bx \in \Omega.
\end{align*}
It remains to show that
\begin{equation}
\underline{\underline{\bH}}(t,\bx,\by) = \bcG(t,\bx,\by) \bc(\by).
\label{eq:checkH}
\end{equation}

We check columnwise: For $p = 1,2$ we have 
\begin{align*}
(A + \partial_t^2)\left[ \bcG(t,\bx,\by) \bc(\by) \be_p \right]  = & \sum_{j=1}^2
c_{j,p}(\by)(A + \partial_t^2)\left[ \bcG(t,\bx,\by) \be_j \right] ,
\end{align*}
and using equation~\eqref{eq:GF1} we get
\begin{align*}
(A + \partial_t^2 )\left[ \bcG(t,\bx,\by)\bc(\by) \be_p \right] = &
\sum_{j=1}^2 c_{j,p}(\by) \delta'(t) \be_j \delta(\bx-\by) =
\delta'(t) \bc(\by) \be_p \delta(\bx-\by).
\end{align*}
This shows that $\bcG(t,\bx,\by) \bc(\by)$ solves the same equation as $\underline{\underline{\bH}}(t,\bx,\by)$. They  both vanish 
 at $t < 0$ and 
they satisfy the same homogeneous boundary condition, due to the assumption that $\bc = c_o \underline{\underline{\bI}}$ at $\partial \Omega$. Since the wave equation has a unique solution, the result~\eqref{eq:checkH} holds. $\Box$

\section{Proof of Lemma \ref{lem.2}}
 \label{ap:B}
Introduce the $2m \times 2$ auxiliary  fields
\begin{equation}
\bsig_l(\by) = \bi_l^T \bR^{-1} \bV^T(\by;\tilde \bc), \qquad l = 0, \ldots, n-1,
\label{eq:B0}
\end{equation}
and observe that they satisfy
\begin{equation*}
\sum_{l=0}^{n-1} \bi_l \bsig_l(\by) = \underbrace{\sum_{l=0}^{n-1} \bi_l \bi_l^T}_{\bI_{2nm}} \bR^{-1} \bV^T(\by;\tilde \bc)
= \bR^{-1} \bV^T(\by;\tilde \bc).
\end{equation*}
We deduce that 
\begin{align}
\sum_{l=0}^{n-1} \bu_l(\bx) \bsig_l(\by) &= \bcU(\bx) \sum_{l=0}^{n-1} \bi_l  \bsig_l(\by) 
= \bcU(\bx) \bR^{-1} \bV^T(\by;\bc)  \nonumber\\
&\stackrel{\eqref{eq:defV4}}{=} \bV(\bx) \bV^T(\by;\tilde \bc) \stackrel{\eqref{eq:deltaRangeEst}}{=} \bde^{{\rm est}}(\bx,\by;\tilde \bc). \label{eq:B1}
\end{align}
Therefore,
\begin{align}
\cos(t_j \sqrt{\cA}\, )  \bde^{{\rm est}}(\bx,\by;\tilde \bc) &\stackrel{\eqref{eq:B1}}{=} \cos(t_j \sqrt{\cA}\, ) \sum_{l=0}^{n-1} \bu_l(\bx) \bsig_l(\by) \nonumber\\
&\stackrel{\eqref{eq:Sn3}}{=} \cos(t_j \sqrt{\cA}\, ) \sum_{l=0}^{n-1} \cos(t_l \sqrt{\cA}\, ) \bu_0(\bx)
\bsig_l(\by) \nonumber \\
&= \frac{1}{2} \sum_{l=0}^{n-1}\Big[ \cos(t_{j+l}  \sqrt{\cA}\, ) + \cos(t_{|j-l|}  \sqrt{\cA}\, )\Big] \bu_0(\bx)
\bsig_l(\by) \nonumber \\
&\stackrel{\eqref{eq:Sn3}}{=} \frac{1}{2} \sum_{l=0}^{n-1} \Big[ \bu_{j+l}(\bx) + \bu_{|j-l|}(\bx) \Big] \bsig_l(\by).
\label{eq:B2}
\end{align}
With this result and recalling the expression~\eqref{eq:Dinn} of the entries of the data matrices, 
\begin{align*}
\int_{\Omega} d \bx \, \bu_0^T(\bx)  \cos(t_j \sqrt{\cA}\, )  \bde^{{\rm est}}(\bx,\by;\tilde \bc) &= \frac{1}{2} \sum_{l=0}^{n-1} \Big[ \mathbb{D}_{j+l} + \mathbb{D}_{|j-l|} \Big] \bsig_l(\by) \\
&\stackrel{\eqref{eq:calcGram}}{=} \sum_{l=0}^{n-1} \mathbb{M}_{j,l} \bsig_l(\by),
\end{align*}
which is a $2m \times 2$  field. Moreover, using the definition~\eqref{eq:B0} of $\bsig_l$ and that
\[
\sum_{l=0}^{n-1} \mathbb{M}_{j,l} \bi_l^T = \sum_{l=0}^{n-1} \underbrace{\bi_j^T \mathbb{M} \bi_l}_{\bM_{j,l}} \bi_l^T = \bi_j^T \mathbb{M},
\]
we obtain that 
\begin{align}
\int_{\Omega} d \bx \, \bu_0^T(\bx)  \cos(t_j \sqrt{\cA}\, )  \bde^{{\rm est}}(\bx,\by;\tilde \bc) 
= \bi_j^T  \mathbb{M} \bR^{-1} \bV^T(\by;\tilde \bc) \nonumber \\
\stackrel{\eqref{eq:Chol}}{=} \bi_j^T \bR^T \bV^T(\by;\tilde \bc) \stackrel{\eqref{eq:EstWave}}{=} \Big[\bu_j^{\rm est}(\by;\tilde \bc) \Big]^T. \label{eq:PFeq}
\end{align}

Now, recall definition~\eqref{eq:DM3} of $\bu_0$, which says
that its $(s,p)$ column is
\[
\bu_0^{(s,p)}(\bx) = \left|\hat f (\sqrt{\cA}\, )\right| \be_p F^{(s)}(\bx).
\]
Using this in~\eqref{eq:PFeq} and recalling that $\cA$ is symmetric and functions of $\cA$ commute, we get that the $(s,p)$ row in the left hand side of~\eqref{eq:PFeq} is
\begin{align}
\hspace{-0.1in} \int_{\Omega} d \bx \, \be_p^T F^{(s)}(\bx)  \cos(t_j \sqrt{\cA}\, ) \left|\hat f (\sqrt{\cA}\, )\right| \bde^{{\rm est}}(\bx,\by;\tilde \bc) &\stackrel{\eqref{eq:estg1}}{=} \int_{\Omega} d \bx \, \be_p^T F^{(s)}(\bx)  \tensor{\bg}^{\rm est} (t_j,\bx,\by;\tilde \bc) \nonumber \\
&~~\approx \be_p^T  \tensor{\bg}^{\rm est} (t_j,\bx_s,\by;\tilde \bc). \label{eq:B5}
\end{align}
Here the approximation is due to the assumption of point-like support of the antenna, modeled by $F^{(s)}$, 
and  also that functions in the range of $|\hat f ( \sqrt{\cA}\, )|$ are smooth.  

The result~\eqref{eq:estg} follows from~\eqref{eq:PFeq}-\eqref{eq:B5}. The proof of~\eqref{eq:trueg} is the same, once we replace $\bV(\cdot ;\tilde \bc)$ by the true $\bV$. $ ~~ \Box$

\section{Proof of Proposition \ref{prop.1}}
 \label{ap:C}
Using the jump conditions induced by the derivative of the Dirac delta in~\eqref{eq:GF1}, we can write equations~\eqref{eq:GF1}--\eqref{eq:GF2} in the equivalent form  
\begin{align*}
(A +\partial_t^2) \bcG(t,\bx,\by) = \tensor{\bf 0}, \qquad t > 0, ~~ \bx \in \Omega, \\
\bcG(0+,\bx,\by) = \underline{\underline{\bI}} \delta(\bx-\by), \quad
\partial_t \bcG(0,\bx;\by) = \tensor{\bf 0}, \qquad \bx \in \Omega.
\end{align*}
By definition~\eqref{eq:estg1}, the field $\tensor{\bg}^{\rm est}$ satisfies the same equation as $\bcG$, but has a different 
initial condition
\begin{align*}
(A +\partial_t^2) \tensor{\bg}^{\rm est}(t,\bx,\by;\tilde \bc)  = \tensor{\bf 0}, \qquad t > 0, ~~ \bx \in \Omega, \\
\tensor{\bg}^{\rm est}(0,\bx,\by;\tilde \bc) = |\hat f(\sqrt{\cA})| \bde^{\rm est}(\bx,\by;\tilde \bc), \quad
\partial_t \tensor{\bg}^{\rm est}(0,\bx,\by;\tilde \bc)  = {\bf 0}, \qquad \bx \in \Omega.
\end{align*}
Here we used that $\cA$ is the restriction of $A$ on $\cD \setminus \mbox{null}(A)$
and also that the columns of $\bde^{\rm est}$ do not have components in $\mbox{null}(A)$.  By the principle of linear superposition, 
\begin{equation}
 \tensor{\bg}^{\rm est}(t,\bx,\by;\tilde \bc) = \int_{\Omega} d \bxi \, \bcG(t,\bx,\bxi) |\hat f(\sqrt{\cA})| \bde^{\rm est}(\bxi,\by;\tilde \bc).
 \label{eq:superp}
 \end{equation}

Similar to what we have done in equation~\eqref{eq:u5}, we decompose $\bcG$ as follows
\begin{align}
\bcG(t,\bx;\by) =  \bbG(t,\bx;\by) +\bGa(t,\bx;\by). 
\label{eq:GFDec}
\end{align}
The first term solves the same equation as $\bcG$, but its initial state is the approximation of $\tensor{\bI} \delta(\bx-\by)$ 
in the range of $A$ and therefore $\cA$,
\begin{equation}
\bde^\cA(\bx,\by) = \sum_{j=1}^\infty \bphi_j(\bx) \bphi_j^T(\by).
\label{eq:approxDe}
\end{equation}
It can be written in compact form as
\begin{equation}
\bbG(t,\bx;\by) = 1_{[0,\infty)}(t) \cos (t \sqrt{\cA} \, ) \bde^\cA(\bx,\by),
\label{eq:defGa2}
\end{equation}
and has a series expansion in the eigenfunctions of $\cA$, which form an orthonormal  basis of $\mbox{range}(\cA)$.
The second term in~\eqref{eq:GFDec} has columns that lie in the null space of $A$. It solves the equation
\begin{equation}
\partial_t^2  \bGa(t,\bx,\by) = \delta'(t) \left[ \underline{\underline{\bI}} \delta(\bx-\by) - \bde^\cA(\bx,\by) \right],
\label{eq:deftGa}
\end{equation} 
and it is given by
\begin{equation}
\bGa(t,\bx;\by) = 1_{[0,\infty)}(t)  \left[ \underline{\underline{\bI}} \delta_{\by}(\bx) - \bde(\bx;\by) \right].
\label{eq:deftildeG}
\end{equation}

Note  that since $\mbox{null}(A) = [\mbox{range}(\cA)]^\perp$,  $\bGa$ has no contribution  in~\eqref{eq:superp}  i.e., 
\begin{equation}
 \tensor{\bg}^{\rm est}(t,\bx,\by;\tilde \bc) = \int_{\Omega} d \bxi \, \bbG(t,\bx,\bxi) |\hat f(\sqrt{\cA})| \bde^{\rm est}(\bxi,\by;\tilde \bc).
 \label{eq:superp1}
 \end{equation}
Using the expression~\eqref{eq:defGa2} of $\bbG$ in equation~\eqref{eq:superp1} and recalling the spectral decomposition of $\cA$, we get 
\begin{align}
\tensor{\bg}^{\rm est}(t,\bx,\by;\tilde \bc)&= \int_{\Omega} d \by \, \cos ( t \sqrt{\cA} \, ) \bde^\cA(\bx,\bxi)|\hat f(\sqrt{\cA})| \bde^{\rm est}(\bxi,\by;\tilde \bc)   \nonumber \\
&=  \sum_{q=1}^\infty \cos (t \sqrt{\theta_q} \, ) \bphi_q(\bx) \int_{\Omega} d \bxi \, \bphi_q^T(\bxi)|\hat f(\sqrt{\cA})| \bde^{\rm est}(\bxi,\by;\tilde \bc)  \nonumber \\
&= \sum_{q=1}^\infty \cos (t \sqrt{\theta_q} \, ) |\hat f( \sqrt{\theta_q} \, )| \bphi_q(\bx) \int_{\Omega} d \bxi \, \bphi_q^T(\bxi) \bde^{\rm est}(\bxi,\by;\tilde \bc).\label{eq:P3.3}
\end{align}

Next, we use the expression~\eqref{eq:checkf} of the signal $\check f$ to get the convolution. First, we deduce that  
\begin{align*}
\int_{-\infty}^\infty dt \, \check f(t) \cos(\om t) &\stackrel{\eqref{eq:checkf}}{=}
\int_{-\infty}^\infty \frac{d \om'}{2\pi} |\hat f(\om')| \int_{-\infty}^\infty dt \, \cos(\om't) \cos(\om t)   \\
&= \int_{-\infty}^\infty \frac{d \om'}{2} |\hat f(\om')| \left[ \delta(\om-\om') + \delta(\om+\om') \right] \\
&= \frac{1}{2} \left[ |\hat f(\om)| + |\hat f(-\om)| \right]= |\hat f(\om)|.
\end{align*}
Therefore, we can write
\begin{align*}
&|\hat f( \sqrt{\theta_q} \, )|\cos (t \sqrt{\theta_q} \, )  = \int_{-\infty}^\infty d t' \, \check f(t') \cos (t' \sqrt{\theta_q} \, ) 
\cos (t \sqrt{\theta_q} \, ) \\
&\quad = \frac{1}{2}\int_{-\infty}^\infty d t' \, \check f(t') \cos [(t-t') \sqrt{\theta_q} \, ] +
 \frac{1}{2}\int_{-\infty}^\infty d t' \, \check f(t') \cos [(t+t') \sqrt{\theta_q} \, ].
 \end{align*}
Changing $t' \mapsto -t'$ in the last term and using that $\check f$ is even, we get
\begin{align}
|\hat f( \sqrt{\theta_q} \, )|\cos (t \sqrt{\theta_q} \, )|  &= \int_{-\infty}^\infty d t' \, \check f(t') \cos [(t-t') \sqrt{\theta_q} \, ] = \check f(t) \star_t  \cos \Big(t \sqrt{\theta_q} \, \Big) . \label{eq:C9}
\end{align}

Substituting~\eqref{eq:C9} in~\eqref{eq:P3.3} evaluated at $\bx = \bx_s$ and switching the convolution with the series (justified via the dominated convergence theorem), we get
\begin{align}
\tensor{\bg}^{\rm est}(t,\bx_s,\by;\tilde \bc) &= \check f(t) \star_t\sum_{q=1}^\infty \cos (t \sqrt{\theta_q} \, ) \bphi_q(\bx_s) \int_{\Omega} d \bxi \, \bphi_q^T(\bxi)\bde^{\rm est}(\bxi,\by;\tilde \bc) \nonumber \\
&= \check f(t) \star_t \int_{\Omega} d \bxi \, \left[\cos ( t \sqrt{\cA} \, ) \bde^\cA(\bx_s,\bxi) \right]\bde^{\rm est}(\bxi,\by;\tilde \bc) \nonumber \\
&\stackrel{\eqref{eq:defGa2}}{=} \check f(t) \star_t \int_{\Omega} d \bxi \, \bbG(t,\bx_s;\bxi) \bde^{\rm est}(\bxi,\by;\tilde \bc)  \nonumber \\
&= \check f(t) \star_t \int_{\Omega} d \bxi \, \bcG(t,\bx_s;\bxi)\bde^{\rm est}(\bxi,\by;\tilde \bc) .
\label{eq:P3.4}
\end{align}
Here the second equality is by the definition of $ \cos ( t \sqrt{\cA} \, )$ and equation~\eqref{eq:approxDe}  and the last equality follows from the decomposition~\eqref{eq:GFDec}. The third equality requires more explanation. Indeed, when taking the convolution, the cosine will be evaluated at negative times if $t = O(t_F)$, so why can we use the causal Green's function?  This is where we use the assumption that  $\bde^{\rm est}(\bxi,\by;\tilde \bc)$ peaks at
points $\bxi$ in the vicinity of $\by$ and that $\bx_s$ is far from $\by$. By causality, we have
$\tensor{\bg}^{\rm est}(t,\bx_s,\by;\tilde \bc) \approx 0$ for $t  = O(t_F)$, so we are not interested in such early times.

To complete the proof, we use Lemma \ref{lem.dyads}. We get
\begin{align*}
\tensor{\bg}^{\rm est}(t,\bx_s,\by;\tilde \bc) &= \check f(t) \star_t \int_{\Omega} d \bxi \, \bc^{-1}(\bx_s) \partial_t \bG(t,\bx_s,\bxi)
\bc^{-1}(\bxi) \bde^{\rm est}(\bxi,\by;\tilde \bc) .
\end{align*}
Switching the time derivative  in the convolution to  $\check f$ and recalling that
$\bc(\bx_s) = c_o \underline{\underline{\bI}}$, we get the result. $~ \Box$

\section{Numerical simulations}
\label{ap:num} 
Here we describe briefly the setup of the numerical simulations and the computational cost.
\subsection{Setup}
To compute the data matrices~\eqref{eq:newD}, we solve the wave equation~\eqref{eq:Sn1}. We use a finite difference time domain numerical method~\cite{FDTD}, with discretization of the operator $A$ on a Lebedev 
grid with step size $\ell$. This consists of two Yee grids, shifted by $\ell/2$, on which we discretize the first and second component of the electric field, respectively~\cite{Lebedev64,Lebedev64II,Yee66,Nauta2013}. The time derivative is discretized with a second order, centered 
difference scheme, on a uniform time grid with small step $\Delta t$ satisfying the CFL condition. The time step $\tau$ used in the computation of the ROM is an integer multiple of $\Delta t$. 

The initial condition~\eqref{eq:DM3} is computed using the spectral decomposition of the discretization of the operator $A$. Equations ~\eqref{eq:defu},~\eqref{eq:u5},~\eqref{eq:u8} and~\eqref{eq:Sn0} and the finite wave speed imply that the initial condition depends 
only on the medium near the array. Thus, it can be computed using the discretization of $A$ in a smaller domain around the array.

The probing signal has the Fourier transform
\begin{equation}
\hat f(\om) = \frac{\om^2}{2} \Big[ e^{-\frac{(\om-\om_o)^2}{2 \om_b^2}} + e^{-\frac{(\om+\om_o)^2}{2 \om_b^2}} \Big],
\label{eq:Num1}
\end{equation}
with standard deviation $\om_b$ chosen so that the highest frequency at $-25$dB  cut-off is $\om_{c} = 5/3\om_o$.
The central wavelength is $\la_o = 26.7 \ell$ and the wavelength at the cut-off frequency is $\la_c = 16 \ell$.

The array of antennas is located at a distance $8 \ell$ from the top boundary.  Unless stated otherwise, for imaging, the antennas are separated by the distance $\la_c/4$ and the data are sampled at time step $\tau = 0.3 \pi/\om_c$. For the inversion, we used the antenna separation $\la_c/2$ and the time step $\tau = 0.45 \pi/\om_c$.

The parametrization~\eqref{eq:Inv4}-\eqref{eq:Inv5} of the search dielectric permittivity tensor is done using 
the Gaussian basis functions
\begin{equation}
\phi_j(\bx) = \exp \left[ - \frac{(x_1-X_{1,j})^2}{2 \sigma_1^2} - \frac{(x_2-X_{2,j})^2}{2 \sigma_2^2}\right],
\end{equation}
where $\{(X_{1,j},X_{2,j})\}_{j=1}^N$ is a rectangular lattice in the inversion domain, with spacing $\la_c/4$ along the $x_1$ axis and 
$5 \la_c/16$ along the $x_2$ axis. The standard deviations are chosen as $\sigma_1=2.3 \ell$ and $\sigma_2=2.9 \ell$ to be able to represent a wide range of media and partition the search domain well.

The Tikhonov regularization parameter $\nu$ is chosen adaptively, based on the eigenvalues of the ``Hessian" in the  
Gauss-Newton method. For $N$ basis functions as in equation~\eqref{eq:Inv5}, we see from ~\eqref{eq:Inv4} that we have $3N$ unknowns. If $\{\la_j\}_{j=1}^{3N}$ are the eigenvalues of the Hessian, sorted in descending order, then in all our simulations 
we set $
\nu = \la_{\mbox{round}(0.9 N)}.
$

\subsection{Computational cost}
We begin with the computational cost of solving the forward problem, and thus computing the data matrices $\bbD(t)$. 
If we have a mesh size with $N_1$ points in the $x_1$ direction and $N_2$ points in the $x_2$ direction, and 
 we use $2m$ excitations, then at each time step we multiply a sparse matrix of size $4(N_1N_2) \times 4(N_1 N_2)$ with a vector of size $4N_1 N_2 \times 2 m$ at a cost of $O(mN_1N_2)$. The cost of computing $\bbD(t)$, at  $t \in (0,T)$, with $T = N_t \Delta t$, is therefore, $O(mN_1N_2 N_t)$. 
 
The cost of each Gauss-Newton iteration for optimizing our objective function or the FWI objective function, is 
 dominated by the computation of the Jacobian. There is an efficient way of computing the Jacobian for FWI, using the 
 adjoint formula~\cite{virieux2009overview}. The extension of this formula to our objective function can be made, but we have not done so. With such a formula, the Jacobian could be computed very efficiently. At the moment, the Jacobian is computed using finite differences and it is the bottleneck of our computations. 
 
 The extra cost of computing our objective function vs. FWI is due to the block Cholesky algorithm, 
 which takes $O(m^3 n^3)$ operations. 
 
 If $mn \gg 1$, which is likely the case in three dimensions, the dominant computational cost should be in solving the normal equations for the Gauss-Newton updates. This can be handled using iterative methods. Moreover, one can reduce the computational 
 cost by working with a collection of sub-arrays. This idea has been used for a different problem in~\cite{borcea2014model}.

\bibliographystyle{spmpsci}      
\bibliography{biblio.bib}

\end{document}